\documentclass[11pt,a4paper]{article}

\usepackage{amsmath}
\usepackage{amsfonts}
\usepackage{amsthm}
\usepackage{amssymb}
\usepackage{mhequ}
\usepackage{mhsymb}
\usepackage{enumitem}
\usepackage{graphicx}
\usepackage{mathrsfs}
\usepackage{xcolor}
\usepackage{hyperref}
\usepackage{fullpage}
\usepackage{esint}

\usepackage{pgfplots}
\pgfplotsset{compat=1.8}

\usepackage{tikz}
\usetikzlibrary{math,calc,trees,positioning,patterns,arrows.meta}
\usepgfplotslibrary{groupplots}

\usepackage{xpatch}
\tracingpatches
\xpatchcmd{\proof}{\itshape}{\bf}{}{}
\xpatchcmd{\example}{\itshape}{\bf}{}{}

\usepackage{todonotes}

\numberwithin{equation}{section}

\newtheorem{theorem}{Theorem}[section]
\newtheorem{lemma}[theorem]{Lemma}
\newtheorem{proposition}[theorem]{Proposition}
\newtheorem{corollary}[theorem]{Corollary}

\theoremstyle{definition}
\newtheorem{remark}[theorem]{Remark}
\newtheorem{definition}[theorem]{Definition}
\newtheorem{notation}[theorem]{Notation}
\newtheorem{assumption}[theorem]{Assumption}

\theoremstyle{remark}
\newtheorem{example}[theorem]{Example}

\def\R{{\mathbb R}}

\def\Z{{\mathbb Z}}
\def\N{{\mathbb N}}
\newcommand{\bbS}{{\mathbb S}}

\newcommand{\bscal}[1]{\Big\langle #1 \Big\rangle}

\newcommand{\cD}{{\mathscr{D}}}

\newcommand{\id}{\mathrm{Id}}

\newcommand{\T}{\mathrm{T}}

\newcommand{\Span}{\operatorname{span}}

\newcommand{\cl}{\operatorname{cl}}
\newcommand{\bd}{\operatorname{bd}}

\newcommand{\roof}[1]{\lceil #1 \rceil}
\newcommand{\floor}[1]{\lfloor #1 \rfloor}

\newcommand{\bone}{\mathbf{1}}

\newcommand{\mbX}{\mathbf{X}}

 \usepackage{stmaryrd}
 \newcommand{\disc}[1]{{\talloblong #1 \talloblong}}

\usepackage{mathrsfs}

\newcommand{\bphi}{\boldsymbol{\phi}}

\newcommand{\bpsi}{\boldsymbol{\psi}}
\newcommand{\bk}{\mathbf{k}}

\newcommand{\bt}{\mathbf{t}}

\newcommand{\bm}{\mathbf{m}}

\newcommand{\mfo}{\mathfrak{o}}
\newcommand{\mfr}{\mathfrak{r}}

\newcommand{\mfK}{\mathfrak{K}}

\newcommand{\mfR}{\mathfrak{R}}
\newcommand{\mfB}{\mathfrak{B}}

\newcommand{\s}{\mathfrak{s}}
\newcommand{\mfc}{\mathfrak{c}}
\newcommand{\mfn}{\mathfrak{n}}
\newcommand{\mfe}{\mathfrak{e}}
\newcommand{\mfC}{\mathfrak{C}}

\newcommand{\mfT}{\mathfrak{T}}
\newcommand{\mfW}{\mathfrak{W}}
\newcommand{\mfV}{\mathfrak{V}}
\newcommand{\mfU}{\mathfrak{U}}

\newcommand{\mrd}{\mathop{}\!\mathrm{d}}

\newcommand{\dist}[1]{{|\!|\!| #1 |\!|\!|}}

\newcommand{\drivers}{\mathbb{M}}

\renewcommand{\d}{\partial}

\newcommand{\sC}{\mathscr{C}}

\newcommand{\trinorm}[1]{{|\!|\!| #1 |\!|\!|}}

\title{Large field problem in coercive singular PDEs}

\author{Ilya Chevyrev\thanks{SISSA, Trieste, 34136, Italy.
\href{mailto:ichevyrev@gmail.com}{\tt ichevyrev@gmail.com}}
\and Massimiliano Gubinelli\thanks{Mathematical Institute, University of Oxford, Oxford, OX2 6GG, UK. \href{mailto:gubinelli@maths.ox.ac.uk }{\tt gubinelli@maths.ox.ac.uk}}}

\date{\today}

\begin{document}

\maketitle

\begin{abstract}
We derive a priori estimates for singular differential equations of the form
\[
\mathcal{L} \phi = P(\phi,\nabla\phi) + f(\phi,\nabla\phi)\xi
\]
where $P$ is a polynomial, $f$ is a sufficiently well-behaved function, and $\xi$ is an irregular distribution such that the equation is subcritical.
The differential operator $\mathcal L$ is either a derivative in time, in which case we interpret the equation using rough path theory,
or a heat operator, in which case we interpret the equation using regularity structures.
Our only assumption on $P$ is that solutions with $\xi=0$ exhibit coercivity.
Our estimates are local in space and time, and independent of boundary conditions.

One of our main results is an abstract estimate that allows one to pass from a local coercivity property to a global one using scaling, for a large class of equations.
This allows us to reduce the problem of deriving a priori estimates to the case when $\xi$ is small.
\end{abstract}

\tableofcontents

\section{Introduction}

Global aspects of solution theory to stochastic partial differential equations (SPDEs) remain an open problem. In particular, the large field regime---where solutions may become arbitrarily large---poses significant challenges for a priori estimates and existence. 

We develop a strategy for a priori estimates based on the fact that the obstruction to such estimates is given by the presence of solutions which blow up at specific space-time points. In the neighbourhood of such points, the solution is large and the equation can be approximated by its scale-invariant version. In the subcritical regime, this corresponds to a semi-classical situation where the noise become small and the equation is dominated by its deterministic part.
If the deterministic equation has a coercive non-linearity then blow-up cannot occur. This intuition can be made rigorous and we show that it leads to a priori estimates for a large class of subcritical singular SPDEs.

We note an intriguing convergence of ideas, where one begins with an example that violates a desired estimate (or property) and, through a scaling argument, arrives at a contradiction by reducing the problem to a much simpler case—typically in a ‘semi-classical’ regime where the noise is negligible.

In particular, we are aware of two other situations where this happens:
\begin{itemize}
  \item \textbf{Symmetry:} In \cite{CCHS_2D,CCHS_3D,BGHZ22} the authors use this line of reasoning to prove structural properties of renormalisation, in particular the equivariance under symmetries of the renormalized equation. It were these works that inspired us to tackle the large field problem via a similar strategy.

  \item \textbf{Regularity:} This line of reasoning is used to prove Schauder estimates in a kernel-free way, see e.g. \cite{Simon_97,Sauer_Smith_25_Schauder,Esquivel_Weber_24_fractional}.
  The idea is to zoom into a space-time point where the desired estimate fails and then use scaling to reduce the problem to a situation where the contributions by the noise are small. In this regime, the desired estimate can be proved by a Liouville property of the linear part of the equation.
\end{itemize}

In all these cases, we see that a particular property of the SPDE is related to some limit where the noise disappears, either in small scales or for large values of the field, i.e. near an eventual blow-up, or in a regime where a symmetry is broken at the semi-classical level. 

It would be interesting to explore this heuristic further especially in connection with geometric aspects of singular SPDEs.

\textbf{Setting and main results.}
We consider a partial differential equation (PDE) of the form
\begin{equ}\label{eq:equation}
\mathcal{L} \phi = P(\phi,\nabla\phi) + f(\phi,\nabla\phi)\xi\;.
\end{equ}
Here, $\CL$ is the heat operator $\d_t - \Delta = \d_{1} - \sum_{i=2}^d \d_i^2$ on $\R^d$,
$P$ is a polynomial of $\phi$ and its spatial derivative $\nabla \phi = (\d_i \phi)_{i=2}^d$,
and $f$ is a function of $(\phi,\nabla\phi)$.
We allow the case $d=1$, in which case $\CL=\d_t$, so the equation is an ordinary differential equation (ODE) and $P$ and $f$ depend only on $\phi$.
The equation is posed for a (vector-valued) function
$
\phi\colon \Omega\to \R^m
$,
where $m\geq 1$ and $\Omega = (-1,0)\times(-1,1)^{d-1}$ is the backwards parabolic ball.

Here $\xi$ is a distribution in a H\"older--Besov space $\CC^{\beta}(\Omega,\R^n)$ with $\beta\leq 0$
(like $\phi$, we allow $\xi$ to be vector-valued).
We are especially interested in the case where \eqref{eq:equation} is \emph{singular}, meaning that $\beta$ is sufficiently negative to render the equation classically ill-posed.
In this case, we interpret \eqref{eq:equation} either using rough paths \cite{Lyons98} (if $d=1$) or using regularity structures \cite{Hairer14} (if $d\geq 2$).
In both cases, to give meaning to the solution, we require a rough path\,/\,model $Z$ above $\xi$.

Our assumptions regarding the equation can be stated roughly as follows:
\begin{itemize}
  \item \eqref{eq:equation} is scaling subcritical in the sense of \cite{Hairer14}.
  \item The equation $\CL \phi = P(\phi,\nabla\phi)$ is scale-invariant with exponent $\alpha>0$. Moreover, there exists $\delta>0$ such that, if $\CL \phi = P(\phi,\nabla\phi)$ on $\Omega$ with $\|\phi\|_\infty \leq 1$, then $|\phi(0)| < 1-\delta$.
  \item $f\colon \R^m \to L(\R^n,\R^m)$ has sufficient regularity and growth at infinity.
\end{itemize}
The first condition is a bound on the regularity of $\xi$ and is necessary to apply the theories of rough paths or regularity structures.
For example, for $d\geq 2$, if $P = -|\phi|^2\phi$ and $f$ is constant, then we require $\beta>-3$, and if $f$ is not constant but does not depend on $\nabla\phi$, then we require $\beta>-2$.
The condition roughly states that the solution $\phi$ is, on small scales, a perturbation of the solution to the linear equation $\CL \psi = \xi$
(in the language of QFT, it corresponds to super-renormalisability).

In the second condition, `scale-invariance' means that, if $\phi$ solves $\CL \phi = P(\phi,\nabla\phi)$, then 
\begin{equ}[eq:rescale]
\psi(t,x) = \lambda^{\alpha}\phi(z+(\lambda^2 t, \lambda x))
\end{equ}
solves the same equation for all $\lambda>0$ and $z\in\R^d$.
For example, we can take $P$ of the form $P(\phi,\nabla\phi) = \phi^p + \phi^{\frac{p-1}{2}} \nabla \phi$
and the corresponding exponent is $\alpha= \frac{2}{p-1}$.
(The condition implies that $P$ depends on $\nabla\phi$ in an affine way.)
By a scaling argument, the second part of the condition is actually equivalent to the much stronger `coming down from infinity' property which states that $|\phi(0)|\lesssim 1$ whenever $\phi$ solves $\CL \phi = P(\phi,\nabla\phi)$ on $\Omega$ (uniformly in the boundary conditions).
The fact that $P$ is a polynomial is not crucial for our argument
and we make this assumption only to simplify the proof.

We leave the conditions on $f$ intentionally vague for now, as their formulation is lengthy and depends on whether we are in the ODE case ($d=1$) or the PDE case ($d\geq2$).
We highlight, however, that our regularity assumptions do not impose uniqueness of solutions to \eqref{eq:equation}.

Under these conditions, we derive a priori estimates for any solution $\phi$ of the remainder equation of \eqref{eq:equation}.
By `solution of the remainder equation' we mean that $\phi$ is the function-like remainder of the `full' solution.
In the case of rough paths or when $f$ is not constant, $\phi$ is already the full solution.
But if $f$ is constant and $d\geq 2$, then, like in e.g. \cite{BCCH21}, we write the `full' solution as $\bar \phi = X + \phi$, where $X$ is an explicit polynomial function of $\xi$ that, in general, is a distribution not in $L^\infty$.
Since we derive bounds in $L^\infty$, we focus henceforth on $\phi$.

One heuristic formulation of our main result is

\begin{theorem}\label{thm:heuristic}
Suppose $\phi\colon \Omega \to \R^m$ solves the remainder equation of \eqref{eq:equation}
with input rough path\,/\,model $Z$ above $\xi$.
Then
\begin{equ}[eq:heuristic_bound]
|\phi(z)| \lesssim \max\{\disc{f,Z}_{z,\lambda}, \dist{z}^{-\alpha}\}
\end{equ}
where $\dist{z}$ is the distance of $z$ to the parabolic boundary of $\Omega$ and $\disc{f,Z}_{z,\lambda}$ is a `norm' of the pair $(f,Z)$ that depends on the growth of the derivatives of $f$ at infinity and on the size of $Z$ in a ball centred at $z$ of radius $\lambda \leq \frac12\dist{z}$.
\end{theorem}

We give a precise and sharper formulation for the rough path case in Theorem \ref{thm:RP_bound} and for the regularity structure case in Theorem \ref{thm:SPDE}.
The conditions on $f$ and $P$ are made precise in the respective sections.

\textbf{Method of proof.}
Our proof of Theorem \ref{thm:heuristic} is based on a scaling argument and proceeds in two steps.

First, we suppose the pair $(f,Z)$ is suitably `small' and that $\|\phi\|_{\infty;\Omega}\leq 1$.
Recalling $\delta>0$ from our assumption on $P$, we prove in this case that $|\phi(0)|<1-\delta/2$.
This step is significantly easier than handling \eqref{eq:equation} right away because smallness of $\|\phi\|_\infty$ and $(f,Z)$ can be used to show $\|\phi-\mathring\phi\|_\infty<\delta/2$, where $\mathring\phi$ solves $\CL\mathring\phi = P(\mathring\phi,\nabla\mathring\phi)$ with the same boundary conditions as $\phi$.
Then $|\mathring\phi| < 1 - \delta$ by assumption, which leads to the claimed bound $|\phi(0)|<1-\delta/2$.
At this point, we have not used scale-invariance of the equation.

Second, no longer assuming any smallness, we rescale $\phi\mapsto \psi$ according to the transformation \eqref{eq:rescale} around a point $z$ that maximises the ratio of the left- to right-hand side of \eqref{eq:heuristic_bound}.
We choose $\lambda < 1$ in \eqref{eq:rescale} such that $|\psi(0)|_{\infty;\Omega}= 1-\delta/2$ and remark that $\psi$ solves the same equation except $(f,Z)$ becomes rescaled.
Due to subcriticality and our growth assumptions on $f$,
this rescaled version of $(f,Z)$ becomes `small'.
This implies $\|\psi\|_{\infty;\Omega} > 1$ by the first step,
which readily leads to a contradiction on the maximality of $z$.

To illustrate the idea in a simple setting, we implement this strategy in Section \ref{sec:classical} for the PDE $\CL \phi =  - \phi^p + \xi$ in the classically well-posed regime $\beta>-2$.

The argument in the second step is rather general and does not rely on the fact that we work with differential equations.
We therefore find it useful to state it as an abstract local-to-global coercivity bound, a simplified and heuristic version of which is as follows.

Suppose we are given
\begin{itemize}
  \item a metric space $(\Sigma,\rho)$ with a `boundary' $\d\Sigma\subset\Sigma$ and a distinguished point $o\in\Sigma$,
  \item an embedding $T_{z,\lambda}\colon\Sigma\to\Sigma$ for every $z\in\Sigma$ and $0<\lambda\leq \dist{z}$, where $\dist{z} = \inf_{y\in\d\Sigma}\rho(y,z)$,
  such that $T_{z,\lambda}(o)=z$,
  \item a set of `drivers' $\drivers$ with a `norm' $\disc{\cdot}\colon\drivers \to [0,\infty)$
  and `scaling map' $R_{z,\lambda} \colon \drivers\to \drivers$
  such that
$\disc{R_{z,\lambda} M} \leq \lambda^\alpha \disc{M}$ for some $\alpha>0$, and
  \item a set of `solutions' $\bbS_M$ for every $M\in \drivers$, elements of which are functions $\phi\colon \Sigma\to \R^m$,
  such that $\lambda^\alpha \phi \circ T_{z,\lambda}\in \bbS_{R_{z,\lambda} M}$ for all $\phi \in \bbS_M$.
\end{itemize}

Finally, suppose there exists $\delta > 0$ such that, if $\disc{M} \leq
1$ and $\| \phi \|_{\infty} \leq 1 + \delta$, then
\begin{equ}
  | \phi (o) | < 1\;.
\end{equ}
We interpret this final assumption as a \emph{local coercivity} condition.

\begin{theorem}\label{thm:abstract_heuristic}
In the above setting, every $\phi\in\bbS_M$ satisfies
\begin{equ}
| \phi(z) | \lesssim  \max\{ \disc{M} ,
    \dist{z}^{-\alpha} \}\;.
\end{equ}
\end{theorem}

We refer to Section \ref{sec:abstract}, especially Theorem \ref{thm:a_priori_abstract}, for a precise formulation,
and to Section \ref{sec:first_applications} for simple applications of this result to Young ODEs.

We believe this scaling argument is quite general and can be adapted to other settings, such as to discrete systems.
One can also adapt it to incorporate boundary data in the estimates, but we refrain from doing so in this work as it complicates the abstract setup.

Several further remarks are in order.

\begin{itemize}
\item \textbf{Relation to SPDEs.} Our results are inspired by recent developments in \emph{stochastic} PDEs.
Indeed, many interesting examples of \eqref{eq:equation} arise from taking $\xi$ as a \emph{random} distribution (e.g. white noise).
In applications of rough paths\,/\,regularity structures to such equations, one would construct a random rough path\,/\,model above $\xi$.
See \cite{CH16,LOTT24,HS23_BPHZ} for systematic ways to build renormalised random models above a class of distributions via smooth approximations, and \cite{FV10,FrizHairer20} for the case of rough paths.
In the case of regularity structures, for these smooth, renormalised models, the corresponding solution to \eqref{eq:equation} solves a \emph{renormalised} PDE,
see \cite{BCCH21}.
The results of the present article are entirely deterministic and we only require a rough path\,/\,model as input.

  \item \textbf{Notion of solution.}
  In the case of regularity structures, our notion of solution to \eqref{eq:equation} is somewhat new (see Definition \ref{def:solution_SPDE})
  and is inspired by Davie's notion of solution to rough differential equations \cite{Davie_08}, which we recall in Section \ref{sec:RPs}.
  Our notion of solution is consistent with that of \cite{Hairer14}, but it avoids making sense of calculus operations on the space of modelled distributions.
  We also do not require $\phi$ to solve a (possibly renormalised) PDE.
  (We note that a similar notion was independently proposed by \cite{BOS_25_Phi4} for the $\Phi^4$ equation in the multi-index approach to regularity structures.)

  \item \textbf{Local theory.}
  Our abstract formulation of Theorem \ref{thm:abstract_heuristic} is indifferent to the underlying solution theory of the PDE.
  The only part where we use regularity structures is in the verification of the local coercivity condition,
  and we believe that, for certain equations,
  one could verify the condition with other methods, such as paracontrolled analysis \cite{GIP15}, renormalisation group and flow approach \cite{Kupiainen2016,Duch25_flow},
  or the multi-index approach to regularity structures \cite{OW19,LOT23,OSSW25_a_priori}.
  We believe it would be of interest if our scaling argument can be combined with other local solution theories to broaden the scope of equations covered (e.g. to quasi-linear equations).

  \item \textbf{Refinements of regularity structures.}
  As part of our argument, we make small refinements of standard estimates in regularity structures.
  This includes the reconstruction theorem (Lemma \ref{lem:loc_reconst}) and multi-scale Schauder estimates (Lemma \ref{lem:integration}),
  for which we prove bounds that are uniform in the size of the truncation of the integration kernel.
  We also point out that, to bypass issues related to non-uniqueness of reconstructions of singular modelled distributions,
  we show that a distribution $\xi\in\CD'(B)$ which is in $\CC^\eta$ in an open ball $B\subset \R^d$
  admits an extension to $\CC^\eta(\R^d)$ with support contained in the closure of $B$, see Appendix \ref{sec:extension} --
  this result is similar to several other extension theorems in the literature, but our proof appears different and rather elementary.
\end{itemize}

\textbf{Comparison to other results.}
Space-time local bounds of the form \eqref{eq:heuristic_bound} have previously been considered for special cases of \eqref{eq:equation} in \cite{MW20_reaction,MoinatWeber20,CMW23,BCMW22,Jin_Perkowski_25},
and our formulation of Theorem \ref{thm:heuristic} is very much inspired by these works.

Our estimates are more general in the sense that they contain those in the aforementioned works as special cases.
Specifically, \cite{MW20_reaction} considers the PDE $\CL\phi = -\phi^p + g(\phi) + \xi$, where $g$ is bounded and $\xi\in\CC^\beta$ with $\beta>-2$, i.e. the classically well-posed regime,
and we essentially recover this case in Section \ref{sec:classical}.
In \cite{MoinatWeber20,CMW23}, the authors study the PDE $\CL\phi = -\phi^p +\xi$ in the full subcritical regime $\beta>-3$ and \cite{Jin_Perkowski_25} analyses $\CL\phi = -\phi^2 + \phi\xi$ in the first singular regime $\xi \in \CC^{-1-\kappa}$ for $\kappa>0$ small.
Section \ref{sec:SPDE} generalises these results.
Finally, \cite{BCMW22} considers the rough path case ($d=1$) where $f$ either behaves like a polynomial or has bounded derivatives of all orders. Our estimates in Section \ref{sec:RPs} agree in these two cases but we treat more general forms of $f$ (e.g. we prove stronger bounds for $f$ of the form $f(\phi)=\phi^q\sin(\phi^p)$).

We believe that the major advantage of our approach is its simplicity and its applicability.
For example, we avoid delicate applications of the maximum principle, so are able to treat vector-valued and scalar-valued equations on an equal footing (\cite{MW20_reaction,MoinatWeber20,CMW23} work only in the scalar-valued case, though one does strongly expect that their method can be extended to the vector-valued case with extra efforts).
Additionally, we are able to work with the original formulation of models and modelled distributions due to Hairer \cite{Hairer14}, and only require minor refinements of the reconstruction and integration bounds.
Relatedly, we require no qualitative smoothness assumption on $\xi$ (or its lift to a model)
and do not need to consider renormalised PDEs.
Furthermore, our method applies to stochastic quantisation equations of tilted measures as considered in \cite{Hairer_Steele_22} and we believe can be used to show optimal integrability of the invariant measure of $\CL \phi = -\phi^p +\xi$ in the full subcritical regime.

\bigskip

\textbf{Open problems.}
\begin{itemize}
  \item \textbf{Optimality.} We suspect that for singular ODEs (rough paths), our results are sharp in terms of the scaling exponents of $f$ and $\xi$.
  Furthermore, for singular PDEs (regularity structures), we suspect our results are also sharp for \emph{polynomial} $f$.
  However, there is a gap between our assumptions for ODEs and PDEs in the case of \emph{non-polynomial} $f$. E.g. we can leverage decay of $f$ at infinity to get improved bounds for ODEs, but do not know if this is possible for PDEs.
  Additionally, if $f$ and its derivatives are bounded, then, in the case of PDEs, we require $P$ to have sufficiently large degree depending on the regularity of $\xi$, while there is no restriction on the degree of $P$ in the case of ODEs;
  compare Assumption \ref{as:f_poly} with Assumptions \ref{as:eta}-\ref{as:Delta} and Section \ref{sec:linear_condition_RP},
  see also Remark \ref{rem:RP_assump}.
  We know of cases where the exponents in our bounds for PDEs
  are not sharp but did not find a satisfactory way to improve the bounds in general (see Remark \ref{rem:optimality}).
  We find it an interesting problem to determine if the gaps between the two cases can be closed.
  
  \item \textbf{Role of locality.}
  In the way we apply our general scaling principle, we rely on the locality of the heat operator $\CL$.
  It would be interesting to find extensions of the method that are able to handle non-local equations, such as the stochastic quantisation equation of the fractional $\Phi^4$ model analysed recently in \cite{Esquivel_Weber_24_fractional,DGR_23_fractional}.
  
  \item \textbf{Role of coercivity.}
  Our argument works for equations with super-linear damping terms (such as $\Phi^p$ models), for which we show the space-time localisation bound \eqref{eq:heuristic_bound}.
  There are singular stochastic PDEs of interest for which global well-posedness has been shown but no space-time localisation bound is expected, see for example \cite{BringmannCaoHiggs,ChevyrevShen23} for the stochastic 2D Abelian Higgs and Yang--Mills equations, \cite{BC24SG,CLFW24, SZZ25_gPAM} for the sine-Gordon and gPAM models, and \cite{Hairer_Rosati_NS,Hairer_Zhao_25_NS} for the stochastic Navier--Stokes equations.
  There are also equations of interest for which only local but no global solution theory is established,
  such as the 3D stochastic Yang-Mills equations \cite{CCHS_3D} and the geometric KPZ equation \cite{BGHZ22}.
  We find it an important problem to understand if scaling methods can be applied to such equations that lack super-linear damping terms. 
\end{itemize}

\subsection{Notation}
\label{sec:notation}

We collect some commonly used notation.
Denote $\N= \{0,1,\ldots\}$.
For $k\in\N$, denote $[k] = \{1,\ldots,k\}$, understood as the empty set if $k=0$.

For a set $S$, we interchangeably treat elements $\bk \in \N^S$ both as a function $\bk\colon S\to\N$ and as a multi-set $\bk = (k_1,\ldots,k_{|\bk|})$ with elements $k_i\in S$, where we denote by $|\bk|$ the total number of elements in $\bk$.
We denote $\bk! = \prod_{k\in S} \bk(k)!$.

We write $\bone_A$ for the indicator function of a set $A$.
For a function $\xi\colon A\to E$, where $E$ is a normed space, we denote
\begin{equ}
\|\xi\|_{\infty;A} = \sup_{a\in A}|\xi(a)|\;.
\end{equ}
We will sometimes use the shorthand $\xi_a=\xi(a)$ for $a\in A$.

Let $\floor{\alpha}\in\Z$ and $\roof{\alpha}\in\Z$ denote the floor and roof of $\alpha\in\R$ respectively.
Unless otherwise specified, for $\alpha >0$, we let $\CC^\alpha$ denote the space of $N$-times differentiable functions, where $N=\roof{\alpha}-1\in\N$, whose $N$-th derivative is $\gamma$-H\"older continuous, where $\gamma=\alpha-N \in (0,1]$.
We let $\CC^\infty$, $\CC^\infty_c$, and $\CC$ denote the space of smooth, smooth and compactly supported, and continuous functions respectively.
Unless otherwise specified or clear from the context, the target space of functions is $\R$.

We write $f*g$ for the convolution $f*g(x) = \int f(x-y)g(y)\mrd y$ of two functions $f,g$ defined on suitable subsets of $\R^d$.

We write $X\lesssim Y$ and $X\ll Y$ to indicate $X\leq CY$ and $CX \leq Y$ respectively for some $C\geq 1$ sufficiently large whose dependence on different parameters is made explicit or is clear from the context.

For vector spaces $X,Y$, we write $L(X,Y)$ for the space of linear maps $X\to Y$.

\subsubsection{Scaling}
\label{sec:scaling}

Let $d\geq 1$ and consider scaling $\s = (\s_1,\ldots,\s_d)\in\N^d$ with $\s_i \in [1,\infty)$ for all $i\in[d]$.
Denote $
|\s| = \s_1+\cdots+\s_d
$ and, for $x=(x_1,\ldots,x_d)\in\R^d$, define
\begin{equ}
|x|_\s = \max_{i\in[d]}|x_i|^{1/\s_i}\;.
\end{equ}
We equip $\R^d$ with the corresponding metric $(x,y)\mapsto |x-y|_\s$.
We also equip $\R^d$ with multiplication by $\lambda>0$ defined by \begin{equ}
\lambda \cdot_\s x = (\lambda^{\s_1}x_1,\ldots,\lambda^{\s_d}x_d)
\end{equ}
for which $|\lambda \cdot_\s x|_\s = \lambda|x|_\s$.

For $z \in \R^d$ and $\lambda > 0$, define the affine bijection
$T_{z, \lambda} \colon \R^d \to \R^d$ by
\begin{equ}[eq:T_lambda_def]
  T_{z, \lambda} (y) = z + \lambda \cdot_\s y
  \;,
\end{equ}
and for $\psi \colon \R^d\to \R$ denote
\begin{equ}
\psi^{\lambda}_{z}(y) = \lambda^{-|\s|} \psi (\lambda^{-1}\cdot_\s(y-z)) = \lambda^{-|\s|}\psi(T_{z,\lambda}^{-1}(y))\;.
\end{equ}
Define the open unit ball
\begin{equ}[eq:parabolic_ball]
\mfB(x,h) = \{y\in\R^d\,:\, |y-x|_\s < h\}
\;.
\end{equ}
For $r\in\N$, denote
\begin{equ}[eq:Br_def]
\CB^r = \{\psi\in\CC^\infty_c(\mfB(0,1))\, :\, \|D^k\psi\|_\infty \leq 1 \text{ for all } |k|_\s\leq r\}
\;.
\end{equ}
We sometimes call $\psi^\lambda_x$, where $\psi\in\CB^r$ and $x\in\R^d$, a \emph{test function at scale $\lambda>0$}.


\subsubsection{Parabolic case}
\label{sec:PDE_notation}

In Sections \ref{sec:classical} and \ref{sec:SPDE}, where we work with parabolic PDEs, we use the following notation.
Fix $d\geq 2$ an integer and let $\CL = \d_t-\Delta = \d_1 - \sum_{i=2}^d \d_i^2$ be the heat operator on $\R^d$.
We set
\begin{equ}[eq:parab_scale]
\s = (\s_1,\ldots,\s_d)=(2,1,\ldots,1)\;.
\end{equ}
Define the backwards parabolic open ball and its parabolic closure
\begin{equ}
\Omega = (- 1, 0) \times (- 1, 1)^{d-1} \subset\R^d
 \;,
 \qquad
 \bar\Omega = [- 1, 0) \times [- 1, 1]^{d-1}
 \;.
\end{equ}
We denote the corresponding parabolic boundary by
\begin{equ}
\d\Omega = \{z\in\bar\Omega\,:\,|z|_\s=1\}
\;.
\end{equ}
We also write $\cl \mfK$ for the topological closure of a subset $\mfK\subset \R^d$,
so that
\begin{equ}
\cl \Omega = [-1,0]\times[-1,1]^{d-1}
\;.
\end{equ}
(Note that $\bar\Omega\neq \cl \Omega$ and similarly $\d\Omega$ differs from the topological boundary of $\Omega\subset\R^d$.)

Define also the rescaled and shifted open parabolic ball
\begin{equ}
B_z(\lambda) = T_{z,\lambda}(\Omega)\;.
\end{equ}

\textbf{Distributions.}
For an open set $\mfK\subset\R^d$, let $\CD'(\mfK)$ denote the space of distributions on $\mfK$.
For $\xi \in \CD' (\mfK)$, $\beta \in \R$, $h>0$, $r\in\N$, and $x \in \mfK$ such that $B_x(h)\subset \mfK$,
define
\begin{equation}
  \| \xi \|_{\beta ; x ;  h; r} = \sup_{\psi\in\CP^r} \sup_{\lambda\in (0,h)} | \lambda^{- \beta} \scal{  \xi, \psi^{\lambda}_x}  | \;, \label{eq:Cbeta-def}
\end{equation}
where 
\begin{equ}
\CP^r = \{\psi\in\CC^\infty_c(\Omega)\, :\, \|D^k\psi\|_\infty \leq 1 \text{ for all } |k|_\s\leq r \} \subset \CB^r
\;.
\end{equ}
For $\psi\in\CP^r$, we sometimes call $\psi^\lambda_x$ a \emph{non-anticipative} test function at scale $\lambda>0$.
If $r=\max\{0,\floor{1-\beta}\}$ and $h=1$, we drop $r,h$ from the notation and write
\begin{equ}
\| \xi \|_{\beta ; x ; h} = \| \xi \|_{\beta ; x ; h; r}\;,
\qquad
\| \xi \|_{\beta ; x } = \| \xi \|_{\beta ; x ; 1} 
\;.
\end{equ}
For $\beta \leq 0$, let $\CC^{\beta} (\mfK)$ denote the space of distributions $\xi\in\CD'(\mfK^1)$, where
\begin{equ}[eq:fattening]
\mfK^1 = \cup_{x\in \mfK} B_x(1)
\end{equ}
(the $1$-fattening of $\mfK$ using backwards parabolic balls),
equipped with the semi-norm
\begin{equ}
\| \xi \|_{\CC^{\beta} (\mfK)} \eqdef \sup_{x \in \mfK} \| \xi \|_{\beta ; x} \; .
\end{equ}
For $\beta \leq 0$, $\xi\in\CD'(\mfK)$, $x\in\R^d$, and $h>0$ such that $T_{x,h}(\Omega^1) \subset \mfK$, define further
\begin{equ}
\| \xi \|_{\CC^\beta ; x ;  h} = \sup_{y \in B_x(h)} \|\xi\|_{\beta;y; h}
\;,
\qquad
\|\xi \|_{\CC^\beta ; x} = \| \xi \|_{\CC^\beta ; x ;  1}
\;.
\end{equ}
Note that $B_{y}(h)\subset T_{x,h}(\Omega^1)\subset \mfK$ for all $y\in B_x(h)$, so that $\|\xi\|_{\beta;y; h}$ indeed makes sense.
 
Observe that, for all for $\lambda,\delta>0$,
\begin{equ}
\scal{\xi\circ T_{z,\lambda}, \psi^{\delta}_x}
=
\scal{\xi, \psi^{\lambda\delta}_{T_{z,\lambda}(x)}}
\end{equ}
and therefore, for $\beta\in\R$,
\begin{equ}[eq:scale_pointwise]
\|\xi\circ T_{z,\lambda}\|_{\beta;x;h;r}
=
\lambda^\beta \|\xi\|_{\beta;T_{z,\lambda}(x);\lambda h;r}
\end{equ}
and, for $\beta \leq 0$,
\begin{equ}[eq:scaling_local]
\| \xi \circ T_{z, \lambda} \|_{\CC^{\beta} (\Omega)}
= \lambda^{\beta} \| \xi
   \|_{\CC^{\beta} ; z ; \lambda}\;.
\end{equ}

\section{Main idea in classical setting}\label{sec:classical}

In this section, we demonstrate the main idea of our approach in the simple
setting of classically well-posed non-linear heat equations.
We use notation as in Section \ref{sec:PDE_notation}.

Let $p \in (1,\infty)$ and, for $x\in\R^m$, write $x^p = |x|^{p-1}x\in\R^m$.
Suppose $\phi \colon \Omega \rightarrow
\mathbb{R}^m$ satisfies
\begin{equ}[eq:classical_PDE]
  \CL \phi =  - \phi^{p} + \xi
\end{equ}
where $\xi \in \CC^{\beta} (\Omega)$ with $\beta \in (- 2, 0]$.

\begin{remark}
By parabolic regularity, since $\beta>-2$, \eqref{eq:classical_PDE} admits a unique solution $\phi\in\CC(\Omega)$ for any boundary data in $\CC(\d\Omega)$.
Below we use existence (but not uniqueness) of solutions.
\end{remark}

We derive an a priori estimate on $\phi$ that is independent of boundary data.
To give some intuition for the statement and proof of the estimate, consider the \emph{scaling critical} dimension
\begin{equ}[eq:alpha_def]
\alpha = 2/(p-1)
\end{equ}
for the `leading terms' $\CL \phi  + \phi^p$ in \eqref{eq:classical_PDE}.
Here $\alpha$ is derived from the equation
\begin{equ}
\alpha + 2 = \alpha p \;,
\end{equ}
which represents that $\CL \phi$ and $\phi^p$ scale in the same way under the transformation $\phi\mapsto \lambda^{\alpha} \phi \circ T_{z,\lambda}$.
Namely, if $\phi$ solves \eqref{eq:classical_PDE}, then $\psi = \lambda^{\alpha} \phi \circ T_{z,\lambda}$ solves
\begin{equs}[eq:intertwine-classical]
    \CL \psi = ( -\lambda^{\alpha + 2} \phi^p +
    \lambda^{\alpha + 2} \xi  ) \circ T_{z, \lambda}
    &= -\psi^p + \lambda^{ \alpha + 2
    } \xi \circ T_{z, \lambda} \;,
  \end{equs}
  where we used that $\lambda^{\alpha + 2} \phi^p\circ T_{z,\lambda} = \lambda^{\alpha+2 - \alpha p} \psi^p$ and $\alpha + 2 - \alpha p=0$.
Then, denoting
\begin{equ}[eq:eta-def-classical]
\eta = \lambda^{\alpha + 2} {\xi}  \circ T_{z, \lambda}
    \in \CC^{\beta} (\Omega)
\end{equ}
and using \eqref{eq:scaling_local}, we have
\begin{equ}
    \| \eta \|_{\CC^{\beta} (\Omega)} = \lambda^{\alpha + 2 + \beta } \| \xi \|_{\CC^{\beta };z; \lambda}\; . \label{eq:eta-rho-simple}
  \end{equ}
  Hence, for $\lambda\ll 1$, $\psi$ solves the same equation but with a small driver.
Remark now that, if $\| \xi \|_{\CC^{\beta };z;\lambda} \gg 1$ and we choose $\lambda \ll 1$ such that $\|\eta\|_{\CC^{\beta}(\Omega)}=r$ for some $r>0$ small but strictly positive,
then, from local stability of the equation (and ignoring for now the effects of boundary conditions),
we expect $\|\psi\|_{\infty;\Omega} = \lambda^{\alpha}\|\phi\|_{\infty;B_{z}(\lambda)} \lesssim 1$.
Using \eqref{eq:eta-rho-simple}, we thus expect the estimate
\begin{equ}
\|\phi\|_{\infty;B_{z}(\lambda)} \lesssim \lambda^{-\alpha} = \lambda^{-\rho (\alpha+2 + \beta)} \sim \| \xi \|_{\CC^{\beta };z; \lambda}^{\rho}
\end{equ}
where
\begin{equ}[eq:rho_def_classical]
\rho = \frac{\alpha}{\alpha + 2 + \beta} \in (0,\infty)\;.
\end{equ}

The following proposition makes this heuristic precise and further shows space-time bounds independent of the boundary conditions (i.e. space-time localisation).
Define the distance-to-the-boundary function $\dist{\cdot}  \colon\cl\Omega\to [0,1]$ by
\begin{equ}[eq:parabolic_dist]
\dist{z} = \inf_{y \in \d \Omega} |z -y|_\s
\;.
\end{equ}

\begin{proposition}
  \label{prop:classical} There exists $C>0$
  such that, if $\xi,\phi \in \CC(\cl\Omega)$ satisfy \eqref{eq:classical_PDE} in $\Omega$,
  then for all $z \in \Omega$
\begin{equation}\label{eq:less-sharp}
  | \phi (z) | 
    \leq 
    C \max\{  \|
    \xi \|^{\rho}_{\CC^{\beta}(\Omega)} , \dist{z}^{- \alpha} \}\;,
  \end{equation}
  where $\rho$ is given by \eqref{eq:rho_def_classical}
  and where we extend $\xi$ by $0$ outside $\cl\Omega$
  to make sense of $ \|\xi \|_{\CC^{\beta}(\Omega)}$.

More precisely, there exists $C>0$
such that
  \begin{equation}
    | \phi (z) | 
    \leq C \max\{  \|
    \xi \|^{\rho}_{\CC^{\beta} ; z ; \mu_z} , \dist{z}^{- \alpha} \} \;, \label{eq:a-priori}
  \end{equation}
  where  $\mu_z = (C| \phi (z) |^{-1/\alpha})\wedge (\frac{1}{4}\dist{z})$.
\end{proposition}

  \begin{remark}
  With simple modifications, we can handle the equation $\CL \phi = - \phi^p + \sum_{k=0}^{p-1}\xi_k\phi^k$ where $\xi_k \in \CC^{\beta_k}$ for $\beta_k\leq 0$ satisfying scaling subcriticality for all $0\leq k\leq p-1$
  \begin{equ}
  \beta_k - \alpha k > - 2 - \alpha
  \end{equ}
  and complementary Young regularity for all $1\leq k \leq p-1$
\begin{equ}
\beta_k + 2 + \min_{0 \leq i \leq p - 1} \beta_i > 0
\;.
\end{equ}
This would cover the `remainder' in the stochastic quantisation
    equation of $\Phi^4_2$ considered by Da Prato--Debussche \cite{DPD02}, a priori estimates for which were studied in \cite{MW17Phi42} with different methods  (and without space-time localisation bounds).

  We can also consider source terms with mixed homogeneity, such as $\CL \phi = - \phi^p + \xi +\zeta$ with $\zeta\in\CC^\gamma$ where $\gamma \in (-2,0]$,
  and the term $\|
    \xi \|^{\rho}_{\CC^{\beta} ; z ; \mu_z}$ in \eqref{eq:a-priori} would become $\|
    \xi \|^{\rho}_{\CC^{\beta} ; z ; \mu_z}
    \vee \|
    \zeta \|^{\frac{\alpha}{\alpha+2+\gamma}}_{\CC^{\gamma} ; z ; \mu_z}$.
  Taking $\gamma=0$ recovers the main result of \cite{MW20_reaction}.
  \end{remark}

\begin{remark}
The qualitative assumption that $\phi$ and $\xi$ are continuous is used only to simplify the proof.
We will later prove a more general result without this assumption.
\end{remark}

\begin{remark}
There are two reasons that the `local' quantity $\|
    \xi \|^{\rho}_{\CC^{\beta} ; z ; \mu_z}$ in \eqref{eq:a-priori} is more natural than the `global' quantity $\|
    \xi \|^{\rho}_{\CC^{\beta}(\Omega)} $ in the less sharp bound \eqref{eq:less-sharp}.

First, since $\mu_z \leq \frac{1}{4}\dist{z}$,
$\| \xi \|_{\CC^{\beta} ; z ; \mu_z}$ depends only on the quantities $\scal{\xi,\psi}$ where $\supp\psi \subset B_z(\frac12\dist{z})$ for $z\in\Omega$.
In particular, $\| \xi \|_{\CC^{\beta} ; z ; \mu_z}$ depends only on $\xi$ as an element of $\CD'(\Omega)$.
In contrast, $\| \xi \|_{\CC^{\beta}(\Omega)}$ artificially depends on $\xi$ outside $\Omega$.

Second, for $\beta \leq \gamma \leq 0$, we have the interpolation bound
\[ \| \xi \|_{\CC^{\beta} ; z ;  \eta} \leq \eta^{\gamma - \beta} \| \xi
   \|_{\CC^{\gamma} ; z ;  \eta} . \]
Suppose now we know \eqref{eq:a-priori} for some $\beta \leq 0$.
Then, for $\gamma \in [\beta,0]$,
either $|\phi(z)| \lesssim \dist{z}^{-\alpha}$, or
\[
|\phi(z)| \lesssim \| \xi \|^{\rho }_{\CC^{\beta} ; z ; \mu_z} 
\leq \mu_z^{\rho (\gamma - \beta)} \| \xi \|_{\CC^{\gamma} ; z ;  \mu_z}^{\rho}
\lesssim | \phi (z) |^{- \rho (\gamma - \beta) / \alpha} \| \xi
   \|_{\CC^{\gamma} ; z ; \mu_z}^{\rho}
   \]
and hence $| \phi (z) |^{1 + \rho (\gamma - \beta) / \alpha} \lesssim \| \xi \|_{\CC^{\gamma} ; z ; \mu_z}^{\rho}$ and
\begin{equ}
| \phi (z) | \lesssim \| \xi \|_{\CC^{\gamma} ; z ; \mu_z}^{\rho / (1 + \rho (\gamma - \beta) / \alpha)
}
   \;.
\end{equ}
Observe now that $1 + \rho (\gamma - \beta) / \alpha = 1 + \frac{\gamma -
\beta}{\alpha+2 - \alpha i + \beta} = \frac{\alpha + 2  + \gamma}{\alpha+2  + \beta}$
and thus $\rho / (1 - \rho (\gamma - \beta) / \alpha) = \frac{\alpha}{\alpha + 2 + \gamma}$,
which is the correct exponent for the bound \eqref{eq:a-priori} if we consider $\gamma$ in place of $\beta$.
Therefore \eqref{eq:a-priori} for a \textit{single} choice of $\beta$ implies
the same bound for all $\gamma \in [\beta,0]$ with the correct exponents.
No such implication holds for the bound \eqref{eq:less-sharp}.
\end{remark}

The proof of Proposition \ref{prop:classical} relies on a scaling argument
plus the following simple coercivity property of the equation with small
drivers.

\begin{lemma}
  \label{lem:coercive_classical}
  There exist $\eps,\delta, r > 0$
  such that, if $\xi\in\CD'(\Omega^1)$ and $\phi \in \CC (\cl \Omega)$ satisfy \eqref{eq:classical_PDE} on $\Omega$ and $
  \| \xi \|_{\CC^{\beta}(\Omega)} \leq r$ and $\| \phi \|_{\infty;\Omega}
  \leq \eps+\delta$, then $| \phi (0) | < \eps$.
\end{lemma}

\begin{proof}
The proof is easy and relies on the maximum principle and a comparison to the case $\xi=0$. We give the details as we follow the same strategy later in a more complicated setting.

Consider $\mathring\phi\in\CC(\Omega)$ solving $\CL
  \mathring\phi = - \mathring\phi^p$ in $\Omega$, i.e. \eqref{eq:classical_PDE} with $\xi=0$.
  Then, for every $\eps>0$, there exists $\delta = \delta(\eps)\in (0,\frac{\eps}{2})$ such that, if $\|\mathring \phi \|_{\infty;\Omega}
  \leq \eps+\delta$, then $| \mathring \phi (0) | < \eps-\delta$.
  
  Consider now $\phi\in\CC(\cl\Omega)$ solving \eqref{eq:classical_PDE} on $\Omega$ with $\| \phi \|_{\infty;\Omega}  \leq \eps+\delta$.
  Let $\mathring\phi $ be as above with boundary data $\mathring\phi \restriction_{\d\Omega} = \phi \restriction_{\d\Omega}$
  (such $\mathring\phi$ exists by parabolic regularity).
Then $\|\mathring\phi\|_{\infty;\Omega} \leq \eps+\delta$ by the maximum principle, and thus $|\mathring\phi(0)|<\eps-\delta$.

Furthermore, on $\Omega$,
\begin{equ}
\CL(\phi-\mathring\phi) = \mathring\phi^p - \phi^p + \xi
\;.
\end{equ}
Consider now a truncation $K$ of the heat kernel
such that $\supp(K)\subset B_0(1/2)$ and $\CL K = \delta_0 + U$ where $U$ is smooth and $\supp(U) \subset B_0(1/2)$.
Define $Y\colon\bar\Omega\to\R$ by
\begin{equ}[eq:diff_classical]
\phi - \mathring\phi = K*g + Y
\end{equ}
where $g = \xi + \bone_{\Omega}(\mathring\phi^p - \phi^p) \in \CD'(\Omega^1)$.
Then, on $\Omega$,
\begin{equ}[eq:Y_equ_classical]
\CL Y + U* g = 0\;. 
\end{equ}
Since $\beta>-2$ and $U$ is smooth with small support, we have $\|U*\xi\|_{\infty;\Omega} \lesssim \|\xi\|_{\CC^\beta(\Omega)} \ll \delta$ whenever $\|\xi\|_{\CC^\beta(\Omega)} \leq r$ and $r\ll \delta$.
Since $\|\mathring \phi\|_{\infty;\Omega}\vee \|\phi\|_{\infty;\Omega}\leq 2\eps$, we also have
\begin{equ}
\|U*\{\bone_{\Omega}(\mathring\phi^p - \phi^p)\}\|_{\infty;\Omega}\lesssim  \eps^{p-1}\|\mathring\phi-\phi\|_{\infty;\Omega}
\;.
\end{equ}
It follows that
\begin{equ}[eq:Ug]
\|U*g\|_{\infty;\Omega}\leq c \delta + C  \eps^{p-1} \|\mathring\phi-\phi\|_{\infty;\Omega}
\end{equ}
where $c\downarrow 0$ as $r\downarrow0$.
Here and below, $C>0$ is a constant whose value may change from line to line, but does not depend on $\eps>0$.

In a similar way, using a scale decomposition of $K$ (see, e.g. \cite[Lem.~5.5]{Hairer14}) and the fact that $(\mathring\phi-\phi)\restriction_{\d\Omega}=0$, we obtain for $r$ sufficiently small
\begin{equ}[eq:Kg]
\|Y\|_{\infty;\d\Omega} \leq
\|K*g\|_{\infty;\bar\Omega}\leq \delta/4 + C\eps^{p-1} \|\mathring\phi-\phi\|_{\infty;\Omega}
\;.
\end{equ}
It follows from \eqref{eq:Y_equ_classical}, \eqref{eq:Ug}, \eqref{eq:Kg}, and the maximum principle that, for $r$ sufficiently small,
\begin{equ}
\|Y\|_{\infty;\bar\Omega} \leq \delta/2 + C \eps^{p-1} \|\mathring\phi-\phi\|_{\infty;\Omega}
\;.
\end{equ}
Inserting this back into \eqref{eq:diff_classical} and combining with (the second bound in) \eqref{eq:Kg} we obtain
\begin{equ}
\|\mathring\phi-\phi\|_{\infty;\Omega} \leq 3\delta/4 + C \eps^{p-1} \|\mathring\phi-\phi\|_{\infty;\Omega}
\;.
\end{equ}
It follows that, for all $\eps>0$ sufficiently small, $\|\mathring\phi-\phi\|_{\infty;\Omega}\leq \delta$
and therefore, since $|\mathring\phi(0)|<\eps-\delta$, one has $|\phi(0)|<\eps$
as desired.
\end{proof}


\begin{proof}[Proof of Proposition \ref{prop:classical}]
Let $\eps,\delta$ be as in Lemma \ref{lem:coercive_classical} and $h>1$ such that
\begin{equ}[eq:h_bound]
1+\eps^{1/\alpha}(\eps+\delta)^{-1/\alpha}h \leq h\;.
\end{equ}
 Since we assumed that $\phi,\xi\in\CC(\cl\Omega)$, there exists $z \in \cl\Omega$ which maximises
  \begin{equation}
    \frac{| \phi (z) | }{\max\{ \| \xi
    \|^{\rho}_{\CC^{\beta} ; z ; h \lambda_z} , \dist{z}^{- \alpha} \}} \eqdef \frac{|\phi(z)|}{Q} \eqdef C_* \label{eq:z_def}
  \end{equation}
  where we set $\lambda_z^{\alpha} = (\eps | \phi (z) |^{-1} )\wedge (\frac{1}{4h}\dist{z})^\alpha$.
  Without loss of generality, we assume $\dist{z}>0$.
  We will derive a bound on $C_*$, from which the proof will follow with $\mu_z = h\lambda_z$.
  
  We first claim that there exists $C_0 >0$ such that, if $C_*>C_0$, then
  \begin{enumerate}[label=(\alph*)]
  \item\label{pt:first_bound} $\eps | \phi (z) |^{-1} \leq (\frac1{4h}\dist{z})^\alpha$, so in particular $\lambda_z^{\alpha} = \eps|\phi (z) |^{-1}$, and
  \item for all $x \in B_z(\lambda_z)$,
  \begin{equation}
    \dist{x}^\alpha > \eps(\eps+\delta)^{-1} \dist{z}^\alpha \;. \label{eq:dist_bound}
  \end{equation}
  \end{enumerate}
  Indeed, both
  statements hold if $\lambda_z \ll \dist{z}$. Since
  \[ \lambda_z^{\alpha} \leq \eps|\phi(z)|^{-1} =  \eps C_*^{- 1} Q^{- 1} \leq \eps C_*^{- 1} \dist{z}^{\alpha} \;, \]
  the existence of $C_0$ follows.
  We suppose henceforth that $C_*>C_0$ as otherwise \eqref{eq:a-priori} holds with $C_*=C_0$ and we are done.
  
  Next, we define $\psi = \lambda^{\alpha}_z \phi \circ T_{z, \lambda_z} \in
  \CC (\cl\Omega)$ and observe that
  \begin{equation}
    | \psi  (0) | = \lambda^{ \alpha}_z | \phi (z) | = \eps
    \;. \label{eq:psi-zero}
  \end{equation}
  Furthermore, by \eqref{eq:intertwine-classical}, in $\Omega$,
  \begin{equation}
    \CL \psi = - \psi^p +  \eta \label{eq:psi-equ}
  \end{equation}
  with $\eta\in\CC^\beta(\Omega)$ as in \eqref{eq:eta-def-classical}.
  Then, by \eqref{eq:eta-rho-simple} and the definition of $\rho$ in \eqref{eq:rho_def_classical},
  \begin{equation}
    \| \eta \|_{\CC^{\beta} (\Omega)}^{\rho} = \lambda_z^{(\alpha + 2  + \beta) \rho } \| \xi \|_{\CC^{\beta } ; z ; \lambda_z}^{\rho} \leq \lambda_z^{\alpha} Q  = \eps C_*^{- 1}\;. \label{eq:eta-rho}
  \end{equation}
  
  We next claim that $\| \psi \|_{\infty;\Omega} \leq \eps+\delta$,
  which is equivalent to
  \begin{equ}
  \sup_{x \in B_{z}(\lambda_z)} \eps | \phi (x) | \leq (\eps+\delta)|\phi(z)|
  \;.
  \end{equ}
  Arguing by contradiction, suppose there exists
  $x\in B_{z}(\lambda_z)$ such that
  \begin{equ}[eq:x_too_large]
  \eps | \phi (x) | > (\eps+\delta)|\phi(z)|\;.
  \end{equ}
  Then
  \begin{equ}
  \eps|\phi(x)|^{-1} < \eps^2(\eps+\delta)^{-1}|\phi(z)|^{-1}
  \leq
  \eps(\eps+\delta)^{-1} \Big(\frac{\dist{z}}{4h}\Big)^\alpha
  < \Big(\frac{\dist{x}}{4h} \Big)^\alpha
  \;,
  \end{equ}
  where we used \eqref{eq:x_too_large} in the first bound, \ref{pt:first_bound} in the second bound,
  and \eqref{eq:dist_bound} in the third bound.
  In particular, recalling $\lambda_x^\alpha = (\eps|\phi(x)|^{-1}) \wedge (\frac1{4h}\dist{x})^\alpha$,
  \begin{equ}
  \lambda_x^\alpha = \eps|\phi(x)|^{-1} < \eps^2(\eps+\delta)^{-1}|\phi(z)|^{-1} = \eps(\eps+\delta)^{-1} \lambda_z^\alpha\;.
  \end{equ}
Therefore, by \eqref{eq:h_bound}, we have $\lambda_z + h\lambda_x \leq h\lambda_z$,
and thus $B_x(h\lambda_x) \subset B_{z}(h\lambda_z)$
and
\begin{equ}[eq:xi_balls]
\| \xi \|^{\rho}_{\CC^{\beta}; x ; h\lambda_x}
  \leq 
   \| \xi \|^{\rho}_{\CC^{\beta} ; z ; h\lambda_z}\;.
\end{equ}
It follows that
  \begin{equs}
  | \phi (x) |
  &\leq C_* \max\{
     \| \xi \|^{\rho}_{\CC^{\beta} ; x ;  h\lambda_x}
     , \dist{x}^{- \alpha} \}
     \\
     &\leq
     C_* \max\{ \| \xi \|^{\rho}_{\CC^{\beta} ; z ; h \lambda_z}
     , \eps^{-1}(\eps+\delta) \dist{z}^{- \alpha} \} \leq \eps^{-1} (\eps+\delta)
     |\phi(z)|\;,
    \end{equs}
    where we used the definition of $C_*$ in \eqref{eq:z_def} for the first and third bounds,
    and \eqref{eq:dist_bound} and \eqref{eq:xi_balls} for the second bound.
    This contradicts \eqref{eq:x_too_large} and concludes the proof of $\| \psi \|_{\infty;\Omega} \leq \eps+\delta$.

  To conclude the proof, by \eqref{eq:eta-rho}, there exists $C_1>0$ such that, if $C_*>C_1$, then $\| \eta
  \|_{\CC^{\beta}(\Omega)} \leq r$ where $r$ is from Lemma
  \ref{lem:coercive_classical}.
  Since $\psi$ solves \eqref{eq:psi-equ} and $\|
  \psi \|_{\infty} \leq \eps + \delta$, Lemma
  \ref{lem:coercive_classical} implies that $| \psi (0) | < \eps$ if $C_*>C_1$.
  But this is a
  contradiction to \eqref{eq:psi-zero}, so we conclude that $C_*\leq C_1$.
\end{proof}

\section{Abstract bounds}\label{sec:abstract}

We now flesh out the abstract conditions that make the above argument work. An
important remark is that it is hardly important that $\phi$ solves a
PDE - the important ingredients are Lemma \ref{lem:coercive_classical}, which
controls the behaviour of the solution for small drivers,
and the scaling relation \eqref{eq:intertwine-classical}.
Our abstract assumptions are as follows.

\textbf{Domains.} Let $(\Sigma, \rho)$ be a metric space, called the `domain', with a distinguished set
$\d \Sigma \subset \Sigma$ called the `boundary'.
We assume that, for all $z\in \Sigma$,
\[
\dist{ z} \eqdef \inf_{x \in \d \Sigma } \rho (x, z) \in [0,\infty) \; .
\]

We assume that, for every $z \in \Sigma$ and $0 < \lambda \leq
\dist{z}$, there are sets $\Sigma_{z, \lambda}$ with a
`boundary' $\d \Sigma_{z, \lambda} \subset \Sigma_{z, \lambda}$ and an injection
$T_{z, \lambda} \colon \Sigma_{z, \lambda} \hookrightarrow \Sigma$ such that $T_{z, \lambda}
(o_{z, \lambda}) = z$ for some $o_{z, \lambda} \in \Sigma_{z, \lambda}$.

\begin{remark}
In this section, and this section only, we deviate from the notation of Section \ref{sec:notation}, so the map $T_{z, \lambda}$ is \emph{not} assumed to be of the form \eqref{eq:T_lambda_def}.
\end{remark}

\begin{notation}
  Below we introduce several objects indexed by $z \in \Sigma$ and $0 < \lambda
  \leq \dist{z}$. Unless otherwise stated, we let
  $z, \lambda$ run over all such pairs.
\end{notation}

We suppose that
\begin{equ}[eq:diameter]
  {\sup_{x \in \Sigma_{z, \lambda}}}  \rho(z, T_{z,
  \lambda}(x)) \leq \lambda  \;. 
\end{equ}
See Figure \ref{fig:domains} for an illustration.

\begin{figure}[ht]
\centering
\begin{tikzpicture}[line width=1pt,>=Stealth,
                    every node/.style={font=\small},
                    scale=0.7]

\draw plot[smooth cycle, tension=0.8] coordinates{
  (-4, 0) (-3, 1) (-0.8, 0.8) (1.8, 3.0)
  (3.8, 1.9) (4.3, 0) (3.6,-2.1) (0.8,-3.3) (-2.5,-2.4)};
\node[font=\large] at (0.4, 2.6) {$\d\Sigma$};
\node[font=\large] at (-2.5,-0.5) {$\Sigma$};

\coordinate (z) at (2.5,0);
\def\lam{0.9}
\draw[red,thick,fill=red!15] (z) circle (\lam);
\fill (z) circle (2pt);
\node[xshift=3pt,yshift=-5pt] at (z) {$z$};

\draw[thin,<->]
      (z) -- ($(z)+(-\lam*0.69,-\lam*0.69)$) node[midway,xshift=-5pt,yshift=5pt] {\scriptsize $\lambda$};

\node at ($(z)+(0,\lam)+(0,0.3)$) {$T_{z,\lambda}(\Sigma_{z,\lambda})$};
\coordinate (C) at (7,-1.5);
\def\R{2.3}
\draw[red,thick,fill=red!15] (C) circle (\R);
\fill (C) circle (2pt);
\node at ($(C)+(0,-0.3)$) {$o_{z,\lambda}$};
\node[above,yshift=-3pt] at ($(C)+(0,\R)$) {$\d\Sigma_{z,\lambda}$};
\node at ($(C)+(\R/2,-\R/2)$) {$\Sigma_{z,\lambda}$};

\draw[thick,->]
      ($(C)+(-0.8,0.25)$) -- ($(z)+(0.5,-0.15)$)
      node[midway,xshift=20pt,yshift=1pt] {$T_{z,\lambda}$};

\coordinate (SigBound) at (4.3,0);   
\draw[thin,<->]
      (z) -- (SigBound) node[midway,xshift=6pt,yshift=6pt] {{\scriptsize $\dist{z}$}};

\end{tikzpicture}
\caption{Example of domain configuration.}
\label{fig:domains}
\end{figure}
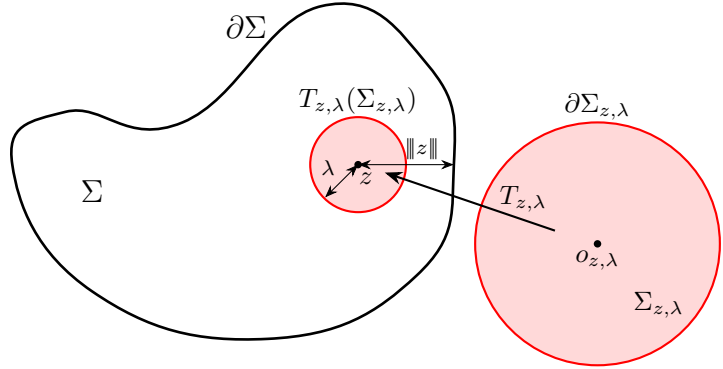

\textbf{Drivers.} There are sets of `drivers' $\drivers$ and $\mathbb{M}_{z, \lambda}$, a `scaling' map $R_{z, \lambda} \colon \mathbb{M} \rightarrow
\mathbb{M}_{z, \lambda}$,
and a `norm' $\| \cdot \|_{z, \lambda} \colon \drivers_{z,\lambda} \to [0,\infty)$.
We also suppose there are `local norms' $\disc{\cdot}_{z,\lambda} \colon \drivers \to [0,\infty)$ satisfying the following properties.
Consider $\lambda \leq \dist{z}$, $x\in T_{z,\lambda}(\Sigma_{z,\lambda})$, and $\lambda'\leq\dist{x}$
such that $\rho(x,z) + \lambda' \leq \lambda$
(this condition in particular implies that the ball of radius $\lambda'$ centred at $x$ is contained in the ball of radius $\lambda$ centred at $z$, see Figure \ref{fig:balls}).
Then
\begin{equ}[eq:balls]
\disc{M}_{x,\lambda'} \leq \disc{M}_{z,\lambda}\;.
\end{equ}
Furthermore, we assume that there exists $\alpha>0$ such that
\begin{equ}[eq:M_norms]
\|R_{z,\lambda} M\|_{z,\lambda} \leq \lambda^\alpha \disc{M}_{z,\lambda}\;.
\end{equ}

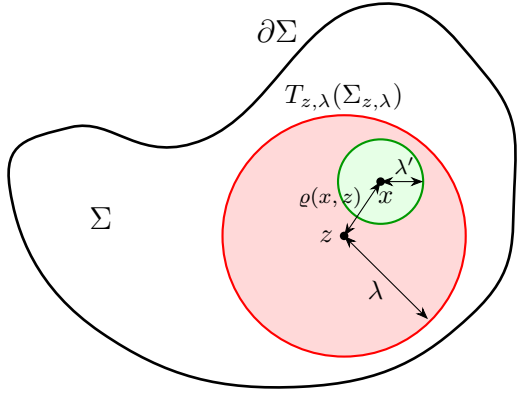
\begin{figure}[ht]
\centering

\begin{tikzpicture}[>=Stealth, line width=1pt,
                    every node/.style={font=\small},
                    scale=0.8]

\draw plot[smooth cycle, tension=.8] coordinates{
  (-4,0) (-3,1) (-.8,.8) (1.8,3) (3.8,1.9) (4.3,0)
  (3.6,-2.1) (.8,-3.3) (-2.5,-2.4)};
\node[font=\large] at (.4,2.6) {$\d\Sigma$};
\node[font=\large] at (-2.5,-.5) {$\Sigma$};

\coordinate (z) at (1.5,-0.8);   
\def\lam{2}                  

\draw[red,thick,fill=red!15] (z) circle (\lam);
\fill (z) circle (2pt);
\node[left] at (z) {$z$};

\node at ($(z)+(0,\lam)+(0,0.3)$) {$T_{z,\lambda}(\Sigma_{z,\lambda})$};

\draw[<->,thin]
      ($(z)+(0.69*\lam,-.69*\lam)$) -- (z)
      node[midway,below left=-4pt] {$\lambda$};

\def\alpha{.35}
\coordinate (x) at ($(z)+(.6,.9)$);   
\def\rx{\alpha*\lam}

\draw[green!60!black,thick,fill=green!10] (x) circle (\rx);
\fill (x) circle (2pt);
\node[xshift=2pt,yshift=-6pt] at (x) {$x$};

\coordinate (xp) at ($(x)+(\rx,0)$);
\draw[<->,thin] (x) -- (xp) node[midway,yshift=6pt,xshift=1pt] {\scriptsize $\lambda'$};

\draw[<->,thin]
      (z) -- (x) node[midway,yshift=4pt,xshift=-12pt] {\scriptsize $\rho(x,z)$};

\end{tikzpicture}

\caption{Example configuration of $\lambda,\lambda',x,z$.}
\label{fig:balls}
\end{figure}

\begin{remark}
  The maps $R_{z, \lambda}$ are the abstract analogues of $\xi \mapsto
  \lambda^{\alpha + 2 } \xi \circ T_{z, \lambda}$ from
  \eqref{eq:eta-def-classical}.
\end{remark}

\textbf{Solutions.}
Let $(E,|\cdot|)$ be a normed vector space.
We equip $E^{\Sigma_{z,\lambda}}$
with the uniform norm $\| \phi \|_{\infty}
= \sup_{x \in \Sigma_{z, \lambda}} | \phi (x) |$.
For $\alpha>0$ as in \eqref{eq:M_norms}, define $R_{z, \lambda} \colon E^\Sigma \rightarrow E^{\Sigma_{z,\lambda}}$ by
\[ R_{z, \lambda} (\phi) = \lambda^{\alpha} \phi \circ T_{z,
   \lambda} \;. \]
   For every $M\in\drivers$ (respectively $M\in \drivers_{z,\lambda}$)
   consider a subset $\bbS_M \subset E^{\Sigma}$ (respectively $\bbS_M \subset E^{\Sigma_{z,\lambda}}$),
   which we call the set of `solutions' with respect to $M$.
   We assume that $\bbS_M$ is closed under rescaling in the sense that,
   for all $M\in\drivers$,
   \begin{equ}[eq:sol_rescale]
   R_{z,\lambda} \bbS_{M} \subset \bbS_{R_{z,\lambda} M}\;.
   \end{equ}


\textbf{Local coercivity.} Our final assumption is that there exists $\delta > 0$ such that, for all $z\in \Sigma$, $0< \lambda \leq \dist{z}$,
$M \in \mathbb{M}_{z, \lambda}$ with $\| M \|_{z, \lambda} < 1$, and $\phi \in \bbS_{M}$ with $\| \phi \|_{\infty} < 1 + \delta$, we have
\begin{equation}
  | \phi (o_{z, \lambda}) | < 1\;.
  \label{eq:local_coercivity_abstract}
\end{equation}
(This is the analogue of Lemma \ref{lem:coercive_classical}.)

\begin{theorem}
  \label{thm:a_priori_abstract}
  With the above assumptions, consider $M\in \drivers$ and $\phi \in \bbS_M$.
 Let
  \[
  \theta = (1+\delta)^{-1/\alpha}\;,
  \qquad
  \nu \geq (1-\theta)^{-1}
  \;. \]
  Then for all $z \in \Sigma$
  \begin{equ}[eq:a-priori-interior]
    | \phi(z) | \leq  \max\{ \disc{M}_{z, \lambda_z/(1-\theta) } , \nu^\alpha \dist{z}^{-\alpha} \}
  \end{equ}
  where $\lambda_z = |\phi(z)|^{-1/\alpha} \wedge (\frac1{\nu}\dist{z})$.
\end{theorem}


Before the proof, we give an immediate corollary in which we slightly generalise the form of the local coercivity assumption.

\begin{corollary}\label{cor:parameter_choice}
Suppose that, instead of local coercivity,
there exist $\eps,\kappa,r>0$ such that, for all $z\in \Sigma$, $0< \lambda \leq \dist{z}$,
$M \in \mathbb{M}_{z, \lambda}$ with $\| M \|_{z, \lambda} < r$, and $\phi \in \bbS_{M}$ with $\| \phi \|_{\infty} < \eps$, we have $| \phi (o_{z, \lambda}) | < \eps-\kappa$.
Denote $1+\delta = \frac{\eps}{\eps-\kappa}$ and $\theta = (1+\delta)^{-1/\alpha}$,
and let $\nu\geq (1-\theta)^{-1}$.
Then
\begin{equ}
    | \phi(z) | \leq \frac{\eps}{1+\delta} \max\{ r^{-1}\disc{M}_{z, \lambda_z/(1-\theta)} ,
    \nu^\alpha \dist{z}^{-\alpha} \}
  \end{equ}
where $\lambda_z = (\frac{1+\delta}{\eps}|\phi(z)|)^{-1/\alpha} \wedge (\frac1{\nu}\dist{z})$.
\end{corollary}

\begin{proof}
The norms $|\phi|^\star = \frac{1+\delta}{\eps}|\phi|$ and $\|M\|^\star_{z,\lambda} = r^{-1}\|M\|_{z,\lambda}$ verify local coercivity,
so the conclusion follows from
Theorem \ref{thm:a_priori_abstract}.
\end{proof}

In the rest of the section, we prove Theorem \ref{thm:a_priori_abstract}. Our assumptions are designed so that the proof is similar to that of Proposition \ref{prop:classical}. Nonetheless, because this is one of our main results, we give the details for completeness.

\begin{proof}[Proof of Theorem \ref{thm:a_priori_abstract}]
Arguing by contradiction, suppose \eqref{eq:a-priori-interior} does not hold for some $z\in\Sigma$.
Then there exists $g>\nu$ sufficiently close to $\nu$ such that
  \begin{equation}
    C_* \eqdef \sup_{z \in \Sigma} \frac{| \phi (z) |} {\max\{ \disc{M}_{z, \lambda_z / (1-\theta)},
    g^\alpha \dist{z}^{-\alpha} \}
    } > 1
    \label{eq:C-def} \;.
  \end{equation}
  Let $z \in \Sigma$ such that
  \begin{equation}
    | \phi (z) | = \bar C Q\;,\quad Q =  \max\{ \disc{M}_{z,\lambda_z/(1-\theta)}, g^\alpha \dist{z}^{-\alpha}  \}
    \label{eq:z_def_abstract}
  \end{equation}
  where $\bar{C} \in (1,C_*]$ is sufficiently close to $C_*$ such that
  \begin{equation}
  C_*/\bar{C}  < (1-g^{-1})^{\alpha} (1+\delta)  \label{eq:barC} \;.
  \end{equation}
  (Note that the right-hand side of \eqref{eq:barC} is strictly larger than $1$ since $g>\nu\geq (1-\theta)^{-1}$, thus such $z$ and $\bar C$ exist.)
  Then, using \eqref{eq:z_def_abstract}, $\bar C>1$, and  $Q^{- 1 / \alpha} \leq
  g^{-1} \dist{z}$, we have
  \begin{equ}[eq:lambda-bound]
    \lambda_z \leq |\phi(z)|^{-1/\alpha} = (\bar C Q)^{-1/\alpha} < g^{-1} \dist{z}
    \;.
  \end{equ}
  In particular, $|\phi(z)|^{-1/\alpha}< \nu^{-1}\dist{z}$ and therefore $\lambda_z^\alpha = |\phi(z)|^{-1}$.

  Next, denote $\psi = R_{z, \lambda_z} \phi = \lambda_z^{\alpha} \phi
  \circ T_{z, \lambda_z}$.
  Remark that  $\psi\in \bbS_{R_{z, \lambda_z}M}$ by \eqref{eq:sol_rescale}.
  Furthermore
  \begin{equation}
    | \psi  (o_{z, \lambda_z}) | = \lambda_z^{\alpha} | \phi (z) | = 1\;. \label{eq:psi-zero-bound}
  \end{equation}
  Note that, by \eqref{eq:M_norms},
  \begin{equation}
    \| R_{z, \lambda_z} M \|_{z, \lambda_z}
    \leq \lambda_z^\alpha Q
    = \bar C^{-1} < 1
    \;. \label{eq:RM-bound}
  \end{equation}
  We now claim that $\|\psi\|_{\infty} < 1+\delta$,
  which is equivalent to
  \begin{equ}
  \sup_{x \in T_{z, \lambda_z}( \Sigma_{z, \lambda_z})} |\phi(x)| < (1+\delta)|\phi(z)|
  \;.
  \end{equ}
  Suppose for sake of contradiction that there exists $x \in T_{z, \lambda_z}( \Sigma_{z, \lambda_z})$ such that
  \begin{equ}[eq:x_exist]
  |\phi(x)| \geq (1+\delta)|\phi(z)|\;.
  \end{equ}
  By \eqref{eq:lambda-bound},
$
  \lambda_z < g^{-1} \dist{z} < \dist{z}
$.
  Furthermore, $\rho (z, x) \leq \lambda_z$ by \eqref{eq:diameter}.
  Therefore
\begin{equ}[eq:boundary-dist-xy]
  \dist{x} \geq \dist{z}  - \rho (x,
  z) \geq \dist{z} - \lambda_z
  > (1-g^{-1}) \dist{z}
  \;.
\end{equ}
  Then, recalling $\theta = (1+\delta)^{-1/\alpha}$,
  \begin{equ}
  |\phi(x)|^{-1/\alpha} \leq \theta |\phi(z)|^{-1/\alpha} < \theta g^{-1}
  \dist{z} < g^{-1} \dist{x} < \nu^{-1}\dist{x}
  \end{equ}
  where we used \eqref{eq:x_exist} in the first bound, \eqref{eq:lambda-bound} in the second bound, and \eqref{eq:boundary-dist-xy} and the fact that $\theta < 1-g^{-1}$ due to $g>\nu\geq (1-\theta)^{-1}$ in the third bound.
  Hence
  \begin{equ}[eq:lambda_x]
  \lambda_x = |\phi(x)|^{-1/\alpha} \leq (1+\delta)^{-1/\alpha}|\phi(z)|^{-1/\alpha} = \theta \lambda_z
  \end{equ}
  and thus
  \begin{equ}
  \rho(x,z) + \lambda_x/(1-\theta) \leq \lambda_z + \lambda_x/(1-\theta)
  \leq \lambda_z/(1-\theta)
  \end{equ}
  where we used $\rho(x,z)\leq\lambda_z$ in the first bound
  and
  \eqref{eq:lambda_x} in the second bound.
  Therefore, applying \eqref{eq:balls} with $\lambda' =\lambda_x/(1-\theta)$ and $\lambda = \lambda_z/(1-\theta)$, we obtain
  \begin{equ}[eq:M_xy]
  \disc{M}_{x,\lambda_x/(1-\theta)} \leq \disc{M}_{z,\lambda_z/(1-\theta)}\;.
  \end{equ}
  It follows that
  \begin{equs}[eq:phi_x_bound]
  | \phi (x) |
  &\leq
  C_* \max\{ \disc{M}_{x,\lambda_x/(1-\theta)} , g^\alpha \dist{x}^{-\alpha} \}
  \\
  &\leq C_* \max\{ \disc{M}_{z,\lambda_z/(1-\theta)} , g^{\alpha}\dist{z}^{-\alpha}(1-g^{-1})^{-\alpha} \} \leq
     (1-g^{-1})^{-\alpha} C_* Q\;,
    \end{equs}
  where in the first inequality we used the definition of $C_*$ in \eqref{eq:C-def}, and in the second inequality we used \eqref{eq:boundary-dist-xy} and \eqref{eq:M_xy}.
  However, by \eqref{eq:z_def_abstract}-\eqref{eq:barC},
  \begin{equ}
  (1-g^{-1})^{-\alpha} C_* Q = 
  (1-g^{-1})^{-\alpha}C_* |\phi(z)| \bar C^{-1}
  < |\phi(z)|(1 + \delta) \;,
  \end{equ}
  which, together with \eqref{eq:phi_x_bound}, contradicts \eqref{eq:x_exist}. We have thus proved the claim $\|\psi \|_\infty < 1+\delta$.

  Since $\psi \in \bbS_{R_{z,\lambda_z} M}$ with $\|
  \psi \|_{\infty} < 1 + \delta$ and $\| R_{z,\lambda_z} M \|_{z,
  \lambda_z}
  < 1$ by \eqref{eq:RM-bound}, it follows from the
  coercivity assumption \eqref{eq:local_coercivity_abstract} that $| \psi (o_{z,
  \lambda}) | < 1$, which contradicts \eqref{eq:psi-zero-bound}.
\end{proof}

\section{Applications to well-posed equations}
\label{sec:first_applications}

The purpose of this section is to give examples of how Theorem \ref{thm:a_priori_abstract} implies a priori estimates for classically well-posed equations.

\subsection{Classical case revisited}\label{sec:classical-revist}

As a first example, we show how Proposition
\ref{prop:classical} follows from Theorem \ref{thm:a_priori_abstract}.

\textbf{Domains.}
We take our domain as $\Sigma=\cl\Omega$ with boundary $\d \Omega\subset\cl\Omega$ as in Section \ref{sec:PDE_notation}.
We take the metric on $\cl\Omega$ as $\rho(x,y) = |x-y|_\s$, i.e. the parabolic distance, so that
the `distance to the boundary' $\dist{\cdot}\colon \cl\Omega\to [0,1]$ is $\dist{z} = \inf_{y\in\d \Omega} \rho(z,y)$ as in \eqref{eq:parabolic_dist}.

For every $z \in \Omega$ and $0<\lambda \leq \dist{z}$, we take $\Sigma_{z, \lambda} = \Omega$ and $T_{z,\lambda} \colon \Omega \to \cl\Omega$ as defined by \eqref{eq:T_lambda_def}.
(In particular, $o_{z, \lambda} = 0$.)
Note that these choices clearly satisfy \eqref{eq:diameter}.

\textbf{Drivers.}
Our spaces of drivers are $\drivers = \drivers_{z, \lambda} = \CC^{\beta} (\Omega)$ with norm
\begin{equ}
\| \xi \|_{z,\lambda} = \| \xi
   \|^{\rho}_{\CC^{\beta} (\Omega)}
   \end{equ}
where $\rho =
\frac{\alpha}{2+\alpha  + \beta}$ as in \eqref{eq:rho_def_classical}
with $\alpha = 2 / (p - 1)$ as in \eqref{eq:alpha_def}.
The rescaling operators are
\begin{equ}
R_{z, \lambda} \xi = \lambda^{ 2 + \alpha } \xi \circ T_{z, \lambda} 
\;.
\end{equ}
The local norms are
\begin{equ}
\disc{\xi}_{z,\lambda} = \| \xi \|^{\rho}_{\CC^{\beta} ; z ; \lambda}
\end{equ}
for which \eqref{eq:balls} clearly holds and \eqref{eq:M_norms} follows from
\begin{equ}
\|\lambda^{\alpha + 2} \xi  \circ T_{z, \lambda}\|^{\rho}_{\CC^{\beta}(\Omega)}  = \lambda^{(\alpha+2 + \beta)\rho} \|\xi\|_{\CC^{\beta };z;\lambda}^{\rho}
= \lambda^\alpha \|\xi\|_{\CC^{\beta};z;\lambda}^{\rho}\;,
\end{equ}
where we used \eqref{eq:eta-rho-simple} (or \eqref{eq:eta-rho}).

\textbf{Solutions.}
We set $E =\R^m$ as the target space and let $\bbS_\xi$ be the set of functions $\phi\in\CC(\cl\Omega)$ that satisfy \eqref{eq:classical_PDE} in $\Omega$.
The closure under rescaling assumption \eqref{eq:sol_rescale} is verified in \eqref{eq:intertwine-classical}.

\textbf{Local coercivity}, in the form of Corollary \ref{cor:parameter_choice}, follows from
Lemma \ref{lem:coercive_classical}.

We therefore satisfy all the conditions of Corollary \ref{cor:parameter_choice}
and conclude that there exist $C>0$ and $\theta\in(0,1)$ such that, for all $\nu \geq (1-\theta)^{-1}$,
\begin{equ}
| \phi (z) | \leq C \max\{ \disc{\xi}_{z, \lambda_z/(1-\theta)} , \nu^\alpha \dist{z}^{- \alpha} \}\;,
\end{equ}
where $\lambda_z = |\phi(z)|^{-1/\alpha} \wedge (\frac1{\nu}\dist{z})$.
Taking $\nu = 4(1-\theta)^{-1}$ recovers Proposition \ref{prop:classical}.

\subsection{Young ODEs}
\label{subsec:Young}

We consider a Young ODE
\begin{equ}[eq:Young_ODE]
\mrd \phi_t = - \phi_t^p \mrd t + f (\phi_t) \mrd
X_t \;,
\end{equ}
where $p\in (1,\infty)$ and we denote $\phi^p=|\phi|^{p-1}\phi$, $X\in\CC^\gamma([-1,0],\R^n)$, $\phi\in \CC^\gamma([-1,0],\R^m)$, $f\in\CC^\theta(\R^m,L(\R^n,\R^m))$, and $\gamma\in (\frac12,1]$, $\theta\in (\frac{1}{\gamma}-1,1]$.

Note that the condition $f\in\CC^\theta$ implies that $f(\phi)\in\CC^{\theta\gamma}$ where $\gamma+\theta\gamma>1$, thus $\int_s^t f(\phi_u)\mrd X_u$ is well-defined as a Young integral in the sense of \cite[Sec.~4]{FrizHairer20}.
The equation \eqref{eq:Young_ODE} is then understood in the integral sense
\begin{equ}
\phi_t = \phi_s - \int_{s}^t \phi_u^p\mrd u + \int_s^t f(\phi_u)\mrd X_u\;.
\end{equ}
Such equations with $\gamma \in (\frac12,1]$ were studied in \cite{Lyons94}.
In this subsection, we set
\begin{equ}
\alpha = 1/(p-1)
\;.
\end{equ}
For $\beta\in [0,1]$, we define the $\beta$-H\"older semi-norm 
\begin{equ}{}
[\phi]_\beta=\sup_{s\neq t} |t-s|^{-\beta}|\phi_t-\phi_s|
\end{equ}
and likewise for $[f]_\theta$.
We also define
\begin{equ}
\|f\|_{\CC^\theta} = [f]_\theta + \|f\|_\infty
\end{equ} 
and the restricted $\beta$-H\"older semi-norm 
\begin{equ}{}
[\phi]_{\beta;\leq h} = \sup_{0<|t-s|\leq h}|t-s|^{-\beta}|\phi_t-\phi_s|\;.
\end{equ}
We also write $\|\phi\|_{\infty;I}$, $[\phi]_{\beta;I}$, etc. when we restrict the variables $s,t$ to a domain $I$.
For future reference, we record Young's estimate for $\beta,\beta'\in (0,1]$ with $\beta+\beta' > 1$, see e.g. \cite[Sec.~4]{FrizHairer20},
\begin{equ}[eq:Young]
\Big|\int_{s}^t Y_u\mrd X_u - Y_s (X_t-X_s)\Big| \lesssim |t-s|^{\beta+\beta'}[X]_{\beta; [s,t]}[Y]_{\beta'; [s,t]}\;.
\end{equ} 

In what follows, we will derive several a priori estimates on the solution $\phi$ to \eqref{eq:Young_ODE} using Theorem \ref{thm:a_priori_abstract}.
We first place ourselves in the setting of Section \ref{sec:abstract}.

Our \textbf{domains} are just $\Sigma=\Sigma_{z,\lambda}=[-1,0]$ with metric $\rho(x,y) = |x-y|$ and rescaling maps
\begin{equ}
T_{z,\lambda}\colon [-1,0]\to [-1,0]\;,\quad T_{z,\lambda}(x) = z+\lambda x\;.
\end{equ}
Note that $o_{z,\lambda}=0$.

Our spaces of \textbf{drivers} $\drivers=\drivers_{z,\lambda}$ are now all pairs $M=(f,X)$ with $f\in\CC^\theta$ and $X\in\CC^\gamma$ defined on the appropriate spaces.
Define $\psi = \lambda^\alpha \phi\circ T_{z,\lambda}$.
Then
\begin{equ}[eq:intertwine_Young]
\mrd \psi_t = \lambda^{1 + \alpha} (\mrd \phi)_{T_{z,
   \lambda}(t)} = - \lambda^{1 + \alpha - \alpha p} \psi_t^p \mrd t
   + f (\lambda^{-\alpha}\psi) \mrd Y_t
   = -\psi^p_t\mrd t + f_\lambda(\psi) \mrd Y_t\;,
\end{equ}   
where $f_\lambda = f(\lambda^{-\alpha}\cdot)$ and $\mrd Y_t = \lambda^{1 + \alpha} (\mrd X)_{T_{z,\lambda}(t)}$, i.e. $Y_{t} = \lambda^{\alpha} X_{T_{z,\lambda}(t)}$.
This motivates the rescaling map
\begin{equ}
R_{z,\lambda}(f,X) = (f_\lambda, \lambda^\alpha X\circ T_{z,\lambda})\;.
\end{equ}
Note that $[f_\lambda]_{\theta} = \lambda^{-\alpha\theta}[f]_{\theta}$ and thus
\begin{equ}[eq:f_Young]
\|f_\lambda\|_{\CC^\theta} \leq \lambda^{-\alpha\theta}\|f\|_{\CC^\theta}
\end{equ}
and furthermore
\begin{equ}[eq:X_Young]
[X\circ T_{z,\lambda}]_{\gamma;[-1,0]} = \lambda^\gamma[X]_{\gamma;[z-\lambda,z]}\;,
\end{equ}

The \textbf{solution} space $\bbS_M$ for $M=(f,X)$ is the set of all solutions to \eqref{eq:Young_ODE}.
The closure under rescaling assumption \eqref{eq:sol_rescale} is verified in \eqref{eq:intertwine_Young}.

To apply Theorem \ref{thm:a_priori_abstract} (or rather Corollary \ref{cor:parameter_choice}),
it remains to verify the \textbf{local coercivity} condition
for suitable norms $\disc{M}_{z,\lambda}$ and $\|M\|_{z,\lambda}$ satisfying \eqref{eq:balls}-\eqref{eq:M_norms}.
In the rest of the section, we progressively derive sharper and more general bounds by making different choices of norms.

\subsubsection{Simple bound}
\label{sec:Young_nonsharp}

We first derive a simple but non-sharp bound.
Fix $\eps>0$ and suppose $\phi\in\bbS_M$ with
$\|\phi\|_\infty \leq \eps$.
Consider also $\mathring\phi$ solving \eqref{eq:Young_ODE} with $X=0$ and with the same initial condition $\mathring\phi_{-1}=\phi_{-1}$.
Then there exists $\delta>0$ such that $|\mathring\phi_0| < \eps-2\delta$.
Moreover $\phi-\mathring\phi$ solves
\begin{equ}
\mrd (\phi-\mathring\phi)_t = -(\phi^p-\mathring\phi^p)_t\mrd t + f(\phi_t)\mrd X_t\;.
\end{equ}
Therefore, provided that $\eps>0$ is sufficiently small,
$\|\phi-\mathring\phi\|_\infty\lesssim \|\int_{-1}^\cdot f(\phi_t)\mrd X_t\|_\infty$.
In turn,
\begin{equ}
\Big\|
\int_{-1}^\cdot f(\phi_t)\mrd X_t
\Big\|_\infty
\lesssim \|f\|_\infty [X]_\gamma + [f(\phi)]_{\theta\gamma} [X]_\gamma\;.
\end{equ}
Furthermore, $[f(\phi)]_{\theta\gamma}\leq [f]_\theta[\phi]_{\gamma}^\theta$ and, by Young's estimate \eqref{eq:Young},
\begin{equ}
|\phi_t-\phi_s| \lesssim \eps^p |t-s| + |t-s|^\gamma\|f\|_\infty [X]_\gamma + |t-s|^{\gamma+\gamma\theta}[\phi]_{\gamma}^\theta [f]_\theta [X]_\gamma\;.
\end{equ}
Provided $\|f\|_{\CC^\theta}[X]_\gamma$ is small, it follows that
$
[\phi]_\gamma \lesssim 1
$
and thus
\begin{equ}
\|\phi-\mathring\phi\|_\infty \lesssim \|f\|_{\CC^\theta}[X]_\gamma
\;.
\end{equ}
This motivates the norms
\begin{equ}
\disc{(f,X)}_{z,\lambda}
=
\big(
\|f\|_{\CC^\theta}
[X]_{\gamma;[z-\lambda,z]}
\big)^{\rho}
\;,
\qquad
\|(f,X)\|_{z,\lambda}
=
\big(
\|f\|_{\CC^\theta}
[X]_\gamma
\big)^{\rho}
\;,
\end{equ}
where
\begin{equ}
\rho = \frac{\alpha}{\alpha+\gamma - \alpha\theta}\;.
\end{equ}
Then, by \eqref{eq:f_Young}-\eqref{eq:X_Young},
\begin{equ}
\|R_{z,\lambda}(f,X)\|_{z,\lambda}
=
\lambda^{(\alpha+\gamma-\alpha\theta)\rho}
\big(
\|f\|_{\CC^\theta}
[X]_{\gamma;[z-\lambda,z]}
\big)^\rho
\leq \lambda^\alpha\disc{(f,X)}_{z,\lambda}\;,
\end{equ}
which verifies \eqref{eq:M_norms}. Moreover \eqref{eq:balls} clearly holds.

In conclusion,
there exists $r>0$ such that, if $\|(f,X)\|_{z,\lambda}\leq r$, then $\|\phi-\mathring\phi\|_\infty \leq \delta$ and thus $|\phi_0|<\eps-\delta$.
This verifies local coercivity in the form of Corollary \ref{cor:parameter_choice}.
It follows that there exists $C>0$ such that, for all $z\in (-1,0]$,
\begin{equ}[eq:Young_simple]
|\phi_{z}| \leq C \max\{\disc{(f,X)}_{z,\lambda_z}, |z+1|^{-\alpha}\}\;,
\end{equ}
where $\lambda_z= (C|\phi_z|^{-1/\alpha})\wedge \frac{1}{2}|z+1|$.

\subsubsection{Sharper bounds}
\label{sec:sharp_Young}

We next show how, with localised H\"older norms, we can sharpen the bound \eqref{eq:Young_simple}.
Fix again $\eps>0$ and let $\phi\in\bbS_M$ with
$\|\phi\|_\infty \leq \eps$.
Consider $\mathring\phi$ as at the start of Section \ref{sec:Young_nonsharp}.
Since
\begin{equ}[eq:phi_unif]
\|\phi-\mathring\phi\|_\infty\lesssim \Big\|\int_{-1}^\cdot f(\phi_t)\mrd X_t\Big\|_\infty
\;,
\end{equ}
it suffices to bound the right-hand side.

By Young's estimate \eqref{eq:Young}, for all $h\in (0,1]$ and $|t-s|\leq h$ we have
\begin{equ}[eq:Young_est]
\Big|\int_s^t f(\phi_u) \mrd X_u - f(\phi_s)(X_t-X_s)\Big| \lesssim |t-s|^{\gamma+\gamma\theta} [f(\phi)]_{\theta\gamma;\leq h}[X]_{\gamma}
\;.
\end{equ}
By the triangle inequality, for any $t\in [-1,0]$ and $h\in (0,1]$,
\begin{equ}[eq:triangle]
\Big|\int_{-1}^t f(\phi_t)\mrd X_t \Big| \leq
\sum_{0\leq i < (t+1)h^{-1}} \Big|\int_{ih-1}^{t\wedge ((i+1)h-1)} f(\phi_s)\mrd X_s \Big|\;.
\end{equ}
Therefore, applying \eqref{eq:Young_est} to every term on the right-hand side, we obtain
\begin{equ}[eq:Young_int_bound]
\Big\|\int_{-1}^\cdot f(\phi_t)\mrd X_t \Big\|_\infty \lesssim h^{-1}(W_h+V_h)
\;,
\end{equ}
where
\[
W_h = h^\gamma\|f\|_\infty[X]_\gamma
\;,\qquad
V_h = h^{\gamma+\gamma\theta}[f(\phi)]_{\theta\gamma;\leq h}[X]_\gamma
\;.
\]
To bound $V_h$ and $[f(\phi)]_{\theta\gamma;\leq h}$, note that, for $|s-t|\leq h$,
\begin{equ}
|\phi_t-\phi_s| \lesssim \eps^p |t-s| + |t-s|^\gamma h^{-\gamma}W_h + |t-s|^{\gamma+\gamma\theta}h^{-\gamma-\gamma\theta}V_h \;,
\end{equ}
from which it follows that
\begin{equ}{}
[f(\phi)]_{\theta\gamma;\leq h}
\leq [f]_\theta [\phi]^\theta_{\gamma;\leq h}
\lesssim [f]_\theta
(h^{1-\gamma} + h^{-\gamma} W_h + h^{-\gamma}V_h)^\theta
=
[f]_\theta
(h + W_h + V_h)^\theta h^{-\gamma\theta}
\;.
\end{equ}
Therefore
\[
V_h \lesssim h^{\gamma}[X]_\gamma [f]_\theta (h+W_h+V_h)^\theta\;.
\]
Provided $h^\gamma[X]_\gamma[f]_\theta\ll \{h^\gamma[X]_\gamma[f]_\theta(h+W_h)^\theta\}^{1-\theta}$,
i.e.
\[
\Gamma \eqdef (h+W_h)^{\theta-1} h^\gamma[X]_\gamma[f]_\theta
\ll 1
\;,
\]
we obtain by a continuity argument in $h$ that
\begin{equ}
V_h \lesssim h^\gamma[X]_\gamma[f]_\theta(h+W_h)^\theta
=  (h +W_h)\Gamma\;.
\end{equ}
(If $\theta=1$, then we do not need a continuity argument.)

We now choose $h$ so as to make the contributions of $h$ and $W_h$ comparable. Specifically,
suppose that $\|f\|_\infty[X]_\gamma\leq \kappa$ for some small $\kappa>0$.
Choose the smallest $h\in [0,1]$ such that $h^{-1}W_h \leq \kappa$.
For $h=0$, we understand $h^{-1}W_h$ as the limit $h\downarrow0$, which exists by monotonicity. Remark that $h=0$ if and only if $\|f\|_\infty[X]_\gamma=0$ or if $\gamma=1$.

If $\gamma=1$, then $|\int_{-1}^\cdot f(\phi_t)\mrd X_t \|_\infty\lesssim \kappa$ due to  \eqref{eq:Young_int_bound} and the fact that $\lim_{h\downarrow0} h^{-1}V_h=0$.

If $\gamma\in (\frac12,1)$, then
\begin{equ}[eq:Gamma_Young]
h^{-1} V_h \lesssim \Gamma \lesssim h^{\gamma+\theta-1} [f]_\theta[X]_\gamma = (\|f\|_\infty[X]_\gamma)^{\frac{\gamma+\theta-1}{1-\gamma}} [f]_\theta[X]_\gamma 
\leq (\|f\|_\infty^{1-\frac{1-\gamma}{\theta}}\|f\|_{\CC^\theta}^{\frac{1-\gamma}{\theta}}[X]_\gamma)^{\frac{\theta}{1-\gamma}}
\;.
\end{equ}
For all $\gamma\in (\frac12,1]$, this motivates us to (re)define the norms
\begin{equ}
\disc{(f,X)}_{z,\lambda}
=
\big(
\|f\|_\infty^{1-\frac{1-\gamma}{\theta}} \|f\|_{\CC^\theta}^{\frac{1-\gamma}{\theta}}
[X]_{\gamma;[z-\lambda,z]}
\big)^{\rho}
\;,
\qquad
\|(f,X)\|_{z,\lambda}
=
\big(
\|f\|_\infty^{1-\frac{1-\gamma}{\theta}} \|f\|_{\CC^\theta}^{\frac{1-\gamma}{\theta}}
[X]_\gamma
\big)^{\rho}
\;,
\end{equ}
where
\begin{equ}[eq:rho_Young]
\rho = \frac{\alpha}{\gamma(1+\alpha)}\;.
\end{equ}
(Note that we use the \emph{norm} $\|f\|_{\CC^\theta}$ and not the \emph{semi-norm} $[f]_\theta$, so smallness of $\|(f,X)\|_{z,\lambda}$ ensures smallness of $\|f\|_\infty[X]_\gamma$ that we required just above \eqref{eq:Gamma_Young}.)
Then, by \eqref{eq:f_Young}-\eqref{eq:X_Young},
\begin{equs}
\|R_{z,\lambda}(f,X)\|_{z,\lambda}
&=
\lambda^{(\alpha+\gamma)\rho}
\Big(
\|f_\lambda\|_\infty^{1-\frac{1-\gamma}{\theta}}
\|f_\lambda\|_{\CC^\theta}^{\frac{1-\gamma}{\theta}}
[X]_{\gamma;[z-\lambda,z]}
\Big)^\rho
\\
&\leq
\lambda^{(\alpha+\gamma)\rho - \frac{1-\gamma}{\theta}\alpha\theta\rho }
\Big(
\|f\|_\infty^{1-\frac{1-\gamma}{\theta}}
\|f\|_{\CC^\theta}^{\frac{1-\gamma}{\theta}}
[X]_{\gamma; [z-\lambda,z]}
\Big)^\rho
\\
&=\lambda^\alpha\disc{(f,X)}_{z,\lambda}\;,
\end{equs}
which verifies \eqref{eq:M_norms}. Moreover \eqref{eq:balls} clearly holds.

In conclusion, taking $\kappa>0$ above sufficiently small, we can choose $r>0$ such that, if $\|(f,X)\|_{z,\lambda}\leq r$, then, by \eqref{eq:phi_unif} and \eqref{eq:Young_int_bound}, $\|\phi-\mathring\phi\|_\infty \leq \delta$ and thus $|\phi_0|<\eps-\delta$.
This implies the local coercivity condition in the sense of Corollary \ref{cor:parameter_choice}.
It follows that, there exists $C>0$ such that,
for all $z\in (-1,0]$,
\begin{equ}[eq:Young_sharp]
|\phi_{z}| \leq C \max\{\disc{(f,X)}_{z,\lambda_z}, |z+1|^{-\alpha}\}\;,
\end{equ}
where $\lambda_z= (C|\phi_z|^{-1/\alpha})\wedge \frac{1}{2}|z+1|$.
This bound improves \eqref{eq:Young_simple}.

\begin{remark}
\cite[Thm.~1.4, (1.4)]{BCMW22} in the Young regime is a special case of \eqref{eq:Young_sharp}.
\end{remark}


\subsubsection{Example of weighted norm}
\label{sec:Young_weighted}

We now consider a simple example of a weighted norm on $f$.
In Section~\ref{sec:RPs} we will generalise this example (at least for $\theta=1$) by treating more general weights and rough differential equations.

We fix $\eta\in\R$ and, as earlier, $\theta\in (\frac{1}{\gamma}-1,1]$.
For $x\in\R^m$, define
\begin{equ}[eq:x_eta]
|x|_\eta
=
\begin{cases}
|x|^\eta+1 \quad &\text{if } \eta\geq0\;,\\
|x|^\eta \quad &\text{if } \eta<0\;.
\end{cases}
\end{equ}
Note that, for $r\geq1$, $|rx|_\eta \leq r^\eta|x|_\eta$.
We then define the weighted norms
\begin{equs}[eq:weight_norms]
\|f\|_{\infty;\eta} &= \sup_{x\in\R^m} \frac{|f(x)|}{|x|_\eta}\;,
\qquad
[f]_{\theta;\eta} = \sup_{x\neq y,\,\frac12|x|\leq|y|\leq|x|} \frac{|f(x)-f(y)|}{|x-y|^\theta  |x|_{\eta-\theta}}\;,
\\
\|f\|_{\CC^\theta;\eta} &= \|f\|_{\infty;\eta} + [f]_{\theta;\eta}\;.
\end{equs}
Note that, for $\eta<0$ and $\|f\|_{\infty;\eta}<\infty$, one has $|f(x)|\to0$ as $|x|\to\infty$ while $|f(x)|$ may blow up as $x\to0$.
A similar remark applies to $[f]_{\theta;\eta}$ for $\eta<\theta$.

\begin{lemma}\label{lem:f_weighted_scaling}
For all $\eta\in\R$ and $\lambda \in (0,1]$, $\|f_\lambda\|_{\infty;\eta}\leq \lambda^{-\alpha\eta}\|f\|_{\infty;\eta}$
and $[f_\lambda]_{\theta;\eta} \leq \lambda^{-\alpha\eta}[f]_{\theta;\eta}$.
\end{lemma}

\begin{proof}
Since $|\lambda^{-\alpha}x|_\eta \leq \lambda^{-\alpha\eta}|x|_\eta$, one has
\begin{equ}
|f_\lambda(x)| = |f(\lambda^{-\alpha}x)| \leq \lambda^{-\alpha\eta}|x|_\eta\|f\|_{\infty;\eta}\;,
\end{equ}
hence $\|f_\lambda\|_{\infty;\eta}\leq \lambda^{-\alpha\eta}\|f\|_{\infty;\eta}$.
Furthermore, for $\frac12|x| \leq |y| \leq |x|$,
\begin{equ}
|f_\lambda (x) - f_\lambda(y)|
= |f(\lambda^{-\alpha} x) - f(\lambda^{-\alpha} y)|
\leq \lambda^{-\alpha\theta}|x-y|^\theta \lambda^{-\alpha(\eta-\theta)} |x|_{\eta-\theta} [f]_{\theta;\eta}
\;,
\end{equ}
hence $[f_\lambda]_{\theta;\eta} \leq \lambda^{-\alpha\eta}[f]_{\theta;\eta}$.
\end{proof}

Suppose now that $\alpha+\gamma-\alpha\eta>0$.
We then (re)define the norms
\begin{equ}
\disc{(f,X)}_{z,\lambda}
=
\big(
\|f\|_{\infty;\eta}^{1-\frac{1-\gamma}{\theta}} \|f\|_{\CC^\theta;\eta}^{\frac{1-\gamma}{\theta}}
[X]_{\gamma;[z-\lambda,z]}
\big)^{\rho}
\;,
\qquad\|(f,X)\|_{z,\lambda}
=
\big(
\|f\|_{\infty;\eta}^{1-\frac{1-\gamma}{\theta}} \|f\|_{\CC^\theta;\eta}^{\frac{1-\gamma}{\theta}}
[X]_\gamma
\big)^{\rho}\;,
\end{equ}
where
\begin{equ}
\rho = \frac{\alpha}{\alpha+\gamma - \alpha\eta}\;.
\end{equ}
Then, by Lemma \ref{lem:f_weighted_scaling},
\begin{equs}
\|R_{z,\lambda}(f,X)\|_{z,\lambda}
&=
\lambda^{(\alpha+\gamma)\rho}
\big(
\|f_\lambda\|_{\infty;\eta}^{1-\frac{1-\gamma}{\theta}}
\|f_\lambda\|_{\CC^\theta;\eta}^{\frac{1-\gamma}{\theta}}
[X]_{\gamma;[z-\lambda,z]}
\big)^\rho
\\
&\leq
\lambda^{(\alpha+\gamma)\rho - \alpha\eta\rho
}
\big(
\|f\|_{\infty;\eta}^{1-\frac{1-\gamma}{\theta}}
\|f\|_{\CC^\theta;\eta}^{\frac{1-\gamma}{\theta}}
[X]_{\gamma; [z-\lambda,z]}
\big)^\rho
\\
&=\lambda^\alpha\disc{(f,X)}_{z,\lambda}\;,
\end{equs}
which verifies \eqref{eq:M_norms}.

It remains to verify the local coercivity condition in the form of Corollary \ref{cor:parameter_choice}.
As in Section \ref{sec:sharp_Young}, let $\phi$ be a solution with $\|\phi\|_\infty\leq\eps$.
If $|\phi_0|\leq \eps/2$, then we are done.

If $|\phi_t|\in [\frac{\eps}{2},\eps]$ for all $t\in [-1,0]$,
then because the weighted norms of $f$ on the subsets $\{y\,:\,|y|\in [\frac{\eps}{2},\eps]\}$ are, up to a multiple depending on $\eps$, equivalent to the unweighted norms, the argument in Section \ref{sec:sharp_Young} proves that $|\phi_0|<\eps-\delta$ for some $\delta>0$ whenever $\|(f,X)\|_{z,\lambda}\leq r$ for some $r\ll1$, which verifies local coercivity in this case.

Finally, if $|\phi_0|>\eps/2$ and $|\phi_t|<\eps/2$ for some $t\in[-1,0]$, then consider $I=[b,0]$ where $b = \sup\{t\in[-1,0]\,:\,|\phi_t|<\eps/2\}$, so that $I\subset[-1,0]$ is the longest interval containing $0$ for which $|\phi_t|\geq\eps/2$ for all $t\in I$.
Similarly to Section \ref{sec:sharp_Young}, consider the solution $\mathring\phi\colon I\to \R^m$ to \eqref{eq:Young_ODE} with initial condition $\mathring\phi_b = \phi_b$ and $X=0$.
Note that $|\mathring\phi_b|=\eps/2$ thus $|\mathring\phi_0|\leq \frac\eps2$.
Again by equivalence of weighted and unweighted norms of $f$ on $\{y\,:\,|y|\in[\frac\eps2,\eps]\}$,
if $\|(f,X)\|_{z,\lambda}\leq r$ for $r\ll1$,
then $\|\phi-\mathring\phi\|_{\infty;I} < \delta$ for some $\delta\ll\eps$, which finishes the proof of local coercivity.

It follows from Corollary \ref{cor:parameter_choice} that there exists $C>0$ such that,
for all $z\in (-1,0]$,
\begin{equ}[eq:Young_weighted]
|\phi_{z}| \leq C\max\{\disc{(f,X)}_{z,\lambda_z}, |z+1|^{-\alpha}\}\;,
\end{equ}
where $\lambda_z= (C|\phi_z|^{-1/\alpha})\wedge \frac{1}{2}|z+1|$.

\begin{remark}
Recalling that $\alpha=\frac{1}{p-1}$,
we have $\rho = \frac{1}{1+(p-1)\gamma - \eta}$
which matches the exponent in \cite[Thm~1.4, (1.6)]{BCMW22}.
Hence, in the Young regime, \eqref{eq:Young_weighted} generalises the latter by giving precise dependence on $f$, lowering the regularity of $f$ from $\CC^1$ to $\CC^\theta$, and allowing $\eta<1$.
\end{remark}

\begin{remark}\label{rem:sharper}
The norms \eqref{eq:weight_norms} are modelled by the behaviour of polynomials and we chose to consider them for simplicity.
However, the method applies to other norms.
For example, for $\eta\in [0,\theta)$, it is arguably more natural to use the smaller norm $\|f\|_{\CC^\theta;\theta}$ in place of $\|f\|_{\CC^\theta;\eta}$.
Then the same argument as earlier shows
\begin{equ}
|\phi_{z}| \lesssim \max\{\disc{(f,X)}'_{z,\lambda_z}, |z+1|^{-\alpha}\}\;,
\end{equ}
where
\begin{equ}
\disc{(f,X)}'_{z,\lambda} =
\Big(
\|f\|_{\infty;\eta}^{1-\frac{1-\gamma}{\theta}} \|f\|_{\CC^\theta;\theta}^{\frac{1-\gamma}{\theta}}
[X]_{\gamma;[z-\lambda,z]}
\Big)^{\rho}
\end{equ}
and
\begin{equ}
\rho = \frac{\alpha}{\gamma + \alpha(\gamma - (1-\frac{1-\gamma}{\theta})\eta)}\;,
\end{equ}
which sharpens \eqref{eq:rho_Young} in the case that $\eta=0$.
\end{remark}

\section{Rough paths}
\label{sec:RPs}

We now generalise Section \ref{subsec:Young} by considering a \emph{rough differential equation} (RDE)
\begin{equ}[eq:RDE]
\mrd \phi_t = - \phi_t^{p} \mrd t + f (\phi_t) \mrd
   \mbX_t \;,
  \end{equ}
  where $\phi\in\CC^\gamma([-1,0],\R^m)$, $p\in (1,\infty)$ and we write $\phi^p = |\phi|^{p-1}\phi$, $\mbX$ is an $n$-dimensional $\gamma$-H\"older branched rough path for $\gamma\in (0,1]$ in the sense of \cite{Gubinelli_10_BRPs} (see also \cite{Hairer_Kelly_15_BRPs,BC19, BCFP19,FrizHairer20,BCMW22}),
and $f\in \CC^N(\R^m,L(\R^n,\R^m))$ where $N=\floor{\frac{1}{\gamma}}$.
(In this regime, one in general has existence but not uniqueness of solutions to \eqref{eq:RDE}.)

Our main result in this section is Theorem \ref{thm:RP_bound}, which establishes a priori bounds on $\phi$ under very general (multi-linear) conditions on $f$.
In Section \ref{sec:linear_condition_RP}, we show a consequence of Theorem \ref{thm:RP_bound} under a simple (linear) condition on $f$, which strengthens the main result of \cite{BCMW22}.

\begin{remark}\label{rem:drift}
It is not crucial for our argument that the drift is polynomial: by including the drift suitably into the space of `drivers', we can handle any sufficiently regular function that behaves like $-\phi^p$ for large $|\phi|$.
\end{remark}

\begin{remark}
We restrict to $f\in\CC^N$ for simplicity,
but with more careful arguments, similar to those in Section \ref{subsec:Young}, one can handle the general case $f\in\CC^\theta$ where $\theta \in (\frac{1}{\gamma}-1,N]$.
\end{remark}

\subsection{General a priori estimates}
\label{subsec:general_RPs}

Recall that $\mbX$ is indexed by (labelled rooted) forests $\tau$ with number of vertices $|\tau|\leq N$ and with every vertex labelled by an element of the set $[n]$,
and $\mbX^\tau\colon [-1,0]^2\to\R$ is continuous with
\begin{equ}
\|\mbX^\tau\| \eqdef \|\mbX^\tau\|_{\gamma|\tau|}
\eqdef \sup_{s\neq t} \frac{|\mbX^\tau_{s,t}|}{|t-s|^{\gamma|\tau|}}
< \infty\;.
\end{equ}
We denote by $\bone$ the empty forest for which $\mbX^\bone\equiv 1$.
For an interval $I\subset[-1,0]$, we write $\|\mbX^\tau\|_I$ for the corresponding semi-norm obtained by restricting to $s,t\in I$.

Recall that every forest can be written uniquely as $\tau = [\sigma_1]_{i_1}\cdots[\sigma_\ell]_{i_\ell}$, where $i_j\in [n]$, $\sigma_j$ is a forest (possibly empty), $[\sigma]_i$ denotes the tree obtained by attaching a root with label $i$ to $\sigma$, and where the product $[\sigma_1]_{i_1}\cdots[\sigma_\ell]_{i_\ell}$ is unordered (i.e. we consider combinatorial trees).
Remark that, $|[\tau]_i| = |\tau|+1$.

Recall the Butcher elementary differentials $f^\tau\in \CC(\R^m,L(\R^n,\R^m))$ of $f$, indexed by labelled rooted forests $\tau$ with  $0\leq |\tau|\leq N-1$,
which are given inductively by $f^\bone = f$ and then for $\tau = [\sigma_1]_{i_1}\cdots[\sigma_\ell]_{i_\ell}$ by
\begin{equ}[eq:Butcher_def]
f^\tau = (D^\ell f) (f^{\sigma_1}_{i_1},\ldots, f^{\sigma_\ell}_{i_\ell})\;,
\end{equ}
where $f^{\sigma}_i\colon \R^m\to \R^m$ is the $i$-th component of $f^\sigma$ for $i\in[n]$, and
\begin{equ}
D^\ell f \colon \R^m \to L((\R^m)^{\otimes\ell},L(\R^n,\R^m))
\end{equ}
is the $\ell$-th Fr\'echet derivative of $f$.
Remark that $f^\tau$ is a $(|\tau|+1)$-linear function of $f$ involving, in total, $|\tau|$ derivatives of $f$.

We use the notion of solution to \eqref{eq:RDE} from \cite[Prop.~3.8]{Hairer_Kelly_15_BRPs}
(which is the branched analogue of Davie's solutions to geometric RDEs \cite{Davie_08,Friz_Victoir_08_Euler,FV10}).
We say that $\phi$ solves \eqref{eq:RDE} if
\begin{equ}[eq:RDE_Euler]
\phi_t-\phi_s = - \phi_s^{p}(t-s) + \sum_{0\leq |\tau|\leq N-1} f^{\tau} (\phi_s) \mbX^{[\tau]}_{s,t} + o(|t-s|)\;,
\end{equ}
where the sum is over all labelled rooted forests $\tau$ (possibly empty) with number of vertices $0\leq |\tau|\leq N-1$,
and where we treat $\mbX^{[\tau]}$ as $\R^n$-valued with $i$-th component given by $\mbX^{[\tau]_i}$.

We make the following assumption for the rest of this section.

\begin{assumption}\label{as:eta}
For all $0\leq |\tau| < N$ and $0\leq \ell \leq N-|\tau|$,
there exists $\eta(\tau,\ell)\in\R$ such that
\begin{equ}
\|D^\ell f^\tau \|_{\infty;\eta(\tau,\ell)} <\infty
\;,
\end{equ}
where we recall $\|\cdot\|_{\infty;\eta}$ from \eqref{eq:weight_norms}.
\end{assumption}

As in Section \ref{subsec:Young}, we denote in this subsection
\begin{equ}
\alpha = 1/(p-1)\;.
\end{equ}
Define the local norms, for $z\in(-1,0]$ and $0<\lambda\leq |z+1|$,
\begin{equ}[eq:weight_1]
\disc{(f,\mbX)}^{(1)}_{z,\lambda}
=
\max_{0\leq |\tau|\leq N-1} H_\tau^{\rho_{\tau}}\;,\quad
\end{equ}
where
\begin{equ}[eq:H_def]
H_\tau = \|f^\tau\|_{\infty;\eta(\tau,0)}\|\mbX^{[\tau]}\|_{[z-\lambda,z]}
\;,\qquad
\rho_\tau = \alpha/\Delta(\tau,0)
\end{equ}
and, for $0\leq \ell\leq N-|\tau|$,
\begin{equ}[eq:Delta_def]
\Delta(\tau,\ell)
= \gamma(|\tau|+1) + \alpha(1-\ell-\eta(\tau,\ell))\;.
\end{equ}
Define moreover
\begin{equ}[eq:weight_2]
\disc{(f,\mbX)}^{(2)}_{z,\lambda}
=
\max_{0\leq|\tau|\leq N-1}
\max_{1\leq \ell\leq N-|\tau|}
\max_{0\leq |\sigma|\leq N-1}
\Big(
\|D^\ell f^{\tau}\|_{\infty;\eta(\tau,\ell)}
\|\mbX^{[\tau]}\|_{[z-\lambda,z]}
H_\sigma^{\zeta_{\tau,\ell,\sigma}}
\Big)^{\rho_{\tau,\ell,\sigma}}
\end{equ}
where $H_\sigma$ is defined in \eqref{eq:H_def} and
\begin{equ}
\rho_{\tau,\ell,\sigma} = \frac{\alpha}{\Delta(\tau,\ell) + \zeta_{\tau,\ell,\sigma}\Delta(\sigma,0)}
\;,
\qquad
\zeta_{\tau,\ell,\sigma} = \frac{(|\tau|+1)\gamma + \ell-1}{1-(|\sigma|+1)\gamma}\;,
\end{equ}
and where, if $\gamma=1/N$, then the case $|\sigma|=N-1$ is excluded from the $\max$ in \eqref{eq:weight_2}
(so, if $\gamma=N=1$, then simply $\disc{(f,\mbX)}^{(2)}_{z,\lambda}=0$).
Remark that $\zeta_{\tau,\ell,\sigma}>0$.

We now make the following assumption for the rest of the section.
\begin{assumption}\label{as:Delta}
$\rho_\tau$ and $\rho_{\tau,\ell,\sigma}$ are well-defined and positive. Namely,
$\Delta(\tau,0)>0$ and $\Delta(\tau,\ell) + \zeta_{\tau,\ell,\sigma}\Delta(\sigma,0)>0$ for all $0\leq|\tau|\leq N-1$, $0\leq \ell \leq N-|\tau|$, and $0\leq|\sigma|\leq N-1$
(if $\gamma = 1/N$, then we again consider only $0\leq|\sigma|\leq N-2$).
\end{assumption}

\begin{theorem}\label{thm:RP_bound}
There exist $C>0$, depending only on $m,n,\gamma,p$ and the collection of exponents $\eta(\tau,\ell)$, such that, if
$\phi\colon [-1,0]\to\R^m$ solves \eqref{eq:RDE}, then, for all $z\in (-1,0]$,
\begin{equ}[eq:RP_bound]
|\phi_z| \leq C\left( \disc{(f,\mbX)}^{(1)}_{z,\lambda_z}+\disc{(f,\mbX)}^{(2)}_{z,\lambda_z} + |z+1|^{-\alpha} \right)
\end{equ}
where $\lambda_z = (C|\phi_z|^{-1/\alpha}) \wedge (\frac{1}{2}|z+1|)$.
\end{theorem}

\begin{remark}
If $f$ is constant, then $D^\ell f^\tau = 0$ unless $\ell=0$ and $\tau=\bone$.
Therefore, the right-hand side of \eqref{eq:RP_bound} depends only on the level-$1$ path $\mbX^{[\bone]}$, as expected from the fact that the higher level components of $\mbX$ do not affect the solution.
\end{remark}

\subsection{Simple linear condition on forcing}
\label{sec:linear_condition_RP}

Before the proof, we derive a consequence of Theorem \ref{thm:RP_bound} under a \emph{linear} condition on $f$.
Suppose there exist $\eta\in \R$ and $q\geq -1$ such that
\begin{equ}
\|f\|_{\eta,q} \eqdef \max_{0\leq \ell\leq N} \sup_{x\in\R^m} \frac{|D^\ell f(x)|}{(|x|+1)^{\eta+\ell q}}
<\infty
\;.
\end{equ}
A basic example is $f(x) = P(x)\sin(Q(x))$ where $P$ and $Q$ are polynomials of degree $\eta$ and $q+1$ respectively.

\begin{remark}
The restriction to $q\geq -1$ is natural because, when $q<-1$, we have the equivalence of norms $\|f\|_{\eta,q}\asymp \|f\|_{\eta+Nq+N,-1}$, and therefore we can always reduce to the case $q\geq -1$ modulo changing $\eta$.
\end{remark}

For $I\subset[-1,0]$, define the homogenous rough path H\"older norm,
\begin{equ}
\trinorm{\mbX}_I = \max_{1\leq |\tau|\leq N}\|\mbX^\tau\|^{1/|\tau|}_I\;.
\end{equ}

\begin{corollary}
Suppose that
\begin{equ}[eq:RP_linear_assump]
\gamma/\alpha + \gamma - q(1-\gamma) - \eta > 0
\;.
\end{equ}
Then, with notation as in Theorem \ref{thm:RP_bound}, uniformly over all solutions $\phi\colon [-1,0]\to\R^m$ to \eqref{eq:RDE},
\begin{equ}[eq:final_RP_bound]
|\phi_z| \lesssim
\big(\|f\|_{\eta,q}\trinorm{\mbX}_{[z-\lambda_z,z]}\big)^{\frac{1}{\gamma +\gamma/\alpha - q(1-\gamma)- \eta}} + |z+1|^{-\alpha}\;.
\end{equ}
\end{corollary}

\begin{remark}
The two cases considered in \cite{BCMW22} (see Assumptions~1.1-1.2 therein) are (i) $q=\eta=0$, i.e. $f,Df,\ldots, D^Nf$ are bounded, and (ii) $q=-1$ and $\eta \geq 1$,
i.e. $f$ behaves like a polynomial of degree $\eta$
(see \cite[Rem.~1.9]{BCMW22} on the limitation of $\eta\geq 1$).
Our condition \eqref{eq:RP_linear_assump} and bound \eqref{eq:final_RP_bound} agree with \cite[Thm.~1.4]{BCMW22} in these two cases, but we make the improvement that
we allow any $q, \eta\in\R$ and localise the rough path H\"older norm $\trinorm{\mbX}$.
\end{remark}

\begin{remark}
If $\gamma=1$, then the above bound does not depend on $q$, which is consistent with the Young case \eqref{eq:Young_weighted} where we have no dependence on the H\"older semi-norm $[f]_\theta$.
\end{remark}

\begin{proof}
A simple induction shows that
\begin{equ}[eq:D_ell_f_tau]
\|D^\ell f^\tau\|_{\infty;\eta(\tau,\ell)} \lesssim \|f\|_{\eta,q}^{|\tau|+1}
\end{equ}
where $\eta(\tau,\ell) = (|\tau|+\ell) q + (|\tau|+1)\eta$.
Hence Assumption \ref{as:eta} is satisfied with this choice of $\eta(\tau,\ell)$
and we have
\begin{equ}
\Delta(\tau,\ell) = \gamma(|\tau|+1) + \alpha(1-\ell-(|\tau|+\ell) q - (|\tau|+1)\eta)\;.
\end{equ}
Then, by \eqref{eq:D_ell_f_tau},
\begin{equ}[eq:1_simple]
\disc{(f,\mbX)}^{(1)}_{z,\lambda} \lesssim \max_{0\leq |\tau| \leq N-1}
\big(
\|f\|_{\eta,q}\trinorm{\mbX}_{[z-\lambda,z]}
\big)^{\frac{\alpha(|\tau|+1)}{\Delta(\tau,0)}}\;.
\end{equ}
Remark that, denoting $k=|\tau|$ with $0\leq k\leq N-1$,
\begin{equ}[eq:exp_1_bound]
\frac{\Delta(\tau,0)}{\alpha(k+1)} = \frac{\gamma}{\alpha} + \frac{1- k q}{k+1}
- \eta
\geq \frac{\gamma}{\alpha} +  \frac{1}{N} - q + \frac{q}{N} - \eta
\geq  \frac{\gamma}{\alpha} + \gamma - q(1-\gamma) - \eta\;,
\end{equ}
where we used, in the first equality, that $q\geq -1$ and thus $\frac{1- k q}{k+1}$ is increasing in $k$ and, in the second inequality, that $\gamma\leq 1/N$.

Next, denoting $k=|\tau|,j=|\sigma|$ and recalling $\rho_{\tau,\ell,\sigma}$ from \eqref{eq:weight_2}, by a quick calculation,
\begin{equ}[eq:2_exp_simple]
\rho_{\tau,\ell,\sigma} = \frac{1}{\gamma(1+1/\alpha)(k+1+\zeta(j+1)) - \eta(\tau,\ell)-\zeta \eta(\sigma,0)}
\;,
\end{equ}
where we simply denote $\zeta \equiv \zeta_{\tau,\ell,\sigma}$,
and
\begin{equ}
\eta(\tau,\ell)+\zeta \eta(\sigma,0) = q(k+\ell+\zeta j) + \eta(k+1+\zeta(j+1)) = ( q(1-\gamma)+ \eta)(k+1+\zeta(j+1))\;.
\end{equ}
Moreover, due to \eqref{eq:D_ell_f_tau},
\begin{equ}[eq:2_simple]
\disc{(f,\mbX)}^{(2)}_{z,\lambda}
\lesssim \max_{\tau,\ell,\sigma} 
\big(\|f\|_{\eta,q}\trinorm{\mbX}_{[z-\lambda,z]}
\big)^{(k+1+\zeta(j+1))\rho_{\tau,\ell,\sigma}}
\;,
\end{equ}
where the exponent, by \eqref{eq:2_exp_simple}, is given by
\begin{equ}[eq:exp_2]
(k+1+\zeta(j+1))\rho_{\tau,\ell,\sigma} = \frac{1}{\gamma + \gamma/\alpha - q (1-\gamma)- \eta}\;.
\end{equ}
In conclusion, Assumption \ref{as:Delta} is equivalent to \eqref{eq:RP_linear_assump},
and under this assumption, by \eqref{eq:1_simple}-\eqref{eq:exp_1_bound} and \eqref{eq:2_simple}-\eqref{eq:exp_2},
\begin{equ}
\disc{(f,\mbX)}^{(1)}_{z,\lambda}
+ \disc{(f,\mbX)}^{(2)}_{z,\lambda}
\lesssim
1+\big(\|f\|_{\eta,q}\trinorm{\mbX}_{[z-\lambda,z]}
\big)^{\frac{1}{\gamma +\gamma/\alpha - q(1-\gamma)- \eta}}\;.
\end{equ}
Then, by Theorem \ref{thm:RP_bound}, we obtain \eqref{eq:final_RP_bound}.
\end{proof}

\subsection{Proof of Theorem \ref{thm:RP_bound}}
\label{sec:RP_proof}

We now turn to the proof of Theorem \ref{thm:RP_bound}.
The setting is similar to Section \ref{subsec:Young} except that the driver space $\drivers=\drivers_{z,\lambda}$ is now all pairs $M=(f,\mbX)$ where $f\in\CC^N$ and $\mbX\in \sC^\gamma$ for $\gamma\in (0,1]$,
where $\sC^\gamma$ is the space of all $\gamma$-H\"older branched rough paths.

The rescaling map is defined by
\begin{equ}
R_{z,\lambda} (f,\mbX) = (f_\lambda, \delta_{\lambda^\alpha} \mbX \circ T_{z,\lambda})
\end{equ}
where now ${\delta_r\mbX}^\tau = r^{|\tau|} \mbX^\tau$ and, as before, $f_\lambda = f(\lambda^{-\alpha}\cdot)$.

The solution space $\bbS_M$ is similarly the set of all solutions to \eqref{eq:RDE} in the sense of \eqref{eq:RDE_Euler}.
The closure under rescaling $\psi = \lambda^\alpha \phi\circ T_{z,\lambda}$ assumption \eqref{eq:sol_rescale} follows from \eqref{eq:RDE_Euler} once we note that
\begin{equ}
(f_\lambda)^{\tau} = \lambda^{-\alpha |\tau|}f^{\tau}(\lambda^{-\alpha}\cdot)
\end{equ}
since $|\tau|$ is the total number of derivatives in $f^\tau$,
and thus
\begin{equs}
\psi_t-\psi_s
&= -\lambda^{\alpha - p\alpha} \psi_{s}^{p}\lambda(t-s)+
\lambda^{\alpha}\sum_{|\tau|\leq N-1} f^{\tau}(\lambda^{-\alpha}\psi_s) \mbX^{[\tau]}_{T_{z,\lambda} (s), T_{z,\lambda} (t)} + o(|t-s|)
\\
&=
-\psi_{s}^{p} (t-s)
+
\sum_{|\tau|\leq N-1} (f_\lambda)^{\tau}(\psi_s) (\delta_{\lambda^\alpha}\mbX)^{[\tau]}_{T_{z,\lambda} (s), T_{z,\lambda} (t)}
+ o(|t-s|)\;,
\end{equs}
where we used the definition $\alpha=1/(p-1)$ in the second equality.

It remains to define suitable norms on $\drivers,\drivers_{z,\lambda}$ verifying \eqref{eq:balls}-\eqref{eq:M_norms} and the local coercivity condition \eqref{eq:local_coercivity_abstract}.
To this end, we have already introduced the local norms $\disc{(f,\mbX)}^{(1)}_{z,\lambda}$ and $\disc{(f,\mbX)}^{(2)}_{z,\lambda}$ in \eqref{eq:weight_1} and \eqref{eq:weight_2}.
We then define $\disc{(f,\mbX)}_{z,\lambda} = \disc{(f,\mbX)}^{(1)}_{z,\lambda}+\disc{(f,\mbX)}^{(2)}_{z,\lambda}$
and remark that \eqref{eq:balls} clearly holds for this norm.

Furthermore, we define the norm
\begin{equ}
\|(f,\mbX)\| =\|(f,\mbX)\|_{z,\lambda} = \|(f,\mbX)\|^{(1)}+\|(f,\mbX)\|^{(2)}\;,
\end{equ}
where $\|(f,\mbX)\|^{(i)}$ is defined exactly as $\disc{(f,\mbX)}^{(i)}_{z,\lambda}$ but with the H\"older norm $\|\mbX^\tau\|$ taken over the whole interval $[-1,0]$
(in particular, $\|(f,\mbX)\|^{(i)}$ does not depend on $z,\lambda$).

We next verify \eqref{eq:M_norms}.
Observe that
\begin{equ}
D^\ell (f_\lambda)^\tau = \lambda^{-\alpha(|\tau|+\ell)} (D^\ell f^\tau)_\lambda
\end{equ}
and therefore, by Lemma \ref{lem:f_weighted_scaling},
\begin{equ}[eq:D_ell_infty]
\|D^\ell (f_\lambda)^\tau \|_{\infty;\eta(\tau,\ell)} \leq \lambda^{-\alpha(|\tau|+\ell + \eta(\tau,\ell))}\|D^\ell f^\tau\|_{\infty;\eta(\tau,\ell)}\;.
\end{equ}
Furthermore,
\begin{equ}
\|\delta_{\lambda^\alpha}\mbX^{\tau}\circ T_{z,\lambda}\|_{|\tau|\gamma} = \lambda^{|\tau|(\alpha+\gamma)}\|\mbX^{\tau}\|_{[z-\lambda,z]}\;.
\end{equ}
As a consequence
\begin{equ}
\|D^\ell (f_\lambda)^\tau \|_{\infty;\eta(\tau,\ell)} \|\delta_{\lambda^\alpha}\mbX^{[\tau]}\circ T_{z,\lambda}\|_{(|\tau|+1)\gamma}
\leq
\lambda^{-\alpha (|\tau|+\ell + \eta(\tau,\ell)) + (\alpha+\gamma)(|\tau|+1)} \|D^\ell f^\tau\|_{\infty;\eta(\tau,\ell)} \|\mbX^{[\tau]}\|_{[z-\lambda,z]}\;.
\end{equ}
Remark that the exponent of $\lambda$ is $\Delta(\tau,\ell) = \gamma(|\tau|+1) + \alpha(1-\ell-\eta(\tau,\ell))$
from \eqref{eq:Delta_def}.
We thus obtain, for $i=1,2$,
\begin{equ}
\|R_{z,\lambda}(f,\mbX)\|^{(i)}_{z,\lambda} \leq \lambda^{\alpha} \disc{(f,\mbX)}^{(i)}_{z,\lambda}
\end{equ}
from which \eqref{eq:M_norms} follows.

To conclude the proof of Theorem \ref{thm:RP_bound}, it remains to verify the form of local coercivity stated in Corollary \ref{cor:parameter_choice}.
Our argument is self-contained with the exception of the following algebraic and sewing-type lemma.
Fix for the rest of the section a solution $\phi$ to \eqref{eq:RDE}.

\begin{lemma}\label{lem:sewing}
For every $0\leq|\tau|\leq N-1$, define the `remainders' $R^\tau_{s,t} \in L(\R^n,\R^m)$ by
\begin{equ}[eq:R_def]
f^\tau(\phi_t) = 
\sum_{\ell=0}^{N-|\tau|-1} \sum_{(\sigma_1,\ldots,\sigma_\ell)} \frac{(D^\ell f^\tau)}{\ell!} (\phi_s) (f^{\sigma_1}(\phi_s) \mbX_{s,t}^{[\sigma_1]},\ldots, f^{\sigma_\ell}(\phi_s) \mbX_{s,t}^{[\sigma_\ell]})
+ R_{s,t}^\tau
\end{equ}
where the inner sum is over all ordered collections of (possibly empty) forests $(\sigma_1,\ldots,\sigma_\ell)$ such that $\sum_{i=1}^\ell(|\sigma_i|+1)<N-|\tau|$,
and where we recall the notation in \eqref{eq:Butcher_def}-\eqref{eq:RDE_Euler}.
Then, for all $h\in (0,1]$,
\begin{equ}[eq:R_reg]
\|R^\tau\|_{\gamma (N-|\tau|);\leq h} = \sup_{0<|t-s|\leq h} \frac{|R_{s,t}^\tau|}{|t-s|^{\gamma (N-|\tau|)}} < \infty\;.
\end{equ}
Moreover, for $-1\leq s\leq t\leq 0$, define the rough integral
\begin{equ}[eq:phi_integral]
\int_{s}^t f(\phi_u)\mrd \mbX_u =
\phi_t-\phi_s + \int_s^t \phi_u^{p}\mrd u
\;.
\end{equ}
Then we have the `sewing' estimate for $|t-s|\leq h$
\begin{equ}[eq:sewing_est]
\Big|\int_s^t f(\phi_u) \mrd \mbX_u - \sum_{0\leq|\tau|\leq N-1} f^{\tau} (\phi_s) \mbX^{[\tau]}_{s,t} \Big| \lesssim
|t-s|^{(N+1)\gamma}
\max_{0\leq|\tau|\leq N-1} \|R^\tau\|_{\gamma(N-|\tau|);\leq h} \|\mbX^{[\tau]}\|\;.
\end{equ}
\end{lemma}

\begin{proof}
The statement is equivalent to saying that $\frac{1}{\ell!}(D^\ell f^\tau) (\phi_s) (f^{\sigma_1}_{i_1}(\phi_s),\ldots, f^{\sigma_\ell}_{i_\ell}(\phi_s))$ are the derivatives with respect to $\mbX$ of the controlled rough path $f(\phi)$ and $R^\tau_{s,t}$ are the corresponding remainders.
For a proof that $\phi$ solves \eqref{eq:RDE} in the sense of \eqref{eq:RDE_Euler} if and only if $\phi$ is a controlled rough path with derivatives $f^\tau(\phi)$, see \cite[Prop.~3.8]{Hairer_Kelly_15_BRPs}.
Then, for a proof that $f(\phi)$ has the claimed derivatives and remainders, see \cite[Cor.~3.9]{BCMW22} or the more general statement in the setting of regularity structures in Section \ref{sec:elem_diff}.
\end{proof}

\begin{remark}
Since $\phi$ solves \eqref{eq:RDE} in the sense of \eqref{eq:RDE_Euler}, we have
\begin{equ}
\Big|\int_s^t f(\phi_u)\mrd \mbX_u - \sum_{0\leq |\tau|\leq N-1} f^{\tau} (\phi_s) \mbX^{[\tau]}_{s,t} \Big|
= o(|t-s|)
\end{equ}
from which we see that $\int_s^t f(\phi_u)\mrd \mbX_u$ is the usual rough integral in the sense of \cite{MaxControl,Feyel_DLP_06_sewing,Gubinelli_10_BRPs, FrizHairer20}.
The bound \eqref{eq:sewing_est} for $N=1$ is Young's estimate \eqref{eq:Young_est} (for $\theta=1$ therein).
See, for instance, \cite{Gubinelli_10_BRPs,Hairer_Kelly_15_BRPs, BCFP19,FrizHairer20,BCMW22, Chevyrev_25_RP}
for more background on RDEs.
\end{remark}

\begin{remark}\label{rem:zeta}
We do not need to be very careful with the regularity of $R^\tau$ -- our argument works and yields the same a priori bound if we replace $\gamma(N-|\tau|)$ in \eqref{eq:R_reg} by $\zeta - (|\tau|+1)\gamma + 1$ for any $\zeta \in (0,\gamma(N+1)-1]$
and correspondingly replace $(N+1)\gamma$ in \eqref{eq:sewing_est} by $1+\zeta$.
\end{remark}

Let us now fix $\eps>0$ and suppose
$\|\phi\|_\infty \leq \eps$.
By the same argument as above \eqref{eq:Young_weighted}, it suffices to consider the case that $|\phi_t|\in[\frac{\eps}{2},\eps]$
for all $t\in [-1,0]$.
It suffices therefore to consider the case that $\|D^\ell f\|_\infty<\infty$ for all $0\leq \ell\leq N$
and $\eta(\tau,\ell)=0$ in Assumption \ref{as:eta}.

Consider $\mathring\phi\colon [-1,0]\to\R^m$ solving \eqref{eq:Young_ODE} with $\mbX^\tau \equiv 0$ for all non-empty forests $|\tau|>0$ and with the same initial condition $\mathring\phi_{-1}=\phi_{-1}$.
Then $\|\mathring\phi\|_\infty\leq \eps$ and there exists $\delta>0$ such that $|\mathring\phi_0| < \eps-2\delta$.
Moreover
\begin{equ}
\phi_t-\mathring\phi_t = \phi_s-\mathring\phi_s - \int_s^t (\phi_u^{p} - \mathring\phi_u^{p})\mrd u
+ \int_s^t f(\phi_t)\mrd \mbX_t\;.
\end{equ}
Since $p>1$ and $\|\phi\|_\infty,\|\mathring\phi\|_\infty\leq \eps$,
we have $|\phi_u^{p} - \mathring\phi_u^{p}| \lesssim \eps^{p-1}|\phi_u-\mathring\phi_u|$
and therefore, provided that $\eps>0$ is sufficiently small,
\begin{equ}[eq:phi_diff_int]
\|\phi-\mathring\phi\|_\infty
\lesssim
\Big\|\int_{-1}^\cdot f(\phi_t)\mrd \mbX_t
\Big\|_\infty\;.
\end{equ}
To show local coercivity \eqref{eq:local_coercivity_abstract}, it thus suffices to show that $\|\int_{-1}^\cdot f(\phi_t)\mrd \mbX_t\| \to 0$ as $\|(f,\mbX)\|\to 0$.

To this end, for all $h\in [0,1]$ by \eqref{eq:sewing_est} and the triangle inequality applied to increments over a mesh of size $h$ as in \eqref{eq:triangle}, we obtain
\begin{equ}[eq:scale_rough]
\Big\|
\int_{-1}^\cdot f(\phi_t)\mrd \mbX_t
\Big\|_\infty
\lesssim \frac{1}{h}\Big(
W_h
+
V_h
\Big)
\end{equ}
where
\begin{equs}
W_h &= \max_{0\leq|\tau|\leq N-1} h^{\gamma (|\tau|+1)}\|f^{\tau}\|_\infty \|\mbX^{[\tau]}\|\;,
\\
V_h &= h^{\gamma(N+1)} \max_{0\leq|\tau|\leq N-1} \|R^\tau\|_{\gamma(N-|\tau|);\leq h} \|\mbX^{[\tau]}\|\;,
\end{equs}
and $R^\tau$ is defined by \eqref{eq:R_def} in Lemma \ref{lem:sewing}.
(For $h=0$, we interpret \eqref{eq:scale_rough} as the limit $h\downarrow 0$.)

Note that if $\max_{0\leq|\tau|\leq N-1} \|f^{\tau}\|_\infty \|\mbX^{[\tau]}\|=0$, then $\phi=\mathring\phi$ by \eqref{eq:RDE_Euler} and we are done.
To exclude this trivial case, we suppose henceforth that
\begin{equ}[eq:non_trivial]
\max_{0\leq|\tau|\leq N-1} \|f^{\tau}\|_\infty \|\mbX^{[\tau]}\| > 0\;.
\end{equ}

We next estimate $\|R^\tau\|_{\gamma(N-|\tau|);\leq h}$ for some fixed $0\leq|\tau|\leq N-1$.
By Taylor's theorem,
\begin{equ}[eq:f_expan]
f^\tau(\phi_t) = \sum_{\ell=0}^{N-|\tau|-1} \frac{(D^\ell f^\tau)}{\ell!} (\phi_s) (\phi_t - \phi_s)^\ell
+ O(\|D^{N-|\tau|}f^\tau\|_\infty |\phi_t-\phi_s|^{N-|\tau|})\;.
\end{equ}
Define the remainder
\begin{equ}
Q_{s,t} = \phi_t - \phi_s - \sum_{0\leq |\sigma|\leq N-1} f^{\sigma}(\phi_s) \mbX_{s,t}^{[\sigma]}\;.
\end{equ}
Then, by \eqref{eq:phi_integral}-\eqref{eq:sewing_est}, for $|t-s|\leq h$,
\begin{equs}[eq:Q_bound]
|Q_{s,t}| &\lesssim |t-s| + |t-s|^{(N+1)\gamma} h^{-(N+1)\gamma}V_h
\\
&\leq
|t-s|^{(N-|\tau|)\gamma}h^{-(N-|\tau|)\gamma}(h+V_h)
\;.
\end{equs}
Note also that
\begin{equ}[eq:f_X_bound]
|f^{\sigma}(\phi_s)\mbX_{s,t}^{[\sigma]}| \leq |t-s|^{(|\sigma|+1)\gamma}h^{-(|\sigma|+1)\gamma}W_h
\end{equ}
and thus
\begin{equ}[eq:phi_inc_bound]
|\phi_t-\phi_s|
\lesssim |t-s| + |t-s|^{\gamma}h^{-\gamma}W_h + |t-s|^{(N+1)\gamma} h^{-(N+1)\gamma}V_h\;.
\end{equ}

We now expand each factor $\phi_t-\phi_s = Q_{s,t} + \sum_{0\leq|\sigma|\leq N-1} f^{\sigma}(\phi_s) \mbX_{s,t}^{[\sigma]}$ in \eqref{eq:f_expan}
and equate the result with \eqref{eq:R_def}.
We see that the remainder
$R_{s,t}^\tau$ is made of a term of order
\begin{equ}[eq:1st_RP]
O(\|D^{N-|\tau|}f^\tau\|_\infty |\phi_t-\phi_s|^{N-|\tau|})
\end{equ}
plus a finite sum of 
terms which are multiples of
\begin{equ}[eq:R_second]
(D^\ell f^\tau)(\phi_s)Q_{s,t}^q \prod_{i=1}^jf^{\sigma_i}(\phi_s)\mbX_{s,t}^{[\sigma_i]}
\end{equ}
where $q+j=\ell$ for $1\leq \ell < N-|\tau|$
and such that
\begin{equ}
q (N-|\tau|) + \sum_{i=1}^j(|\sigma_i|+1) \geq N-|\tau|
\;.
\end{equ}
(The final bound is always fulfilled if $q\geq 1$.)
From \eqref{eq:Q_bound}-\eqref{eq:f_X_bound},
each term \eqref{eq:R_second} is bounded by
$\|D^\ell f^\tau\|_\infty(h+V_h+W_h)^\ell |t-s|^\theta h^{-\theta}$
with $\theta \geq (N-|\tau|)\gamma$.
Combining with \eqref{eq:phi_inc_bound} to control \eqref{eq:1st_RP},
we obtain
\begin{equ}
\|R^\tau\|_{(N-|\tau|)\gamma;\leq h}
\lesssim
\max_{1\leq \ell \leq N-|\tau|} \|D^\ell f^\tau\|_\infty (h+ W_{h} + V_h)^\ell
h^{-(N-|\tau|)\gamma}\;.
\end{equ}
Multiplying by $h^{(N+1)\gamma}\|\mbX^{[\tau]}\|$
and taking the $\max$ over $|\tau|\leq N-1$, we obtain
\begin{equ}[eq:Vh_bound]
V_h
\lesssim
\max_{0\leq |\tau|\leq N-1}
\max_{1\leq \ell \leq N-|\tau|}
\|D^\ell f^\tau\|_\infty \|\mbX^{[\tau]}\|
h^{(|\tau|+1)\gamma} (h+W_{h}+V_h)^\ell\;.
\end{equ}
Define
\begin{equ}
Z_h =
\max_{0\leq |\tau|\leq N-1}
\max_{1\leq \ell \leq N-|\tau|}
\|D^\ell f^\tau\|_\infty \|\mbX^{[\tau]}\|
h^{(|\tau|+1)\gamma} (h+W_{h})^{\ell-1}\;.
\end{equ}

\begin{lemma}\label{lem:cont}
For every $\kappa>0$ there exists $\eps_0>0$ such that, if $h\in [0,1]$ and $Z_h < \eps_0$, then
\begin{equ}
V_h \leq \kappa (h+W_h)\;.
\end{equ}
\end{lemma}

\begin{proof}
It suffices to consider $\kappa\in(0,1]$.
By assumption \eqref{eq:non_trivial}, $W_{h} \gtrsim h^{N\gamma}$ and $V_h\lesssim h^{(N+1)\gamma}$,
where the proportionality constants can depend on $f,\mbX,\phi$.
It follows that, $V_h \leq \kappa (h+W_{h})$ for all $h$ sufficiently small.

On the other hand, if $V_h \leq \kappa(h+W_h)$, then by \eqref{eq:Vh_bound}
\begin{equ}
V_h \leq C Z_h( h+W_h )\;,
\end{equ}
where $C>0$ does not depend on $f,\mbX,\phi$,
and therefore, if $Z_h<\eps_0 \eqdef \kappa C^{-1}/2$, then
\begin{equ}
V_h \leq \kappa(h+W_h)/2\;.
\end{equ}
Since $V_h \leq \kappa(h+W_h)$ for $h$ sufficiently small and since
$h,W_h,V_h,Z_h$ are all continuous and increasing in $h$,
it follows by continuity that $V_h \leq \kappa(h+W_h)/2$ for all $h$ such that $Z_h < \eps_0$.
\end{proof}

Recall that we assume the exponents in Assumption \ref{as:eta} are $\eta(\tau,\ell)=0$.
Recall also the definition $\|(f,\mbX)\| = \|(f,\mbX)\|^{(1)} + \|(f,\mbX)\|^{(2)}$
where $\|(f,\mbX)\|^{(1)} = \max_{0\leq|\tau|\leq N-1}(\|f^\tau\|_{\infty}\|\mbX^{[\tau]}\|)^{\rho_\tau}$ with $\rho_\tau>0$ (see \eqref{eq:H_def}).
It follows immediately that smallness of $\|(f,\mbX)\|^{(1)}$ implies smallness of $\max_{0\leq |\tau|\leq N-1}\|f^\tau\|_\infty\|\mbX^{[\tau]}\|$.

Following the approach in the Young case in Section \ref{sec:sharp_Young}, we now choose the smallest $h \in [0,1]$ such that $h^{-1} W_h \leq \kappa$ for small $\kappa>0$
(for $h=0$, we interpret $h^{-1}W_h$ as a limit $h\downarrow0$);
such $h\in[0,1]$ exists provided $\max_{0\leq |\tau|\leq N-1}\|f^\tau\|_\infty \|\mbX^{[\tau]}\|\leq \kappa$.

Suppose first $h=0$.
In this case $h^{-1}V_h=0$ (understood as the limit $h\downarrow0$), and thus, by \eqref{eq:scale_rough}, $\|
\int_{-1}^\cdot f(\phi_t)\mrd \mbX_t
\|_\infty \lesssim \kappa$.
Hence, if $\kappa$ is small, then by \eqref{eq:phi_diff_int}, $\|\phi-\mathring\phi\|_\infty < \delta$, which implies $|\phi(0)|<\eps-\delta$.
This implies local coercivity in the form of Corollary \ref{cor:parameter_choice}.
Note that this handles $\gamma=1$ since we always have $h=0$ in this case.

Suppose now $h\in (0,1]$.
Assume for now that $Z_h<\eps_0$ 
for $\eps_0$ as in Lemma \ref{lem:cont}.
Then $V_h \leq \kappa(h+W_h)$ and thus $h^{-1}V_h \leq 2\kappa$.
Hence, by \eqref{eq:scale_rough},
$\|
\int_{-1}^\cdot f(\phi_t)\mrd \mbX_t
\|_\infty$
is small and we conclude in the same way as in the case $h=0$.
Therefore, to complete the verification of local coercivity, it remains to show that $Z_h \downarrow0$ as we take $\|(f,\mbX)\| \downarrow0$.

Recall the definition
\begin{equ}
\|(f,\mbX)\|^{(2)}
=
\max_{0\leq|\tau|\leq N-1}
\max_{1\leq \ell\leq N-|\tau|}
\max_{0\leq |\sigma|\leq N-1}
\Big(
\|D^\ell f^{\tau}\|_\infty
\|\mbX^{[\tau]}\|
h_\sigma^{(|\tau|+1)\gamma+\ell-1}
\Big)^{\rho_{\tau,\ell,\sigma}}
\end{equ}
where $\rho_{\tau,\ell,\sigma}>0$, $h_\sigma = H_\sigma^{\frac{1}{1-(|\sigma|+1)\gamma}}$,
and we recall $H_\sigma = \|f^\sigma\|_\infty \|\mbX^{[\sigma]}\|$.
Furthermore, in the case that $\gamma=\frac{1}{N}$, we exclude $|\sigma|=N-1$ from the $\max$ over $\sigma$.

Recall that we took the smallest $h \in [0,1]$ such that $h^{-1} W_{h}\leq \kappa$ for $\kappa>0$ small. Since we assume $h>0$, we must have $h^{-1}W_h = \kappa$ and thus
\begin{equ}
h = (\kappa^{-1}H_\sigma)^{\frac{1}{1-(|\sigma|+1)\gamma}} \lesssim h_\sigma
\end{equ}
for some $0\leq |\sigma|\leq N-1$; moreover, since $h>0$, if $\gamma=1/N<1$, then this necessarily holds for some $|\sigma|\leq N-2$ (the case $\gamma=N=1$ is handled above).
Consequently, for this $h\in (0,1]$,
\begin{equ}
Z_h \lesssim
\max_{0\leq|\tau|\leq N-1}\max_{1\leq \ell \leq N-|\tau|}
\|D^\ell f^\tau\|_\infty \|\mbX^{[\tau]}\|
h_\sigma^{(|\tau|+1)\gamma + \ell-1}\;,
\end{equ}
where the right-hand side is small whenever $\|(f,\mbX)\|^{(2)}$ is small. This concludes the verification of local coercivity as stated in Corollary \ref{cor:parameter_choice} and thus also the proof of Theorem \ref{thm:RP_bound}.

\section{Singular PDEs}\label{sec:singular}
\label{sec:SPDE}

In this section, we derive a priori estimates on
coercive PDEs in the framework of regularity structures \cite{Hairer14}.

\subsection{Setup and regularity structures}

\label{sec:diff_ops}

We use notation from Section \ref{sec:PDE_notation}.


\subsubsection{Functions}

For a multi-index $k=(k_1,\ldots,k_d)\in\N^d$, denote $|k|_\s = \sum_{i=1}^d k_i \s_i$.
Define
\begin{equ}
\N^d_{<\omega} = \{k\in\N^d\,:\,|k|_\s < \omega\}
\end{equ}
and analogously for $\N^d_{\leq \omega},\N^d_{=\omega}$ and so forth.

Let $(E,|\cdot|)$ be a finite-dimensional normed space, which will be the target space of our solution $\phi$.
We denote
\begin{equ}
\CJ = E^{\N^d_{<2}}\;,
\end{equ}
which we think of as
the space of jets of order $1$ of the solution $\phi$.
We write elements of $\CJ$ in bold font as $\bphi = (\nabla^k\phi)_{|k|_\s<2}$ or $\bphi = (\phi,\nabla \phi)$ with $\nabla \phi = (\nabla^k\phi)_{|k|_\s=1}\in E^{\N^d_{=1}}$.

Let $P \colon \CJ \to E$ be a polynomial which we assume takes the heuristic form
\begin{equ}
P(\bphi) = \phi^p + B(\nabla\phi, \phi^q)
\end{equ}
where $B$ is bilinear and, if $B\neq0$, then $2q+1=p$.
More precisely, we make the following

\begin{assumption}\label{as:P}
There exist $p\geq 2$ and a linear map $F\colon E^{\otimes p}\to E$ such that, either $P(\bphi) = F(\phi^{\otimes p})$, or $q = \frac{p-1}{2}$ is an integer and there exists a linear map $B\colon E^{\N^d_{=1}}\otimes E^{\otimes q}\to E$ such that
$P(\bphi) = F(\phi^{\otimes p}) + B(\nabla\phi \otimes \phi^{\otimes q})$.
We define
\begin{equ}
\alpha = 2/(p-1)\;.
\end{equ}  
\end{assumption}

\begin{remark}
Another way of stating the above assumption on $P$ is that there exists $\alpha>0$ such that,
if we set the degree of $\nabla^k \phi$ to $-\alpha - |k|_\s$,
then $P$ is homogenous of degree $-\alpha-2$,
i.e. the degree of every monomial $\prod_{k}\nabla^k\phi$ of $P$ is $\sum_{k} (-\alpha-|k|_\s) = -\alpha-2$.  
\end{remark}

\begin{remark}\label{rem:P_affine}
Similar to Remark~\ref{rem:drift}, it is not crucial for us that $P$ is a polynomial (although it does simplify the argument).
But in the case that $P$ is a polynomial, it \emph{is} crucial that the dependence of $P$ on $\nabla\phi$ is affine.
A monomial in $P$ of the form $(\nabla\phi)^r$ with $r\geq 2$
would require us to take $r(-\alpha-1) = -\alpha-2$, thus $\alpha\leq 0$, which is not allowed in our argument.
\end{remark}

Assumption \ref{as:P}, together with implicit function theorem and standard existence results to parabolic equations, e.g. \cite[Thm.~5.6]{Lieberman_96_Parabolic}, implies that, for all $\eps\ll1$ and continuous boundary data $\bar\phi \colon \d\Omega \to E$ with $\|\bar\phi\|_{\infty;\d\Omega}\leq \eps$,
there exists a unique $\mathring\phi\in\CC( \cl\Omega, E)$, smooth on $\Omega$, such that $\CL\mathring\phi = P(\mathring\phi)$ on $\Omega$ and $\mathring\phi\restriction_{\d\Omega} = \bar\phi$.
We now make the following additional assumption (which will only be used in Section \ref{sec:local_coer_SPDE})
that this solution exhibits coercivity.

\begin{assumption}[Coercivity]\label{as:coercive}
For all $0<\eps \ll 1$, there exists $\delta>0$ 
such that, if $\bar\phi$ and $\mathring\phi$ are as above, then $|\mathring\phi(0)| < \eps-\delta$.
\end{assumption}

\begin{remark}\label{rem:pos}
We can relax Assumption \ref{as:coercive} in certain cases.
For example, with $E=\R$, the polynomial $P(\bphi) = -\phi^p$ for even $p\geq 2$ does \emph{not} satisfy Assumption \ref{as:coercive}, but it does satisfy the same coercivity property for \emph{positive} $\bar\phi$.
In this case, a trivial modification of our argument, namely restricting the set of solutions in Section \ref{sec:setup_SPDE} to \emph{positive} functions,
yields a priori estimates for the equation (such $P$ appeared in \cite{Jin_Perkowski_25}, see Remark \ref{rem:JP}). 
\end{remark}

\begin{example}
Consider $E=\R^m$ with the Euclidean norm.
Then we can take $P(\bphi) = -|\phi|^{p-1}\phi$ or $P(\bphi) = (v_1,\ldots,v_m)$ where $v_i = -|\phi_i|^{p-1}\phi_i$ and $p\geq 3$ is odd.

For an example involving derivatives, consider a linear map $Q\colon E^{\N^d_{=1}} \to \mfo(E)$, where $\mfo(E)$ is the space of all
$A\in L(E,E)$ such that $\scal{Ax,x} = 0$ for all $x\in E$, where $\scal{\cdot,\cdot}$ is the Euclidean inner product.
Set
\begin{equ}
P(\bphi) = -|\phi|^{2}\phi + Q(\nabla\phi) \phi\;.
\end{equ}
Then $P$ satisfies the above assumptions.
Indeed, Assumption \ref{as:P} holds with $p=3$.
To verify Assumption \ref{as:coercive},
since $\Delta |\phi|^2 = 2\scal{\Delta \phi,\phi} + 2|\nabla\phi|^2$ and thus
$\CL |\phi|^2 = 2\scal{\CL \phi,\phi} - 2|\nabla\phi|^2$,
we have
\begin{equ}
\CL |\mathring\phi|^2
=
2\scal{-|\mathring\phi|^2\mathring\phi + Q(\nabla\mathring\phi)\mathring\phi,\mathring\phi}
-2|\nabla\mathring\phi|^2
=
- 2|\mathring\phi|^4
-2|\nabla\mathring\phi|^2
\leq
- 2|\mathring\phi|^4\;,
\end{equ}
from which the desired condition follows readily from the maximum principle.
\end{example}

\subsubsection{Full equation}

Recall the maps $T_{z,\lambda}\colon\R^d\to\R^d$ from \eqref{eq:T_lambda_def}. For $a>0$, denote
\begin{equ}
\Omega_a = T_{0,a}(\Omega) = (-a^2,0)\times (-a,a)^{d-1}\;,
\qquad
\d\Omega_a = T_{0,a}(\d\Omega)
\;.
\end{equ}
We now consider the equation, posed for $\phi\colon\Omega_a\to E$,
\begin{equ}[eq:SPDE]
\CL \phi = P (\bphi) + f(\bphi)\xi
\end{equ}
where $\bphi = (\nabla^k \phi)_{|k|_\s\leq 1} \colon \Omega_a\to\CJ$ is the jet of $\phi$.
Here $\xi$ is an $\R^n$-valued distribution in $\CC^\beta(\Omega_a)$ for some $\beta<0$ and $f \colon \CJ \to L(\R^n,E)$ is a sufficiently regular function.
We often write $f= (f_1,\ldots,f_n)$ where $f_i\colon \CJ \to E$ and thus
\begin{equ}[eq:f_decompose]
f(\bphi)\xi = \sum_{i=1}^n f_i(\bphi)\xi_i\;.
\end{equ}
We now make the following assumption throughout this section.

\begin{assumption}[Subcriticality]\label{as:subcrit}
The equation \eqref{eq:SPDE} is subcritical.
More specifically,
\begin{enumerate}
  \item if $f$ is constant, then we require $\beta>-\alpha-2$;
  \item if $f$ depends only on $\phi$, then we require $\beta>-2$;
  \item if $f$ depends on $\nabla \phi$, then we require $\beta>-1$.
\end{enumerate}
\end{assumption}

We will later make further growth assumptions on $f$, see Assumption \ref{as:f_poly}.

\subsubsection{Trees, elementary differentials, and models}

As written, the equation \eqref{eq:SPDE} is, in general, classically ill-posed unless $\beta$ is not too negative,
e.g. for constant $f=1$ and $P(\bphi) = -|\phi|^{p-1}\phi$, classical well-posedness requires $\beta>-2$, which is much more restrictive than our requirement $\beta>-2-\alpha = -2-\frac{2}{p-1}$.

The theory of regularity structures \cite{Hairer14} allows one to give meaning to solutions to \eqref{eq:SPDE} for all values of $\beta$ indicated above provided that $\xi$ is enhanced to a so-called \emph{model} over a \emph{regularity structure} determined by the equation \eqref{eq:SPDE}.
We work directly at the level of models, so the distribution $\xi$ will not appear again beyond the equation \eqref{eq:SPDE}.
We do not attempt to fully review the theory of regularity structures, but we review the important definitions and results.
See \cite{Hairer_15_intro_RS,FrizHairer20,Bailleul_Hoshino_20_guide,ESI_lectures} for overviews of the theory.

\begin{definition}[Trees]\label{def:trees}
A \emph{labelled tree}, or simply \emph{tree},
is a connected, acyclic, undirected graph $\tau$ with the following data:
\begin{itemize}
  \item $\tau$ has a distinguished vertex $\rho_\tau$ called the root,
  \item every vertex $v$ of $\tau$ carries a so-called \emph{polynomial} label $\mfn(v) \in\N^d$ and a \emph{noise} label $\Xi(v)\in\{0,1,\ldots, n\}$
  (recall that $n$ is the dimension of the target space of $\xi$),
  \item every edge $e$ carries a \emph{derivative} label $\mfe(e)\in\N^d$.
\end{itemize}
We can write every tree $\tau$ uniquely as
\begin{equ}[eq:tree]
\tau = X^k \Xi_j \prod_{i=1}^\ell \CI^{k_i}[\tau_i]
\end{equ}
where the root $\rho_\tau$ has polynomial label $\mfn(\rho_\tau)=k\in\N^d$, noise label $\Xi(\rho_\tau)= j$, and 
the factor $\prod_{i=1}^\ell \CI^{k_i}[\tau_i]$ means that  $\rho_\tau$ is incident with $\ell\geq 0$ edges with derivative labels $k_1,\ldots,k_\ell$ and where the $i$-th edge is also incident with the root of another tree $\tau_i$.
See Figure \ref{fig:tree}.
Whenever $k_i=0$, we simply write $\CI = \CI^{k_i}$.
We sometimes write $\Xi_0 = \bone$ and drop it from the expression \eqref{eq:tree}, suggestive of the fact that $\Xi_0$ represents the constant `$1$' function in front of $P$.

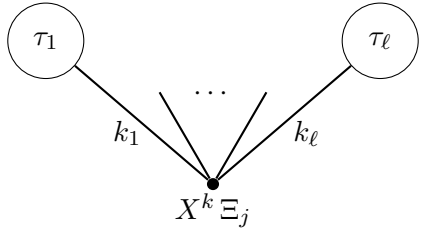
\begin{figure}
\begin{center}
\begin{tikzpicture}[
  grow            = up,
  level distance  = 19mm,
  sibling distance= 22mm,
  edge from parent/.style = {draw,thick},
  every node/.style       = {inner sep=1pt}
]

\node (rho) [fill=black,circle,inner sep=1.6pt] {}
  child { node[circle,draw,minimum size=1cm,inner sep=0pt] {$\tau_\ell$}
          edge from parent node[pos=.55,below,xshift=2mm] {$k_\ell$} }
  child[edge from parent/.style={draw=none}]
        { node[yshift=-20pt] (dots) {$\cdots$} }
  child { node[circle,draw,minimum size=1cm,inner sep=0pt] {$\tau_1$}
          edge from parent node[pos=.55,below,xshift=-1mm] {$k_1$} };

\node at ($(rho)+(0,-10pt)$) {$X^{k}\,\Xi_{j}$}; 

\foreach \angle in {60,120}                   
  \draw[thick] (rho) -- ++(\angle:40pt);

\end{tikzpicture}
\end{center}
\caption{Example of tree notation \eqref{eq:tree}}
\label{fig:tree}
\end{figure}
\end{definition}

\begin{remark}
Our trees are combinatorial, meaning that the order of edges incident to $\rho_\tau$ is not important, i.e. the product in \eqref{eq:tree} is commutative.
\end{remark}

We define the tree factorial $\tau!\in\N$, $\tau!\geq1$, inductively as follows.
Write
\begin{equ}
\tau = X^k \Xi_j \prod_{i=1}^r (\CI^{k_i}[\tau_i])^{m_i}
\end{equ}
where $m_i\geq0$ and $(k_1,\tau_1),\ldots,(k_r,\tau_r)$ are pairwise distinct.
We then define
\begin{equ}[eq:factorial]
\tau! = k! \prod_{i=1}^r m_i! (\tau_i!)^{m_i}
\end{equ}
where we recall $k! = k_1!\cdots k_d!$ for $k\in\N^d$
and an empty product is equal to $1$.
We equip the $\R$-linear span of trees with the inner product
\begin{equ}[eq:inner_product]
\scal{\tau,\sigma} = \delta_{\tau,\sigma} \tau!
\;.
\end{equ}

\begin{definition}\label{def:mfc}
For a tree $\tau$, we define the \emph{homogeneity} $|\tau|_\s\in\R$ as follows. Set
\begin{equ}
|\Xi_j|_\s = \beta 
\quad \text{ for } j=1,\ldots,n 
\;,
\qquad
|\bone|_\s = 0
\;,
\end{equ}
and then inductively
\begin{equ}
|X^k
\tau|_\s = | k |_\s + |\tau|_\s
\;,
\quad
|\tau \sigma|_\s = |\tau|_\s + |\sigma|_\s
\;,
\quad
|\CI^k (\tau)|_\s = |\tau|_\s + 2 - |k|_\s
\;.
\end{equ}
Define also the \emph{critical homogeneity} $\mfc(\tau)$ in the same way except $\mfc(\Xi_j) = -\alpha - 2<\beta$.
\end{definition}

We next follow the setup of \cite{Hairer14,BHZ19,BCCH21} specialised to our equation \eqref{eq:SPDE} and restrict our space of trees as follows.
Denote
\begin{equ}[eq:widehat_J]
\widehat\CJ = E^{\N^d}\;,
\end{equ}
which we think of as the space of jets of all order of $\phi$.
For a finite dimensional vector space $F$ and $k\geq 0$, let
$\CC^k(\widehat \CJ,F)$ denote the space of maps $g\colon \widehat \CJ\to F$ that depend on only finitely many components in $E^{\N^d}$ and are of class $\CC^k$.
Recalling that $\CJ=\N^{d}_{<2}$, note that any map $g\in \CC^k(\CJ , F)$ is well-defined as an element of $\CC^k(\widehat\CJ , F)$ by projecting to levels $|\ell|_\s<2$.

Define for $\ell\in\N^d$ and $g\in\CC^k(\widehat\CJ,F)$
the function $D^\ell g \in \CC^{k-1}(\widehat\CJ, L(E,F))$ as the derivative in the $\ell$-th variable,
i.e. $D^\ell g(\bphi)(h) = D(\bphi)(h_\ell)$
where $D g \colon \widehat\CJ \to L(\widehat\CJ,F)$ is the usual Fr{\'e}chet derivative and $h_\ell\in E^{\N^d}$ maps $\ell$ to $h$ and $k$ to $0$ for $k\neq \ell$.
For $\bk=\{k_1,\ldots,k_{|\bk|}\}\in\N^{\N^d}$ we set
\begin{equ}
D^\bk g = D^{k_1}\cdots D^{k_{|\bk|}}g\colon \CJ \to L(E^{\otimes|\bk|},F)
\end{equ}
whenever $g$ is sufficiently regular.
(Note that $D^\ell,D^k$ commute for $\ell,k\in\N^d$, so that $D^\bk$ is well-defined and compatible with our convention of treating $\bk$ as a multi-set.)

For $i\in[d]$ and $g\in\CC^k(\widehat\CJ,F)$ define
$\d^{i} g \in \CC^{k-1}( \widehat\CJ, F)$ by
\begin{equ}
\d^{i} g(\bphi) = Dg(\bphi)(\bphi_{\cdot+i})
=\sum_{\ell \in \N^d} (D^\ell g)(\bphi)(\bphi_{\ell+i})\;,
\end{equ}
where we use the shorthand $i$ for the singleton $\{i\} \in\N^d$.
We then define $\d^k g$ for any $k\in\N^d$ by composition.
($\d^k$ is well-defined since $\d^i,\d^j$ commute. But note that $\d^k,D^\ell$ do \emph{not} commute.)

We next define subsets $\mfR_j\subset \N^{\N^d}$ for $j\in\{0,1,\ldots, n\}$ (analogous to a `rule' in \cite{BHZ19}, see also \cite[Sec.~1.3]{ESI_lectures}), as follows.
We set 
\begin{equ}[eq:Rzero]
\mfR_0 = \{\bk\in \N^{\N^d}\, :\, D^\bk P \neq 0\} \;.
\end{equ}
Next, recalling Cases 1-3 in Assumptions \ref{as:subcrit}, we define $\mfR_j$ for $j>0$ as follows.
\begin{enumerate}
  \item If $f$ is constant, then $\mfR_j=\emptyset$.
  \item If $f$ depends only on $\phi$, then $\mfR_j = \{\bk\in\N^{\N^d} \,:\, \bk \text{ contains only } 0\}$.
  \item If $f$ depends on $\nabla\phi$, then $\mfR_j = \{\bk\in\N^{\N^d} \,:\, |k|_\s \leq 1 \text{ for all } k\in\bk\}$.
\end{enumerate}

\begin{definition}\label{def:mfT}
Consider a tree $\tau$ of the form \eqref{eq:tree}.
We say that $\tau$ is \emph{conforming}
if each $\tau_i$ is conforming and $\{k_1,\ldots,k_\ell\} \in \mfR_j$.
(The base case is $\tau = X^k \Xi_j$, which is conforming for all $k,j$.)
A \emph{leaf} of $\tau$ is a non-root vertex $v\neq\rho_\tau$ such that $v$ is incident with only one edge.
We call a leaf $v$ of $\tau$ \emph{polynomial} if it has noise label $\Xi(v)=0$.

\begin{itemize}
  \item To describe the right-hand side of the equation \eqref{eq:SPDE}, we define the set $\mfW$ consisting of all conforming trees $\tau$ that have no polynomial leaves.
Denote also
\begin{equ}
\mfW_{\leq\omega} = \{\tau\in\mfW \,:\, -2<|\tau|_\s\leq \omega\}
\end{equ}
and similarly for $\mfW_{< \omega}$.
Let $\CW_{\leq \omega}$ denote the $\R$-linear span of $\mfW_{\leq \omega}$, and similarly for $\CW_{<\omega}$.

\item 
To describe the solution $\phi$ and its derivatives, we define the set $\mfU$ consisting of all trees $\tau$ such that $|\tau|_\s>0$ and $\tau=\CI^{k}[\sigma]$ for some $k\in\N^d$ and $\sigma \in \mfW_{< 0}$.
Remark that $|\tau|_\s < 2$ for every $\tau\in\mfU$.

Consider further the set of polynomials $\bar\mfT = \{X^k\,:\, k\in\N^d\}$.
Let $\bar\mfT_{\leq \omega} = \{X^k\,:\, |k|_\s\leq\omega\}$ and denote by $\bar\CT_{\leq\omega}$ the span of $\bar\mfT_{\leq \omega}$, and likewise for $\bar\CT_{<\omega}$.
Define the linear space
\begin{equ}
\CU_{\leq 2} = \Span(\mfU) \oplus \bar\CT_{\leq 2}\;.
\end{equ}

\item 
Finally, to put all the relevant trees together, we define the set of trees
\begin{equ}
\mfT_{\leq 2} = \mfW_{< 0} \sqcup \mfU \sqcup \bar\mfT_{\leq 2}
\end{equ}
and the linear space
\begin{equ}
\CT_{\leq 2} = \Span(\mfT_{\leq 2}) = \CW_{< 0}\oplus \CU_{\leq 2}\;.
\end{equ}
\end{itemize}
All of the above linear spaces of trees are equipped with the inner product \eqref{eq:inner_product}.
\end{definition}

One should keep in mind that $\CW_{\leq0}$ is used to describe the right-hand side of the \emph{remainder} equation of \eqref{eq:SPDE}
and this is the reason that the trees $\tau\in\mfW$ with $|\tau|_\s\leq -2$ are excluded from $\mfW_{\leq\omega}$.
Remark that $\mfW_{\leq 0} \cap \bar\mfT_{\leq 2} = \{\bone\}$.

\begin{remark}
Not allowing polynomial leaves in $\mfW$ makes our regularity structure (see Definition \ref{def:model}) somewhat simpler.
The same condition is used in \cite{Hairer14},
but one could work without it as done in \cite{BHZ19,BCCH21}.
The trees in $\mfU$ will play a role in the structure group of our regularity structure.
\end{remark}

\begin{definition}[Elementary differentials]\label{def:Upsilons}
For conforming $\tau$ and $f$ sufficiently regular, define
\begin{equ}
\Upsilon^\tau \colon \widehat\CJ \to E
\end{equ}
inductively as follows.
Recalling \eqref{eq:f_decompose}, we denote for $j=1,\ldots,n$
\begin{equ}
\Upsilon^{\Xi_j} = f_j\;,\qquad
\Upsilon^{\bone} =\Upsilon^{\Xi_0} = P \;.
\end{equ}
Then for $\tau = X^k \Xi_j \prod_{i=1}^\ell \CI^{k_i}[\sigma_i]$ we define
\begin{equ}[eq:Upislon_induct]
\Upsilon^\tau(\bphi) = \{\d^k (D^{k_1}\cdots D^{k_\ell} \Upsilon^{\Xi_j})\}(\bphi)(\Upsilon^{\sigma_1}(\bphi),\ldots,\Upsilon^{\sigma_\ell}(\bphi))
\;.
\end{equ}
\end{definition}

\begin{remark}
To check that \eqref{eq:Upislon_induct}
make senses, note that $D^{k_1}\cdots D^{k_\ell} \Upsilon^{\Xi_j} \colon \widehat\CJ \to L(E^{\otimes \ell},E)$,
therefore $\{\d^k (D^{k_1}\cdots D^{k_\ell} \Upsilon^{\Xi_j})\}(\bphi)\in L(E^{\otimes \ell},E)$,
which we evaluate on $(\Upsilon^{\sigma_1}(\bphi),\ldots,\Upsilon^{\sigma_\ell}(\bphi))$ treated as an element of $E^{\otimes\ell}$.
Since the derivatives $D^{k_i}, D^{k_j}$ commute, the order of terms in the $\ell$-tensor $(\Upsilon^{\sigma_1}(\bphi),\ldots,\Upsilon^{\sigma_\ell}(\bphi))$ does not matter,
in agreement with our convention that the product in \eqref{eq:tree} is unordered.
\end{remark}

\begin{remark}\label{rem:dependence}
If $\tau$ is not conforming, then $\Upsilon^\tau=0$
Moreover, if $\tau$ is conforming and $\Upsilon^\tau$ depends on $\nabla^\ell\phi$, then $|\tau|_\s >|\ell|_\s-2$
by Lemma \ref{lem:tau_dependence} below.
In particular, $\Upsilon^\tau$ is a well-defined function
$
\Upsilon^\tau \colon \CJ \to E
$
for all conforming $\tau$ with $|\tau|_\s\leq0$,
and $\Upsilon^\tau$ is constant for all conforming $\tau$ with $|\tau|_\s\leq -2$.
\end{remark}

Throughout this section, similar to~\cite[Assumption 5.3]{Hairer14}, we make the following standard

\begin{assumption}\label{as:non_integer}
Consider $\tau\in\mfW_{\leq 0}$ such that $\tau\neq\bone$.
Then $|\tau|_\s$ is not an integer.
\end{assumption}

\begin{example}
Suppose $\beta$ is irrational. Then Assumption \ref{as:non_integer} is satisfied.
Indeed, $|\tau|_\s\notin\Z$ for any tree $\tau$ with at least one noise label $j\neq0$.
On the other hand, if $\tau \in\mfW$ with all noise labels $\Xi(v)=0$, then, due to Definition \ref{def:mfT}, $\tau=X^k$ for some $k\in\N^d$, which proves the claim.
\end{example}

The next lemma is an important consequence of subcriticality (Assumption \ref{as:subcrit}).

\begin{lemma}\label{lem:lowest_tree}
The set $\{\tau\,:\,\tau \text{ is conforming and } |\tau|_\s <\omega\}$ is finite for any $\omega\in\R$.
Moreover, 
if $\tau$ is conforming, then
$|\tau|_\s \geq \beta$ with equality if and only if $\tau=\Xi_j$ for some $j\neq0$.
\end{lemma}

\begin{proof}
This follows from the proof of \cite[Lem.~8.10]{Hairer14} (see also \cite[Prop.~5.15]{BHZ19}).
\end{proof}

\subsubsection{Growth assumptions on non-linearity}

As in the rough path setting, we need to make an assumption on the growth of $f$ and its derivatives.
However, while the assumptions on $f$ in Section \ref{sec:RPs} came only from scaling,
here we additionally need to use these assumptions to deal with weighted Schauder estimates.
Like in Section \ref{subsec:general_RPs}, we first make assumptions directly on the elementary differentials $\Upsilon$.
In Section \ref{sec:linear_cond_SPDE} we consider simple linear conditions on $f$ when it behaves like a polynomial.

For $k\in\N_{<2}^d$, let
\begin{equ}[eq:gamma_k]
\gamma_k = 1\wedge \inf\{|\CI^k[\tau]|_\s\,:\, \CI^k[\tau] \in \mfU\}
\in (0,1]\;.
\end{equ}

\begin{remark}
$\gamma_k$ is a lower bound on the optimal H\"older exponent of $\nabla^k \phi$, so $\gamma_0$ corresponds to $\gamma$ from Section \ref{sec:RPs}.
One has $\gamma_0 = \beta+2$ if $\beta\in(-2,-1)$ and $\gamma_0 = 1$
if $\beta\in(-1,0)$. 
In all cases, $\gamma_k=\gamma_{\ell}$ if $|k|_\s=|\ell|_s$.
\end{remark}

Define further
\begin{equ}[eq:zeta_def]
\zeta = \min\{|\tau|_\s - \floor{|\tau|_\s}\,:\, \tau\in\mfW \text{ such that } \floor{|\tau|_\s} \leq 0 \text{ and } \tau \neq \bone \}\;.
\end{equ}
Note that $\zeta\in (0,1)$ by Assumption \ref{as:non_integer} and Lemma \ref{lem:lowest_tree}.

\begin{definition}\label{def:B_tau}
For a multi-index $\bk\in\N^{\N^d_{<2}}$, denote $\gamma_\bk = \sum_{k\in\bk}\gamma_k$.
For $\tau\in\mfW_{\leq 0}$,
define
\begin{equs}
\mathring B^\tau &= \{\bk \in \N^{\N^d_{<2}} \, :\, \gamma_\bk < \zeta-|\tau|_\s \text{ and } D^\bk \Upsilon^\tau\neq 0\}
\;,
\\
B^\tau &=
\{\bk \in \N^{\N^d_{<2}} \, :\, [\bk = 0 \text{ or } \bk\setminus\{k\} \in \mathring B^\tau \text{ for all } k\in\bk] \text{ and } D^\bk \Upsilon^\tau\neq 0\}\;.
\end{equs}
\end{definition}



\begin{remark}\label{rem:zeta_optimal}
Our choice for $\zeta$ corresponds to $(N+1)\gamma-1$ in the rough paths case of Section \ref{sec:RP_proof} (see Remark \ref{rem:compare_RP}) and is, in general, the `largest possible' in the following sense:
if we replace $\zeta$ by $\bar\zeta\in (0,\zeta]$,
then, by considerations similar to the proof of Lemma \ref{lem:R_tau_F_bound},
the set of `solutions' to the PDE \eqref{eq:SPDE} in the sense of Definition \ref{def:solution_SPDE} does not change.
\end{remark}

Recall that we write $\bphi = (\phi,\nabla\phi)\in E\oplus E^{\N^d_{=1}}$.
For $\bk\in\N^{\N^d}$, denote
\begin{equ}
|\bk|_\s = \sum_{k\in\bk} |k|_\s\;.
\end{equ}
We make the following assumption (which will only be used in Sections \ref{sec:remainder_SPDE}-\ref{sec:local_coer_SPDE}).
Recall the notation $|x|_\eta$ from \eqref{eq:x_eta}.

\begin{assumption}\label{as:f_poly}
$f$ is sufficiently regular such that $D^\bk\Upsilon^\tau$ is a continuous function for all $\tau\in\mfW_{< 0}$ and $\bk\in B^\tau$.
Moreover, uniformly in $\bphi = (\phi,\nabla\phi)\in\CJ$,
\begin{equ}[eq:Upsilon_bound]
|D^\bk \Upsilon^\tau(\bphi)| \lesssim
|\phi|_{\eta(\tau,\bk)}
+ |\nabla\phi|_{\nabla\eta(\tau,\bk)} |\phi|_{\bar \eta(\tau,\bk)}
\;,
\end{equ}
where $\eta(\tau,\bk),\nabla\eta(\tau,\bk),\bar\eta(\tau,\bk)\geq 0$ satisfy
\begin{equ}[eq:scaling_assump1]
\Delta(\tau,\bk)\eqdef \alpha - \alpha|\bk| + |\tau|_\s+2-|\bk|_\s - \alpha\eta(\tau,\bk) > 0\;,
\end{equ}
\begin{equ}[eq:scaling_assump2]
\alpha \eta(\tau,\bk)
=
(\alpha+1)\nabla\eta(\tau,\bk) + \alpha\bar\eta(\tau,\bk)
\;,
\end{equ}
\begin{equ}[eq:Schauder_assump]
|\tau|_\s +2 - |\bk|_\s > \nabla\eta(\tau,\bk)
\;.
\end{equ}
We let $\trinorm{D^\bk \Upsilon^\tau}$ denote the optimal implicit constant in \eqref{eq:Upsilon_bound} (which depends on $\eta(\tau,\bk)$ etc.).
\end{assumption}

A few remarks are in order.
The inequality \eqref{eq:scaling_assump1} is needed to ensure that our `drivers' shrink under scaling maps $R_{z,\lambda}$ introduced in \eqref{eq:scaling_driver_SPDE} below.
The identity \eqref{eq:scaling_assump2} implies that the two terms on the right-hand side of \eqref{eq:Upsilon_bound} scale in the same way, see Lemma \ref{lem:scaling_f}.
The final condition \eqref{eq:Schauder_assump} is required to close a weighted Schauder estimate in the proof of Lemma \ref{lem:R_tau_F_bound}.

\begin{remark}[PDE vs. ODE]\label{rem:RP_assump}
When translating the notation of this section to rough paths (see Remark \ref{rem:compare_RP}), a forest $\tau$ in Section \ref{sec:RPs} corresponds here to a tree $\tilde\tau$ with $|\tilde\tau|_\s = |\tau|\gamma + \gamma - 1$
and $\eta(\tau,\bk)$ here plays the role of $\eta(\tau,\ell)$ from Assumption \ref{as:eta}.
Since we only consider $|\bk|_\s=0$ in the rough path case, assumption \eqref{eq:Schauder_assump} is automatically satisfied.
Assumption \eqref{eq:scaling_assump1} becomes equivalent to $\Delta(\tau,\ell)>0$ in the notation of \eqref{eq:Delta_def}.
This assumption is simpler but more restrictive than Assumption~\ref{as:Delta} (except when $f$ behaves like a polynomial, in the sense of Section \ref{sec:linear_condition_RP} with $q=-1$, in which case \eqref{eq:scaling_assump1} is equivalent to Assumption~\ref{as:Delta}).
A reason for this gap is that we handle Schauder estimates by integration against the (truncated) heat kernel, which is a highly non-local operation and we could not find a way to extend the localisation argument in Sections \ref{sec:sharp_Young} and \ref{sec:RPs} to the setting here.
Furthermore, in contrast to Assumption \ref{as:eta}, we do not allow $\eta(\tau,\bk) < 0$,
thus we cannot leverage decay of $D^\bk\Upsilon^\tau$ at infinity (see Remark \ref{rem:eta_neg}).
We find it an interesting problem to determine if these gaps between the two settings can be closed.
\end{remark}


\begin{example}\label{ex:only_nabla}
Suppose $f(\bphi) = f(\nabla\phi)$ depends only on $\nabla\phi$.
If $\beta\in (-1/2,0)$, in which case \eqref{eq:SPDE} is classically well-posed,
then $\mfW_{<0} = \{\Xi_j\,:\,1\leq j\leq n\}$ and $\mathring B^{\Xi_j}=\{0\}$,
and it follows that $f$ satisfies Assumption \ref{as:f_poly} provided there exists $1\leq \theta < \frac{\alpha+2+\beta}{\alpha+1}$ such that for all $k\in\N^d_{= 1}$
\begin{equ}
|f(\nabla\phi)| \lesssim |\nabla\phi|_{\theta}
\;,\qquad
|D^k f(\nabla\phi)| \lesssim |\nabla\phi|_{\theta - 1}\;.
\end{equ}
But if $\beta\in (-1,-1/2)$, then $\gamma_k = 1+\beta < 1/2$ for $|k|_\s=1$ and therefore $\mathring B^{\Xi_j}$ can contain every singleton $\{k\}$ with $|k|_\s=1$.
It follows that Assumption \ref{as:f_poly} can only be satisfied if $D^\bk f = 0$ for all $\bk = \{k_1,k_2\}$ for which $|\bk|=|\bk|_\s=2$.
Therefore, in the case $\beta\in (-1,-1/2)$, $f$ must have affine dependence on $\nabla\phi$.
\end{example}

The next lemma implies that, if $f$ is constant, then Assumption \ref{as:f_poly} is satisfied.

\begin{lemma}\label{lem:Upsilon_poly}
Consider conforming $\tau$ and $\bk\in\N^{\N^d_{<2}}$.
Suppose either $f$ is constant or $\tau=\bone$
and define $\eta(\tau,\bk),\bar\eta(\tau,\bk)\in\R$ by
\begin{equ}[eq:eta_f_const]
\alpha\eta(\tau,\bk) = 2+\alpha + \mfc(\tau) - |\bk|\alpha - |\bk|_\s\;,
\qquad
\alpha+1+\alpha\bar\eta(\tau,\bk) = 
\alpha \eta(\tau,\bk)\;.
\end{equ}
\begin{enumerate}
\item If $\eta(\tau,\bk)<0$, then $D^\bk\Upsilon^\tau = 0$.
\item If $\bar\eta(\tau,\bk)<0$, then $D^\bk\Upsilon^\tau$ does not depend on $\nabla \phi$.
\item In all cases, $D^\bk \Upsilon^{\tau}$ has affine dependence on $\nabla\phi$ and
\begin{equ}[eq:Dk_Upsilon_tau]
|D^\bk \Upsilon^{\tau}(\bphi)| \lesssim
|\phi|^{\eta(\tau,\bk)}
+ \delta_{0,|\bk|_\s}|\phi|^{\bar \eta(\tau,\bk)} |\nabla\phi|
\;,
\end{equ}
where the final term is absent if $D^\bk \Upsilon^{\tau}$ does not depend on $\nabla\phi$. 
\end{enumerate}
\end{lemma}

\begin{proof}
Since either $f$ is constant or $\tau=\bone$ by assumption, $\Upsilon^\tau$ is a polynomial
and $\Upsilon^\tau_\lambda = \Upsilon^\tau$.
Then, by Lemma \ref{lem:crit_scaling}, $\Upsilon^\tau \circ Q_\lambda = \lambda^{-2-\alpha-\mfc(\tau)}\Upsilon^\tau$,
thus $\Upsilon^\tau$ is homogenous of degree $2+\alpha+\mfc(\tau)\geq 0$, where we count $\nabla^k\phi$ as having degree $\alpha+|k|_\s$.
It is then clear that, if $\eta(\tau,\bk)<0$, then $D^\bk\Upsilon^\tau = 0$, while if $\bar\eta(\tau,\bk)<0$, then $D^\bk\Upsilon^\tau$ does not depend on $\nabla\phi$.

Suppose that $\Upsilon^\tau$ does not depend on $\nabla\phi$.
Then $D^\bk\Upsilon^\tau=0$ if $|\bk|_\s>0$ and otherwise
\begin{equ}
|D^\bk\Upsilon^\tau(\bphi)| \lesssim |\phi|^{\frac{2+\alpha+\mfc(\tau)-|\bk|\alpha}{\alpha}}
= |\phi|^{\eta(\tau,\bk)}
\;,
\end{equ}
proving \eqref{eq:Dk_Upsilon_tau} in this case.
Suppose now that $\Upsilon^\tau$ depends on $\nabla\phi$.
Then $D^\bk\Upsilon^\tau=0$ if $|\bk|_\s > 1$ since $2(\alpha+1) > 2+\alpha+\mfc(\tau)$.
It follows that $D^\bk\Upsilon^\tau$ has affine dependence on $\nabla\phi$ and
\begin{equation*}
|D^\bk\Upsilon^\tau(\bphi)| \lesssim |\phi|^{\frac{2+\alpha+\mfc(\tau)-|\bk|\alpha - |\bk|_\s}{\alpha}}
+
\delta_{0,|\bk|_\s}
|\phi|^{\frac{1+\mfc(\tau)-|\bk|\alpha }{\alpha}}|\nabla\phi|
=
|\phi|^{\eta(\tau,\bk)}
+
\delta_{0,|\bk|_\s}
|\phi|^{\bar\eta(\tau,\bk)} |\nabla\phi|
\;.\qedhere
\end{equation*}
\end{proof}

The following lemma gives several important cases when \eqref{eq:Schauder_assump} is satisfied with $\nabla\eta(\tau,\bk)=0$.

\begin{lemma}\label{lem:tau_dependence}
Consider a conforming tree $\tau$ and $\bk\in\N^{\N^d}$ such that $|\bk|\geq 1$ and $D^\bk\Upsilon^\tau\neq0$.
Assume that either that $|\bk|=1$ or $\beta\leq-1$.
Then
\begin{equ}[eq:tau_depend]
|\tau|_\s > |\bk|_\s-2
\;.
\end{equ}
\end{lemma}

\begin{proof}
The case $|\bk|=1$ follows immediately from Lemma \ref{lem:tau_lower}.
Suppose now that $\beta\leq -2$.
Then, by Lemma \ref{lem:tau_lower},
\begin{equ}
|\tau|_\s > |\bk|_\s - |\bk|(2+\beta)+\beta
\geq |\bk|_\s-2
\end{equ}
where the final bound follows from $|\bk|\geq 1$ and $2+\beta\leq 0$.
Suppose now that $\beta\in (-2,-1]$.
Because the set $\mfW$ is independent of the choice of $\beta\in (-2,1]$,
it suffices to consider $\beta=\eps-2$ for $\eps>0$ small.
Then, by Lemma \ref{lem:tau_lower},
\begin{equ}
|\tau|_\s > |\bk|_\s - |\bk|(2+\beta)+\beta
= |\bk|_\s - |\bk|\eps + \eps - 2\;,
\end{equ}
and thus $\lim_{\beta\downarrow-2}|\tau|_\s \geq |\bk|_\s-2$.
Because $|\tau|_\s$ is increasing in $\beta$, we obtain $|\tau|_\s > |\bk|_\s-2$.
\end{proof}

\begin{remark}
For $\beta>-1$ and $|\bk|>1$, \eqref{eq:tau_depend} may fail.
E.g. take $E=\R$, $d=2,$ $n=1$, and $f(\bphi) = \sin(\nabla\phi)$ so that $D^\bk f\neq 0$ for all $\bk\in\N^{\N^d_{<2}}$ with $|\bk|=|\bk|_\s$. Taking $|\bk|_\s$ large, we can then make the right-hand side of \eqref{eq:tau_depend} arbitrarily positive with $|\tau|_\s=\beta$ fixed.
\end{remark}

\subsubsection{Models and notion of solution}
\label{sec:models}

Denote by $\CH^+$ the free commutative algebra generated by $X^i$ and trees $\CI^{k}[\sigma]$ in $\mfU$ (see Definition \ref{def:mfT}).
We let $\bone^+\in\CH^+$ denote the empty product, which is the unit under multiplication.
Recall from \cite[Sec.~8.1]{Hairer14} that there exists a linear map  $\Delta\colon \CT_{\leq 2} \to \CH^+\otimes \CT_{\leq 2}$ characterised by
\begin{equ}
\Delta \Xi_j = \bone^+\otimes \Xi_j
\;,\qquad \Delta X^i = X^i\otimes\bone + \bone^+\otimes X^i
\;,
\qquad
\Delta (\tau\sigma) = (\Delta\tau)(\Delta \sigma)
\end{equ}
for $j=0,\ldots,n$ and $i\in[d]$, and
\begin{equ}[eq:Delta_induct]
\Delta\CI^k[\sigma] = (\id\otimes \CI^k)\Delta\sigma + \sum_{|\ell|_\s < |\CI^k[\sigma]|_\s} \CI^{k+\ell}[\sigma]\otimes \frac{X^{\ell}}{\ell!}\;.
\end{equ}
(Note that we swap the order of tensors compared to \cite{Hairer14,BHZ19}.)

\begin{remark}
It is possible that $\tau\sigma\in\mfT_{\leq 2}$ while $\tau,
\sigma$ are not in $\mfT_{\leq 2}$. Therefore, the right-hand side of the inductive step $\Delta(\tau\sigma)=(\Delta\tau)(\Delta\sigma)$ should be understood as one term rather than as separate elements $\Delta\tau,\Delta\sigma\in\CH^+\otimes\CT_{\leq2}$.
This ensures that $\Delta$ is well-defined because every term in the product $(\Delta\tau)(\Delta\sigma)$, upon applying the induction \eqref{eq:Delta_induct}, is in $\CH^+\otimes\mfT_{\leq2}$.
\end{remark}

Every character $g\colon \CH^+ \to \R$ (i.e. a multiplicative linear functional for which $g(\bone^+)=1$) defines an invertible linear map $\Gamma\colon\CT_{\leq 2}\to \CT_{\leq 2}$
by
\[
\Gamma\tau = (g\otimes \id)\Delta\tau\;.
\]
Note that $g$ can be recovered from $\Gamma$ by the identities
\begin{equ}[eq:g_Gamma]
g(\CI^k\sigma) = \scal{\Gamma\CI^k\sigma,\bone}
\;,
\qquad
g(X^i)=\scal{\Gamma X^i,\bone}
\end{equ}
for every $\CI^k\sigma\in\mfU$ and $i\in[d]$.

We let $\CG$ denote the space of linear maps $\Gamma$ defined in this way by a character $g\colon \CH^+ \to \R$.
For all $\Gamma\in\CG$, the map $\id-\Gamma$ is triangular in the sense that, for all $\tau\in\mfT_{\leq 2}$,
\begin{equ}
\tau - \Gamma\tau \in \Span \{\sigma\in\mfT_{\leq 2}\,:\, |\sigma|_\s < |\tau|_\s\}
\;.
\end{equ}

\begin{definition}[Regularity structure and models]\label{def:model}
Let $\CA = \{|\tau|_\s \,:\, \tau\in\mfT_{\leq 2}\}\subset (-2,2]$.
The triplet $(\CT_{\leq 2},\CG,\CA)$ is the \emph{regularity structure} associated to the equation \eqref{eq:SPDE}.
Recall that $\CT_{\leq 2}$ is equipped with the inner product $\scal{\cdot,\cdot}$ as in \eqref{eq:inner_product}, for which $\mfT_{\leq 2}$ is an orthogonal basis.

Consider an open subset $\mfK\subset\R^d$.
Recall the notation $\mfK^1$ from \eqref{eq:fattening}.
Let $r = \floor{1-\min \CA} \in \{1,2\}$.
A \emph{model} on $\mfK$ is a pair $Z=(\Pi,\Gamma)$ where $\Pi=\{\Pi_x\}_{x\in\mfK}$ is a family of maps $\Pi_x\in L(\CT_{\leq2},\CD'(\mfK^1))$ and $\Gamma=\{\Gamma_{xy}\}_{x,y\in\mfK}$ with $\Gamma_{xy} \in \CG$
such that
\begin{enumerate}[label=(\alph*)]
\item\label{pt:alg} $\Gamma_{xy}\Gamma_{yz} = \Gamma_{xz}$ and $\Pi_y = \Pi_x\Gamma_{xy}$,
\item\label{pt:bounds} for all $\tau\in\mfT_{\leq2}$ and $\CI^k[\sigma]\in\mfU$
\begin{equ}[eq:model_size_def]
\|\Pi\|_{\tau;\mfK} \eqdef \sup_{x\in\mfK} \|\Pi_x\tau\|_{|\tau|_\s;x;1;r} <\infty \;,
\qquad
\|\Gamma\|_{\CI^k[\sigma];\mfK} \eqdef \sup_{x\neq y\in\mfK}\frac{\scal{\Gamma_{xy}\CI^k[\sigma],\bone}}{|x-y|_\s^{|\CI^k[\sigma]|_\s}}
<
\infty
\;,
\end{equ}
where we recall the semi-norm $\|\cdot\|_{\omega;x;h;r}$ from \eqref{eq:Cbeta-def}, and
\item\label{pt:poly_model} $(\Pi_x X^k\tau)(y) = (y-x)^k(\Pi_x\tau)(y)$, $\Pi_x\bone\equiv 1$, and $\Gamma_{xy} X^k = (X-(x-y))^k$.
\end{enumerate}
Let us fix a kernel $K\colon\R^d\setminus\{0\}\to\R$, supported on $B_0(1)\setminus\{0\}$, that agrees with the heat kernel on $B_0(1/2)\setminus\{0\}$ and satisfies Assumptions~5.1 and~5.4 of \cite{Hairer14}.
Such a kernel exists by \cite[Lem.~5.5]{Hairer14}.
A model is called \emph{adapted to $K$} if it satisfies the conditions of~\cite[Def.~5.9]{Hairer14}, namely, for all $x\in\mfK$ and $\CI^\ell\tau\in\mfT_{\leq 2}$,
\begin{equ}[eq:adapted]
\Pi_x \CI^\ell \tau = D^\ell (K*\Pi_x \tau - \Pi_x \CJ(x)\tau)
\end{equ}
where the equality is understood in the sense of distributions on $\CD'(\mfK)$ and
where
\begin{equ}[eq:CJ_def]
\CJ(x)\tau = \sum_{|k|_\s<|\tau|_\s+2} \frac{X^k}{k!} (D^k K*\Pi_x\tau)(x)
\in \bar\CT_{\leq 2}
\;.
\end{equ}
\end{definition}

Our notion of model is essentially the same as in \cite{Hairer14}, the main difference being that we consider a local definition on the set $\mfK$
and $\|\Pi\|_{\tau;\mfK}$ depends only on $\Pi_x\tau$ tested against functions supported in the \emph{backwards} parabolic ball around $x\in\mfK$.

The linear map $\CI \colon\CW_{<0}\to \CU_{\leq 2}$ sending $\tau$ to $\CI[\tau]$ is an abstract integration map of order 2 in the sense of \cite[Def.~5.7]{Hairer14}.

\begin{remark}
While we do not assume that $\Pi_x$ is in $\CD'(\R^d)$,
$K*\Pi_x \tau$ and $\Pi_x \CJ(x)\tau$ still makes sense as distributions in $\CD'(\mfK)$ for every $x\in\mfK$.
Indeed, we have a decomposition
\begin{equ}[eq:K_n_def]
K = \sum_{n\geq 0} K_n
\end{equ} 
where $K_n$ is supported on $B_0(2^{-n})$ and satisfies the assumptions of \cite[Sec.~5]{Hairer14}.
Then for every $N\geq 1$, $K_{<N} = \sum_{n<N} K_n$ is smooth and supported in $B_0(1)$,
hence $K_{<N}*\Pi_x\tau$ is a well-defined smooth function on $\mfK$ since we assume $\Pi_x\in\CD'(\mfK^1)$.
On the other hand, let $K_{\geq N} = K-K_{<N}$ and $\CJ_{\geq N}$ be defined as in \eqref{eq:CJ_def} with $K$ replaced by $K_{\geq N}$.
Then, for $\psi\in\CC^\infty_c(\mfK)$, it is straightforward to check that there exists $N$ sufficiently large
such that $\scal{ K_{\geq N}*\Pi_x\tau,\psi}$ and $\scal{\Pi_x \CJ_{\geq N}(x)\tau,\psi}$ are both well-defined by using the knowledge of $\Pi_x$ as an element of $\CD'(\mfK)$ and the bound $\|\Pi\|_{\tau;\mfK} < \infty$.
\end{remark}


In the sequel, we write $\scal{\tau\otimes Y,\sigma} = \scal{\tau,\sigma}Y\in E$ for $\tau, \sigma\in\CT_{\leq 2}$ and $Y\in E$.

\begin{definition}
Consider a model $Z=(\Pi,\Gamma)$ on an open set $\mfK\subset\R^d$ and $\omega\in\R$.
An $E$-valued \emph{modelled distribution} of order $\omega$ on $\mfK$ is a function $F\colon\mfK \to \CT_{\leq 2}\otimes E$ such that $\scal{F_x,\sigma}=0$ for $|\sigma|_\s \geq \omega$ and, uniformly in $x,y\in\mfK$, $\tau\in\mfT_{\leq 2}$, and $|\tau|_\s < \omega$,
\begin{equ}[eq:model_dist_def]
|\scal{F_x,\tau}| \lesssim 1
\;,
\qquad
|\scal{F_x - \Gamma_{xy} F_y,\tau}| \lesssim |x-y|_\s^{\omega-|\tau|_\s}
\;.
\end{equ}
We denote by $\CD^\omega(\mfK,E)$ the space of all such functions.
\end{definition}

In the theory of regularity structures, one typically lifts the PDE \eqref{eq:SPDE} to a so-called `abstract equation' for modelled distributions, which one solves via a fixed point argument.
Since we are interested in a priori estimates rather than local well-posedness, we use the following very general notion of `solution' inspired by Davie's solution to RDEs \eqref{eq:RDE_Euler}.
We note that a similar notion appeared independently in the multi-index approach to regularity structures \cite{BOS_25_Phi4}.

We are interested in modelled distributions $\Phi\colon \mfK\to \CU_{\leq 2}\otimes E$ of the form
\begin{equ}[eq:bphi_def]
\Phi_x = \bphi_x +
\tilde\bphi_x
\;,
\quad
\bphi_x = \phi_x\bone + \sum_{0<|k|_\s\leq 2} \nabla^k\phi_x \frac{X^k}{k!}
\;,
\quad
\tilde\bphi_x = \sum_{|\tau|_\s<0} \Upsilon^\tau(\bphi_x) \frac{\CI[\tau]}{\tau!}\;,
\end{equ}
where we recall $\tau!$ from \eqref{eq:factorial}.
Here and below, summation over $k$ means $k\in\N^d$ and
summation over $\tau$ means $\tau\in\mfT_{\leq 2}$ subject to the given constraints.
We also use the shorthand $Y\tau = \tau\otimes Y$ for $Y\in E$ and $\tau\in\CT_{\leq 2}$.
Furthermore, $\nabla^k\phi\colon \mfK\to E$ is a function that is not necessarily a derivative of $\phi$
and we identify $\bar\CT_{< 2}\otimes E$ with $\CJ = E^{\N^d_{<2}}$ via $\nabla^k\phi \frac{X^k}{k!} \mapsto \nabla^k \phi$ to makes sense of $\Upsilon^\tau(\bphi)$.
Recall that $\Upsilon^\tau\colon\CJ \to E$ is well-defined by Remark \ref{rem:dependence}.

Such modelled distributions in \cite{BCCH21} are called \emph{coherent}.
It was observed therein that solutions to the `abstract' version of \eqref{eq:SPDE} are necessarily coherent.

Note that $\Phi_x$ is clearly determined by $\bphi_x$.
Moreover, by Remark \ref{rem:dependence}, if $\Upsilon^\tau$ depends on $\nabla^k\phi$, then $|\CI[\tau]|_\s > |k|_\s$.
It follows that, for an open set $\mfK$, given $\phi\colon \mfK \to E$ and $\omega > 2$, there exists at most one function $\nabla^k\phi\colon\mfK\to E$ for every $|k|_\s \leq 2$ such that
\begin{equ}
|\phi_x - \scal{\bone,\Gamma_{xy} \Phi_y}| \lesssim |x-y|_\s^{\omega}
\;,
\end{equ}
where we write $\Phi$ for the function in \eqref{eq:bphi_def}.
That is, $\nabla^k\phi$, whenever it exists, is uniquely determined by $\phi$ (the coefficient of $\bone$) once the model $Z$ is given.
In the sequel, we let $\nabla^k \phi$ denote these functions
and let $\Phi$ be defined as in \eqref{eq:bphi_def}.

\begin{definition}[Solution]\label{def:solution_SPDE}
Consider $0<a\leq1$ and a model $Z=(\Pi,\Gamma)$ on $\Omega_a$.
We say that $\phi\in \CC(\Omega_a, E)$ solves the remainder of \eqref{eq:SPDE} for $Z$ if $\Phi$ is in $\CD^{2+\zeta}(\Omega_a,E)$ and
\begin{equ}[eq:sol_def]
\sup_{x\in\Omega_a}
\Big\|\CL \phi - \sum_{\tau \in \mfW_{\leq0}} \Upsilon^\tau(\bphi_x)
\frac{1}{\tau!}
\Pi_x \tau
\Big\|_{\zeta;x;\frac12\dist{x}_a;r}
< \infty
\;,
\end{equ}
where we recall $\zeta\in(0,1)$ from \eqref{eq:zeta_def}, $r$ from Definition \ref{def:model}, and the localised semi-norms \eqref{eq:Cbeta-def},
and we denote
\begin{equ}
\dist{x}_a = \inf_{y\in\d\Omega_a}|x-y|_\s
\;.
\end{equ}
\end{definition}

Readers familiar with regularity structures, will realise that \eqref{eq:sol_def} is essentially the statement that $\CL\phi$ is the reconstruction of $\sum_{\tau \in \mfW_{\leq0}} \Upsilon^\tau(\bphi_x)
\frac{1}{\tau!}
\tau$ on $\Omega_a$ (see also \eqref{eq:local_reconst} below).
As mentioned earlier, we restrict to $|\tau|_\s>-2$ because we are looking at the \emph{remainder} of the equation \eqref{eq:SPDE}, while the `full solution' is given by $\phi_x + \sum_{\tau\in\mfW\,:\,|\tau|_\s \leq -2}\Upsilon^\tau
\frac{1}{\tau!}
\Pi_x \tau$.
The remainder $\phi$ is the same quantity that is bounded in \cite{MoinatWeber20,CMW23,GH21}. 
(For every $|\tau|_\s\leq-2$, one has $\Gamma\tau=\tau$ for all $\Gamma\in\CG$ thus $\Pi_x\tau = \Pi_y\tau$, and, according to Remark \ref{rem:dependence}, $\Upsilon^\tau$ is constant.)

\begin{remark}[Comparison to rough paths]\label{rem:compare_RP}
When comparing to Section \ref{sec:RPs}, $\mbX^{[\tau_1\cdots\tau_n]_i}_{s,t}$ should be thought as the integral over $[s,t]$ of $\Pi_s (\Xi_i\CI[\tau_1]\cdots\CI[\tau_n])$.
It is therefore natural to identify $\tau = \tau_1\cdots\tau_n$ with the vector of trees $\tilde\tau = (\Xi_i\CI[\tau_1]\cdots\CI[\tau_n])_{i=1,\ldots,n}$,
e.g. $\bone$ from Section \ref{sec:RPs} becomes $(\Xi_1,\ldots,\Xi_n)$.
Then, supposing for simplicity that $n=1$, because the homogeneity of the noise is $\beta = \gamma-1$ and $\CI$ is now $1$-regularising,
one has $|\tilde\tau|_\s = (|\tau|+1)(\gamma-1) + |\tau| = |\tau|\gamma + \gamma - 1$,
so the exponent $(N-|\tau|)\gamma$ from \eqref{eq:R_reg} corresponds to $\zeta - |\tilde\tau|_\s$ in \eqref{eq:model_dist_def}.  
\end{remark}

\begin{remark}
Similar to Remark \ref{rem:zeta}, we could work with any $\bar\zeta\in (0,\zeta]$ and obtain the same final result.
Our choice of $\zeta$ in \eqref{eq:zeta_def} is the largest possible in the sense of Remark \ref{rem:zeta_optimal}.
\end{remark}

\subsection{Main result}
\label{sec:main_SPDE}

We can now state the main result of this section.
Recall that Assumptions \ref{as:P}, \ref{as:coercive}, \ref{as:subcrit}, \ref{as:non_integer}, \ref{as:f_poly} are in place throughout this section.

Consider a function $f\colon\CJ\to L(\R^n,E)$ satisfying Assumption \ref{as:f_poly}.
Consider $0<a\leq 1$ and an adapted model $Z=(\Pi,\Gamma)$ on $\Omega_a$. For $z\in\Omega_a$ and $0<\lambda\leq \dist{z}_a$ define
\begin{equ}[eq:Zf_local_norms]
\disc{(f,Z)}_{z,\lambda} = 
\max_{\tau\in\mfW_{<0}}
\max_{\bk \in B^\tau}
\Big(
\trinorm{D^\bk \Upsilon^\tau}\|\Pi\|_{\tau;B_{z}(\lambda)}
\Big)^{\rho(\tau,\bk)}
+
\max_{\CI^k[\sigma] \in \mfU}
\Big(
\trinorm{\Upsilon^\sigma}\|\Gamma\|_{\CI^k[\sigma];B_{z}(\lambda)}
\Big)^{\rho(\sigma,0)}
\end{equ}
where we recall \eqref{eq:model_size_def}, $\trinorm{D^\bk \Upsilon^\tau}$ is taken from Assumption \ref{as:f_poly},
and we define
\begin{equ}[eq:rho_def]
\rho(\tau,\bk)
= \frac{\alpha}{\Delta(\tau,\bk)}
\end{equ}
where $\Delta(\tau,\bk) >0$ is defined in \eqref{eq:scaling_assump1} of Assumption \ref{as:f_poly}.

\begin{theorem}\label{thm:SPDE}
Let  $\phi\in\CC(\cl\Omega_a, E)$ solve the remainder of \eqref{eq:SPDE} in $\Omega_a$ in the sense of Definition \ref{def:solution_SPDE}.
Then there exists $C>0$, not depending on $a,Z,f,\phi$, such that, for all $z\in\Omega_a$,
\begin{equ}
|\phi(z)| \leq C\max\{ \disc{(f,Z)}_{z, \lambda_z} , \dist{z}_a^{-\alpha}\}\;,
\end{equ}
where $\lambda_z = (C|\phi_z|^{-1/\alpha}) \wedge (\frac1{2}\dist{z}_a)$.
\end{theorem}

\begin{remark}
$\disc{(f,Z)}_{z,\lambda_z} $ depends on $\Gamma_{xy}$ only for $x,y\in B_z(\lambda_z)$ and on $\Pi_x$ evaluated on non-anticipative test functions centred at $x\in B_z(\lambda_z)$ but at the \emph{global} scale $1$.
The local dependence on $\Gamma$ is natural (cf. Section \ref{sec:RPs} for rough paths) but
the global dependence on $\Pi$ we believe is an artefact of the way we handle Schauder estimates in our proof (Lemma \ref{lem:integration}).
\end{remark}

\subsection{Examples: linear conditions on \texorpdfstring{$f$}{f}}
\label{sec:linear_cond_SPDE}

In this subsection, we verify Assumption~\ref{as:f_poly} in the case that $f$ behaves like a polynomial and show that the corresponding exponents in $\disc{(f,Z)}_{z,\lambda}$
take a simple form.

For $\tau\in\mfW_{<0}$, let $L(\tau)\geq 1$ denote the number of instances of $\Xi_j$ with $j\in\{1,\ldots,n\}$ in $\tau$ (i.e. the number of `noises').
Define the homogenous model norm
\begin{equ}
\trinorm{Z}_{z,\lambda} = \max_{\tau\in\mfW_{<0}} \|\Pi\|_{\tau;B_z(\lambda)}^{1/L(\tau)}
+
\max_{\CI^k[\sigma]\in\mfU} \|\Gamma\|_{\CI^k[\sigma];B_z(\lambda)}^{1/L(\sigma)}\;.
\end{equ}

Consider a tree $\tau$ of the form \eqref{eq:tree}.
For vertices $v,w$ of $\tau$, we write $v\preceq w$ if $v$ is on the unique path from $w$ to the root $\rho_\tau$
and write $v\prec w$ if in addition $v\neq w$.
We call a tree $\sigma$ a \emph{branch} of $\tau$ if $\sigma$ has a root $\rho_\sigma$, which is a vertex of $\tau$, and contains all vertices $v$ of $\tau$ such that $\rho_\sigma \preceq v$;
all edges and labels of $\sigma$ are inherited from $\tau$.
In other words, $\sigma$ is a branch of $\tau$ if either $\sigma=\tau$ or if $\sigma$ is a branch of $\tau_1,\ldots,$ or $\tau_\ell$.

\begin{lemma}\label{lem:branches}
For any conforming tree $\tau$ and a branch $\sigma$ of $\tau$, one has $|\tau|_\s\geq|\sigma|_\s$ with equality if and only if $\sigma=\tau$.
\end{lemma}

\begin{proof}
Replacing $\sigma$ by a single vertex with noise label $\Xi_1$ keeps the tree conforming, from which, by Lemma \ref{lem:lowest_tree}, it follows that $|\tau|_\s -|\sigma|_\s +\beta \geq \beta$ with equality if and only if $\tau=\sigma$.
\end{proof}

In what follows, we let $f,Z,\phi,\lambda_z$ be as in Theorem \ref{thm:SPDE}.

\subsubsection{\texorpdfstring{$f$}{f} constant}

\begin{corollary}
Suppose that $f$ is constant (which is always the case if $\beta<-2$).
Then
\begin{equ}
|\phi(z)| \lesssim \max\{\trinorm{Z}_{z,\lambda_z}^{\frac{\alpha}{\beta + \alpha +2}}, \dist{x}_a^{-\alpha}\}
\;.
\end{equ}
\end{corollary}

\begin{proof}
Assumption \ref{as:f_poly} is satisfied by Lemma \ref{lem:Upsilon_poly} with $\eta,\bar\eta$ given by \eqref{eq:eta_f_const} (except we take $\bar\eta(\tau,\bk) = 0$ if $D^\bk\Upsilon^\tau$ does not depend on $\nabla\phi$).
Then, for any $|\tau|_\s\in\mfW_{<0}$ and $\bk$,
\begin{equ}
\Delta(\tau,\bk) = |\tau|_\s - \mfc(\tau) = L(\tau)(\beta + \alpha +2)\;.
\end{equ}
The proof follows from Theorem \ref{thm:SPDE}.
\end{proof}

\begin{remark}
In the case that $p=3$ and $P$ does not depend on $\nabla\phi$, this estimate essentially agrees with \cite{MoinatWeber20,CMW23}.
\end{remark}

\subsubsection{\texorpdfstring{$f$}{f} depends only on \texorpdfstring{$\phi$}{phi}}

Let us write $f(\bphi) = f(\phi,\nabla\phi)$ and denote by $D_1 f$ and $D_2 f$ the partial derivatives with respect to $\phi$ and $\nabla\phi$ respectively.

Suppose now that $f$ is non-constant and depends only on $\phi$, i.e. $D_2 f =0$.
Since $f$ is non-constant, we have $\beta \in (-2,0)$.
Let $N=\floor{2/(2+\beta)}$
and suppose that
\begin{equ}
\disc{f}_{\eta} = \max_{l=0,\ldots,N}\sup_{\bphi\in\CJ}\frac{|D_1^l f(\bphi)|}{|\phi|_{\eta-l}}
< \infty
\;,
\end{equ}
where $\eta \geq 0$ and satisfies
\begin{equ}[eq:q_eta]
\alpha (\eta-1) < \beta+2\;.
\end{equ}
Moreover, if $f$ is not a polynomial, we assume that $\eta\geq N$.

\begin{corollary}\label{cor:f_non_const}
Under the above assumptions,
\begin{equ}[eq:f_non_const_bound]
|\phi(z)| \lesssim \max\{ (\disc{f}_{\eta} \trinorm{Z}_{z,\lambda_z})^{\frac{\alpha}{\beta + 2 + \alpha -\alpha \eta}}, \dist{x}_a^{-\alpha}\}
\;.
\end{equ}
\end{corollary}

\begin{remark}
Remark that, if we substitute $2+\beta\rightsquigarrow \gamma$, then condition \eqref{eq:q_eta} matches condition \eqref{eq:RP_linear_assump} with $q=-1$, and the exponent $\frac{\alpha}{\beta + 2 + \alpha -\alpha \eta}$ in \eqref{eq:f_non_const_bound} matches that in \eqref{eq:final_RP_bound}.
\end{remark}

\begin{remark}\label{rem:JP}
\cite[Thm.~4.1]{Jin_Perkowski_25} considers the case above with $p=2$, $f(\bphi)=\phi$, and $\phi$ a positive solution\footnote{Although
Theorem \ref{thm:SPDE}, as written, does not apply in this setting, see Remark \ref{rem:pos} for a trivial fix.} to
$\CL \phi = -\phi^2 +\phi\xi$, where $\xi \in \CC^{\beta}$ with $\beta=-1-\eps$ for $\eps>0$ small.
Taking $\eta=1$, we recover their result restricted to the unit box.\footnote{
\cite[Thm~4.1]{Jin_Perkowski_25} states an a priori estimate that is \emph{uniform} in boxes of diameter $\geq 1$.
Such a bound is useful when the driver is small and the box is large.
While our result does not directly imply this, our method to prove local coercivity by comparing to the `zero-noise' solution can be used to prove such a uniform bound.
Since we find this argument a separate task,
we prefer to avoid further details.}
\end{remark}

For the proof of Corollary \ref{cor:f_non_const}, we require an extra lemma.
For a tree $\tau$ and a vertex $v$ of $\tau$, we call an edge $\{v,w\}$ \emph{outgoing} if $v\prec w$.

\begin{lemma}\label{lem:noise_zero}
Consider $\tau\in\mfW_{<0}$. Then there is at most one vertex $v$ of $\tau$ with noise label $\Xi(v)=0$. Moreover, $v$ has precisely one outgoing edge $e$ for which $|\mfe(e)|_\s=1$.
\end{lemma}

\begin{proof}
Consider a vertex $v$ with $\Xi(v)=0$. Then $v$ can have at most one outgoing edge $e$ with $|\mfe(e)|_\s=1$ since $P$ is affine in $\nabla\phi$ and by the choice of $\mfR_0$ in
\eqref{eq:Rzero}.

Suppose now that $\mfe(e)=0$ for all outgoing edges $e$ of $v$.
Let us write $\sigma_v$ for the unique branch with root $v$.
Then, since $\beta>-2$,
we have $|\sigma_v|_\s>0$,
which contradicts $|\tau|_\s <0$ by Lemma \ref{lem:branches}.
Therefore $v$ has precisely one outgoing edge with $|\mfe(e)|_\s=1$.

Now suppose there is another vertex $w\neq v$ with $\Xi(w)=0$.
Note that $|\sigma_v|_\s > 1+\beta > -1$.
If $w\prec v$, then $|\sigma_w|_\s > 0$ because $|\sigma_v|_\s>-1$, which contradicts $|\tau|_\s<0$ by Lemma \ref{lem:branches}.
Likewise if $v\prec w$.
Suppose now that $v\not\prec w$ and $w\not\prec v$.
We can replace the branches $\sigma_v,\sigma_w$ by single vertices $\Xi_1$ and the resulting tree will be conforming.
Therefore, by Lemma \ref{lem:lowest_tree}, $|\tau|_\s - |\sigma_v|_\s - |\sigma_w|_\s + 2\beta > \beta$,
which implies
\begin{equ}
|\tau|_\s > |\sigma_v|_\s + |\sigma_w|_\s -\beta > 2+\beta > 0
\end{equ}
where we used $|\sigma_v|_\s,|\sigma_w|_\s > 1+\beta$ in the final bound.
This again contradicts $|\tau|_\s<0$.
\end{proof}

\begin{proof}[Proof of Corollary \ref{cor:f_non_const}]

Consider a tree $\tau\in\mfW_{<0}$.
We proceed to verify Assumption \ref{as:f_poly}.

We first claim that, if $\bk \in \mathring B^\tau$, then $\bk=\{0,\ldots,0\}$ contains only $0$ with $|\bk|\leq N-1$; in particular, if $\beta\in(-1,0)$, thus $N=1$, then $\bk = \emptyset$.
Indeed, if $|\bk|_\s > 0$, then $D^\bk\Upsilon^\tau = 0$ by Lemma \ref{lem:noise_zero} and the assumption that $P$ is affine in $\nabla\phi$ and that $f$ does not depend on $\nabla\phi$.
It thus suffices to consider the case that $\bk$ contains only $0$.

Suppose $|\bk|\geq N$.
If $\beta \in (-1,0)$, then $\tau=\Xi_j$ for $j>0$, and $\gamma_0=1$, $\zeta \leq \beta+1$, hence
\begin{equ}
\gamma_{\bk} \geq 1 \geq \zeta-\beta\;,
\end{equ}
thus $\bk\notin \mathring B^\tau$.
If $\beta \in (-2,-1)$, then $\gamma_0 = 2+\beta$.
Note that $N$ is the smallest integer such that
$N(2+\beta)+\beta>0$.
Since there exists $\sigma\in\mfW$ with $|\sigma|_\s = N(2+\beta)+\beta\in (0,1)$, it follows that $\zeta\leq N(2+\beta)+\beta$.
Therefore 
\begin{equ}
\gamma_{\bk} \geq N(2+\beta) \geq \zeta-\beta\;,
\end{equ}
so again $\bk\notin \mathring B^\tau$.
This proves the claim.

To verify Assumption \ref{as:f_poly}, suppose first that $\tau$ has no noise label $\Xi_0$.
Using the shorthand $L=L(\tau)$, we have $|\tau|_\s = (L-1)(\beta+2)+\beta$.
Moreover, for $\bk\in B^\tau$, by the above claim and definition of $\disc{f}_\eta$,
\begin{equ}[eq:DUpsilon1]
|D^{\bk} \Upsilon^\tau(\bphi)| \lesssim \disc{f}_\eta^L |\phi|_{\eta L - (L-1+|\bk|)}
\end{equ}
Set $\nabla\eta(\tau,\bk)=0$ and $\eta(\tau,\bk) = \eta L - (L-1+|\bk|)$.
By $D^\bk\Upsilon^\tau\neq0$ and $|\bk|< N$ and the assumption that $\eta\geq N$ if $f$ is not a polynomial, one has $\eta(\tau,\bk)\geq 0$.
Then condition \eqref{eq:scaling_assump1} is equivalent to
\begin{equ}
\Delta(\tau,\bk) = \alpha(1-|\bk|) + (L-1)(\beta+2)+\beta + 2 - \alpha (\eta L - (L-1+|\bk|))
=
L(\beta+2-
\alpha \eta + \alpha)
>0\;,
\end{equ}
which holds by \eqref{eq:q_eta}.

By Lemma \ref{lem:noise_zero}, it remains to consider the case that $\tau$ has one vertex $v$ with $\Xi(v)=0$ and with one outgoing edge for which $|\mfe(e)|_\s=1$.
Recall by Assumption \ref{as:P} that this can only happen if $p$ is odd and, for any $|k|_\s=1$, $D^k P(\bphi)$
is a homogenous polynomial in $\phi$ of degree $\frac{p-1}{2} = 1/\alpha$.
In this case, $|\tau|_\s = L(\beta+2) - 1$.
Then, for $\bk\in B^\tau$,
\begin{equ}[eq:DUpsilon2]
|D^\bk\Upsilon^\tau(\bphi)| \lesssim \disc{f}_\eta^L |\phi|_{\eta L - (L - 1) + 1/\alpha - |\bk|}
\;.
\end{equ}
Set $\nabla\eta(\tau,\bk)=0$ and $\eta(\tau,\bk) = \eta L - (L - 1) + 1/\alpha-|\bk|$.
Then again $\eta(\tau,\bk)\geq0$ and condition \eqref{eq:scaling_assump1} is equivalent to
\begin{equ}
\Delta(\tau,\bk) = \alpha(1-|\bk|) + L(\beta+2) - 1 + 2 - \alpha (\eta L - (L - 1) + 1/\alpha-|\bk|)
=
L(\beta+2-
\alpha \eta + \alpha)
>0\;,
\end{equ}
which is the same condition as earlier and holds by \eqref{eq:q_eta}.
The bounds \eqref{eq:DUpsilon1}-\eqref{eq:DUpsilon2} imply that,
with the above choices for $\eta(\tau,\bk),\nabla\eta(\tau,\bk)$,
\begin{equ}
\trinorm{D^\bk\Upsilon^\tau} \lesssim \disc{f}_{\eta}^{L}\;.
\end{equ}
Finally, note that condition \eqref{eq:Schauder_assump} is equivalent to $|\tau|_\s + 2 > 0$, which clearly holds due to $|\tau|_\s\geq \beta$ by Lemma \ref{lem:lowest_tree}.
Therefore Assumption \ref{as:f_poly} holds and
Theorem \ref{thm:SPDE} yields the result.
\end{proof}

\begin{remark}\label{rem:general_cond}
Suppose we instead have the more general bound, similar to Section \ref{sec:linear_condition_RP},
\begin{equ}
\disc{f}_{\eta,q} = \sup_{l,\bphi}\frac{|D_1^l f(\bphi)|}{|\phi|_{\eta + l q}}
< \infty
\end{equ}
for some $q\in\R$.
Then, in the first case in the above proof, instead of \eqref{eq:DUpsilon1} one would have
\begin{equ}
|D^\bk \Upsilon^\tau (\bphi)| \lesssim \disc{f}_{\eta,q}^{L} |\phi|_{\eta L + (L-1+|\bk|)q }
\;.
\end{equ}
It is then natural to set $\nabla\eta(\tau,\bk)=0$ and $\eta(\tau,\bk) = \eta L + (L-1+|\bk|)q$ for which
\begin{equ}
\Delta(\tau,\bk)
=L(\beta+2-
\alpha \eta - \alpha q)
-
\alpha|\bk|(q+1)\;,
\end{equ}
and Assumption \ref{as:f_poly} is satisfied if this expression is positive for all $|\bk|,L\leq N$.
Similar considerations apply to the second case in the proof of Corollary \ref{cor:f_non_const}.
However, $\Delta(\tau,\bk)$ would no longer be a multiple of $L$ for $q\neq-1$, so it appears the best bound that Theorem \ref{thm:SPDE} provides in this case, in contrast to Section \ref{sec:linear_condition_RP}, is not simply a power of $\disc{f}_{\eta,q}\trinorm{Z}$.
\end{remark}

\subsubsection{\texorpdfstring{$f$}{f} depends on \texorpdfstring{$\nabla\phi$}{D phi}}

Suppose now that $f$ depends on $\nabla\phi$, which implies $\beta \in (-1,0)$.
Writing $f(\bphi) = f(\phi,\nabla\phi)$, we denote by $D_1 f$ and $D_2 f$ the partial derivatives with respect to $\phi$ and $\nabla\phi$ respectively.
By Example \ref{ex:only_nabla}, unless the equation is classically well-posed ($\beta>-\frac12$), the dependence of $f$ on $\nabla\phi$ must be affine for Assumption \ref{as:f_poly} to hold,
i.e. $D_2^2f=0$,
which we now assume.

Denoting $\theta = \frac{\alpha+1}{\alpha}$, suppose that
\begin{equ}
\disc{f}_{\eta} = \max_{l=0,1}\sup_{\bphi\in\CJ}\frac{|D_1^l f(\bphi)|}{|\phi|_{\eta - l} + |\nabla\phi|_1 |\phi|_{\eta -\theta -l}} +
\max_{l=0,1}\sup_{\bphi\in\CJ}\frac{|D_2 D_1^l f(\bphi)|}{|\phi|_{\eta -\theta -l}}
< \infty
\;,
\end{equ}
where $\eta$ satisfies \eqref{eq:q_eta} and $\eta \geq \theta$.
If $f$ is not a polynomial, we further assume that $\eta\geq \theta+1$.

\begin{corollary}
Under the above assumptions,
\begin{equ}
|\phi(z)| \lesssim \max\{ (\disc{f}_{\eta} \trinorm{Z}_{z,\lambda_z})^{\frac{\alpha}{\beta + 2 + \alpha -\alpha \eta}}, \dist{x}_a^{-\alpha}\}
\;.
\end{equ}
\end{corollary}

\begin{proof}
Consider a tree $\tau\in\mfW_{<0}$.
By Lemma \ref{lem:branches},
all noise labels of $\tau$ must be $\Xi_j$ for $j>0$ and
all derivative labels  must satisfy $|\mfe(e)|_\s=1$ (a noise label $\Xi_0$ or edge $e$ with $\mfe(e)=0$ implies the existence of a branch $\sigma$ with $|\sigma|_\s>0$ because $\beta>-1$).
Since we suppose $f$ is affine in $\nabla\phi$, we have $\Upsilon^\tau \neq 0$ only if $\tau = \Xi_{j_1}\CI^{k_1}[\cdots \CI^{k_{L-1}}[\Xi_{j_{L}}]]$ with $|k_i|_\s=1$ and $j_i\geq 1$ and
where we denote $L=L(\tau)$
i.e. every vertex has degree at most $2$.

Furthermore, one has $\gamma_0 = 1$ and $\zeta\leq\gamma_{k}=1+\beta$ for $|k|_\s=1$.
Therefore $\mathring B^\tau$ consist only of $\bk$ such that $|k|_\s=1$ for all $k\in\bk$.
Since $D_2^2 \Upsilon^\tau=0$, one has $\bk\in B^\tau$ only if
$|\bk|\leq 1$ and $|\bk|_\s\leq 1$.

Note that $|\tau|_\s = (L-1)(\beta+1)+\beta$.
Then for $\bk\in B^\tau$ with $|\bk|=|\bk|_\s = 1$, one has
\begin{equ}[eq:DkUpsilon1]
|D^\bk \Upsilon^\tau (\bphi)| \lesssim \disc{f}_{\eta}^{L} |\phi|_{(\eta -\theta)L}\;.
\end{equ}
Set $\nabla\eta(\tau,\bk)=0$ and $\eta(\tau,\bk) = (\eta -\theta) L$.
Then $\eta(\tau,\bk)\geq 0$ and condition \eqref{eq:scaling_assump1} is equivalent to
\begin{equ}
\Delta(\tau,\bk) =
(L-1)(\beta+1)+\beta + 1
-
(\alpha \eta - \alpha - 1)L
=
L(\beta+2+\alpha - \alpha \eta) > 0\;,
\end{equ}
which holds by the assumption \eqref{eq:q_eta}.

For $\bk\in B^\tau$ with $|\bk|\leq 1$ and $|\bk|_\s=0$, one has
\begin{equ}[eq:DkUpsilon2]
|D^\bk \Upsilon^\tau (\bphi)| \lesssim \disc{f}_{\eta}^{L} |\phi|_{(\eta -\theta)(L-1)}(|\phi|_{\eta -|\bk|} + |\nabla\phi|_1 |\phi|_{\eta -\theta-|\bk|})
\end{equ}
Set $\nabla\eta(\tau,\bk)=1$ and $\eta(\tau,\bk) = (\eta -\theta)(L-1) + \eta -|\bk|$.
Then $\bar\eta(\tau,\bk)=\eta(\tau,\bk)-\theta\geq 0$ due to $D^\bk\Upsilon^\tau\neq0$ and the assumption that $\eta\geq\theta+1$ if $f$ is not a polynomial,
and condition \eqref{eq:scaling_assump1} is equivalent to
\begin{equ}
\Delta(\tau,\bk) = (L-1)(\beta+1)+\beta + 2 +\alpha -\alpha |\bk| - (\alpha \eta - \alpha-1)(L-1) - \alpha \eta + \alpha|\bk|
=
L(\beta + 2 + \alpha - \alpha \eta)
>0\;,
\end{equ}
which is the same condition as earlier and holds by \eqref{eq:q_eta}.
The bounds \eqref{eq:DkUpsilon1}-\eqref{eq:DkUpsilon2} imply that,
with the above choices for $\eta(\tau,\bk),\nabla\eta(\tau,\bk)$,
\begin{equ}
\trinorm{D^\bk\Upsilon^\tau} \lesssim \disc{f}_{\eta}^{L}\;.
\end{equ}
Note also that \eqref{eq:Schauder_assump} holds for the above choices because $|\tau|_\s>-1$.
Therefore Assumption \ref{as:f_poly} holds and
Theorem \ref{thm:SPDE} yields the result.
\end{proof}

\begin{remark}
Similar to Remark \ref{rem:general_cond},
under a natural generalisation of $\disc{f}_{\eta}$ in the spirit of Section \ref{sec:linear_condition_RP},
the best bound that Theorem \ref{thm:SPDE} provides appears not to be a power of $\disc{f}_{\eta,q}\trinorm{Z}$.
\end{remark}

\subsection{Setup for proof}
\label{sec:setup_SPDE}

In the rest of the article, we prove Theorem \ref{thm:SPDE}.
As usual, we apply  Corollary \ref{cor:parameter_choice}, so we put ourselves in the setting of Section \ref{sec:abstract}.
We suppose
Assumptions \ref{as:P}, \ref{as:coercive}, \ref{as:subcrit}, \ref{as:non_integer}, \ref{as:f_poly} are in place in the rest of the paper, but remark that Assumption \ref{as:f_poly} is used only in Sections \ref{sec:remainder_SPDE}-\ref{sec:local_coer_SPDE}
and Assumption \ref{as:coercive} is used only in Section \ref{sec:local_coer_SPDE}.

\textbf{Domains.}
Our basic domain is $\Sigma=\cl\Omega_a$ for $a\in (0,1]$ with parabolic boundary $\d\Omega_a$, metric $(x,y)\mapsto |x-y|_\s$, and distance to boundary $\dist{z}_a$.

Our rescaled domains are $\Sigma_{z,\lambda}=\cl \Omega_b$ where $0<b\ll1$ will be chosen later ($b$ and $a$ are not related in any way).
For $z\in\Omega_a$ and $0<\lambda\leq \dist{z}_a$,
the injection $T_{z,\lambda} \colon \Omega_b\to\Omega_a$ is the map from \eqref{eq:T_lambda_def},
which clearly satisfies \eqref{eq:diameter} since $|z-T_{z,\lambda}x|_\s \leq \lambda b \leq \lambda$ for all $x\in\Omega_b$.

\textbf{Drivers.}
We let $\drivers$ denote the set of pairs $(f,Z)$ where $Z$ is a model on $\Omega_a$ adapted to $K$ (Definition \ref{def:model}) and $f\colon \CJ\to L(\R^n,E)$ satisfies Assumption \ref{as:f_poly}.
We make the same definition for $\drivers_{z,\lambda}$ except that $Z$ is a model on $\Omega_b$
and we require it to be adapted to the rescaled kernel
\begin{equ}[eq:K_lambda_def]
K^\lambda (z) = \lambda^{|\s|-2} K(\lambda\cdot_\s z)
\end{equ}
(i.e. Definition \ref{def:model} holds for $Z=(\Pi,\Gamma)$ but with $K$ replaced by $K^\lambda$.)

\begin{remark}
In Sections \ref{sec:first_applications}-\ref{sec:RPs},
we apply Theorem \ref{thm:a_priori_abstract} with $\drivers_{z,\lambda}=\drivers$.
Here, in contrast, $\drivers_{z,\lambda} \neq \drivers$ in general.
Moreover, while $\drivers_{z,\lambda}$ does not depend on $z$, it \emph{does} depend on $\lambda$.
\end{remark}

Our rescaling maps $R_{z,\lambda} \colon\drivers\to \drivers_{z,\lambda}$ are defined as follows.
For $\lambda>0$, define the map
\begin{equ}[eq:rescale_SPDE]
Q_\lambda \colon \CJ\to\CJ\;,
\qquad
\nabla^k (Q_\lambda \bphi) = \lambda^{-\alpha - |k|_\s}\nabla^k \phi
\;,
\end{equ}
which is a linear bijection.
For $f\colon \CJ \to L(\R^n,E)$, denote
\begin{equ}
f_\lambda = f\circ Q_\lambda \colon \CJ \to L(\R^n,E)\;.
\end{equ}
Recall the critical scaling $\mfc(\tau)$ from Definition \ref{def:mfc}.
Define the linear bijection
\begin{equ}
\mfC_\lambda\colon\CT_{\leq 2}\to\CT_{\leq 2}
\;,
\qquad
\mfC_\lambda\tau = \lambda^{-\mfc(\tau)}\tau
\end{equ}
for all $\tau\in\mfT_{\leq2}$.
For a model $Z = (\Pi,\Gamma)$ on $\Omega_a$, we define for $x,y\in\Omega_b$,
\begin{equs}
(R_{z,\lambda}\Pi)_x \tau
&=
\big(\Pi_{T_{z,\lambda}(x)} \mfC_\lambda \tau \big)\circ T_{z,\lambda}\;,
\label{eq:RPi_def}
\\
(R_{z,\lambda}\Gamma)_{x,y}
&=
\mfC_\lambda^{-1}
\Gamma_{T_{z,\lambda} (x), T_{z,\lambda}(y)}
\mfC_\lambda
\;,
\label{eq:RGamma_def}
\end{equs}
and then $R_{z,\lambda}Z = (R_{z,\lambda}\Pi,R_{z,\lambda}\Gamma)$.
We show in Lemma \ref{lem:model_rescale} that $R_{z,\lambda}Z$ is a model on $\Omega_b$ adapted to $K^\lambda$.
We then define
\begin{equ}[eq:scaling_driver_SPDE]
R_{z,\lambda}\colon \drivers \to \drivers_{z,\lambda}
\;,
\qquad
R_{z,\lambda}(f,Z) = (f_\lambda,R_{z,\lambda}Z)\;.
\end{equ}

To introduce our `norms', we define for $\lambda>0$, $\tau\in\mfW_{<0}$,
and a model $Z = (\Pi,\Gamma)$ on $\Omega_b$, 
\begin{equ}[eq:Pi_tau_lambda]
\|\Pi\|_{\tau;\lambda} = \sup_{x\in\Omega_b}
\|\Pi_x\tau\|_{|\tau|_\s ; x ; 1/\lambda ; r}
\;.
\end{equ}
Note that for $\lambda\ll 1$, $\|\Pi\|_{\tau;\lambda}$ depends on $\Pi_x$ evaluated on non-anticipative test functions with support of diameter order $1/\lambda\gg 1$.
Then define for $z\in \Omega_a$, $0<\lambda\leq \dist{z}$, and $(f,Z)\in\drivers_{z,\lambda}$,
\begin{equ}[eq:fZ_norms]
\|(f,Z)\|_{z,\lambda}
=
\max_{|\tau|_\s < 0}
\max_{\bk \in B^\tau}
\Big(
\trinorm{D^\bk \Upsilon^\tau}\|\Pi\|_{\tau;\lambda}
\Big)^{\rho(\tau,\bk)}
+
\max_{\CI^k[\sigma] \in \mfU}
\Big(
\trinorm{\Upsilon^\sigma}\|\Gamma\|_{\CI^k[\sigma];\Omega_b}
\Big)^{\rho(\sigma,0)}
\end{equ}
where $\rho(\tau,\bk)$ is defined in \eqref{eq:rho_def}.
We define the local norms $\disc{(f,Z)}_{z,\lambda}$ by \eqref{eq:Zf_local_norms}.
Note that \eqref{eq:balls} clearly holds. We show that \eqref{eq:M_norms} holds in Lemma \ref{lem:model_scaling}.

\textbf{Solutions.} For a driver $M=(f,Z) \in \drivers$, the set of solutions $\bbS_M$ is defined as $\phi\in \CC(\cl\Omega_a,E)$ that solve the remainder equation of \eqref{eq:SPDE} on $\Omega_a$ as in Definition \ref{def:solution_SPDE}.
Likewise, for $M=(f,Z)\in\drivers_{z,\lambda}$, the set
$\bbS_M$ consists of $\phi\in \CC(\cl\Omega_b,E)$
that solve the remainder equation of \eqref{eq:SPDE} on $\Omega_b$.
Closure under rescaling \eqref{eq:sol_rescale} is verified in Lemma \ref{lem:sol_scale}.

\textbf{Local coercivity.} This is the hardest part of the proof and is verified in Lemma \ref{lem:local_coer_SPDE}.

\subsection{Scaling}

Let $z\in\Omega_a$, $0<\lambda\leq\dist{z}$, $(f,Z)\in\drivers$,
and $b\in (0,1]$.

\begin{lemma}\label{lem:model_rescale} 
$R_{z,\lambda}Z$ is a model on $\Omega_b$ that is adapted to $K^\lambda$ from \eqref{eq:K_lambda_def}.
Moreover
\begin{equ}[eq:Gamma_scale_bound]
\|R_{z,\lambda}\Gamma\|_{\CI^k[\sigma];\Omega_b} = \lambda^{|\sigma|_\s - \mfc(\sigma)}\|\Gamma\|_{\CI^k[\sigma];B_z(\lambda b)}
\;,
\end{equ}
\begin{equ}[eq:Pi_scale_bound]
\|R_{z,\lambda} \Pi\|_{\tau;\lambda} = \lambda^{|\tau|_\s - \mfc(\tau)} \|\Pi\|_{\tau;B_z(\lambda b)}
\;.
\end{equ}
\end{lemma}

\begin{proof}
We use the shorthands $R = R_{z,\lambda}$ and $T = T_{z,\lambda}$.
We first show \eqref{eq:Gamma_scale_bound}-\eqref{eq:Pi_scale_bound} and \ref{pt:alg}-\ref{pt:poly_model} of Definition~\ref{def:model}.
The algebraic identities
\begin{equ}
(R \Gamma)_{x,y}(R \Gamma)_{y,z} = (R \Gamma)_{x,z}\;,
\qquad
(R \Pi)_x (R\Gamma)_{x,y} = (R\Pi)_y
\end{equ}
follow easily from the definitions, which verifies Definition~\ref{def:model}\ref{pt:alg}.

For the analytic bounds of Definition~\ref{def:model}\ref{pt:bounds}, we have for $\CI^k[\sigma]\in\mfU$ and $x,y\in\Omega_b$,
\begin{equ}
|\scal{(R\Gamma)_{x,y}\CI^k [\sigma],\bone}|
=
\lambda^{-\mfc(\CI^k \sigma)} \scal{\Gamma_{T(x),T(y)}\CI^k [\sigma],\bone}
\end{equ}
and since $|T(x)-T(y)|_\s = \lambda |x-y|_\s$ and $|\CI^k \sigma|_\s-\mfc(\CI^k\sigma) = |\sigma|_\s - \mfc(\sigma)$,
we obtain \eqref{eq:Gamma_scale_bound}.
Furthermore, by \eqref{eq:scale_pointwise},
\begin{equ}
\|(R\Pi)_x \tau \|_{|\tau|_\s;x;1/\lambda;r} = \lambda^{|\tau|_\s - \mfc(\tau)} \|\Pi_{T(x)}\tau\|_{|\tau|_\s;T(x);1;r}
\end{equ}
which proves \eqref{eq:Pi_scale_bound} and verifies Definition~\ref{def:model}\ref{pt:bounds}.
To prove that $R Z$ is a model, it remains to verify Definition~\ref{def:model}\ref{pt:poly_model}.
To this end,
\begin{equ}[eq:RPi_X]
(R\Pi)_x X^k\tau
=
\lambda^{-|k|_\s}
\{\Pi_{T(x)} X^k \mfC_\lambda \tau\}\circ T
=
\lambda^{-|k|_\s} (T(\cdot)-T(x))^k (R\Pi)_x\tau
=
(\cdot-x)^k (R\Pi)_x\tau
\;,
\end{equ}
which verifies the first identity in Definition~\ref{def:model}\ref{pt:poly_model}.
The second identity follows similarly.

Finally, to show adaptedness of $RZ$ to $K^\lambda$, we need to verify \eqref{eq:adapted}.
For $\ell=0$ in \eqref{eq:adapted},
\begin{equ}[eq:adapted_1]
(R\Pi)_x \CI\tau = \lambda^{-\mfc(\tau)-2} (\Pi_{T(x)} \CI\tau )\circ T
=
\lambda^{-\mfc(\tau)-2}
\{
K * \Pi_{T(x)}\tau
-
\Pi_{T(x)}\CJ(T(x))\tau
\}
\circ T
\;,
\end{equ}
where we used the definition of $R$ in the first equality and the adaptedness of $Z$ in the second equality.
Note that, by definition of $R$ and of $K^\lambda$ in \eqref{eq:K_lambda_def},
\begin{equ}[eq:KPi_Tx]
\lambda^{-\mfc(\tau)-2}
\{K * \Pi_{T(x)}\tau\}\circ T
=
K^\lambda * (R\Pi)_x\tau
\;.
\end{equ}
Moreover,
\begin{equs}[eq:RCJ]
\lambda^{-\mfc(\tau)-2}
\{
\Pi_{T(x)}\CJ(T(x))\tau
\}
\circ T
&=
\lambda^{-\mfc(\tau)-2} \sum_{|k|_\s<|\tau|_\s+2}
\frac{\lambda^{|k|_\s} (\cdot-x)^k}{k!}
\{D^k K * \Pi_{T(x)}\tau\}(T(x))
\\
&=
\sum_{|k|_\s<|\tau|_\s+2} \frac{\Pi_x X^k}{k!} 
\{D^k K^\lambda *(R\Pi)_x\tau\}(x)
\\
&= (R\Pi)_x (R\CJ)(x)\tau
\end{equs}
where $R\CJ$ is defined as in \eqref{eq:CJ_def} with $\Pi$ replaced by $R\Pi$ and $K$ replaced by $K^\lambda$, and
where for the first equality we used \eqref{eq:RPi_X} with $\tau=\bone$ (so that $\Pi_x\bone=1$),
while for the second equality we used that $\lambda^{|k|_\s} \{D^k K * \Pi_{T(x)}\tau\}(T(x)) = D^k (\{K * \Pi_{T(x)}\tau\} \circ T)(x)$
followed by \eqref{eq:KPi_Tx}.
Substituting \eqref{eq:KPi_Tx}-\eqref{eq:RCJ} into \eqref{eq:adapted_1} verifies \eqref{eq:adapted} for $R\Pi$ and $K^\lambda$ and $\ell=0$.
The case $\ell\neq0$ follows from
$\Pi_x\CI^\ell \tau = D^\ell \Pi_x\CI\tau$ and thus $(R\Pi)_x \CI^\ell\tau = \lambda^{-\mfc(\tau)-2+|\ell|_\s} (D^\ell \Pi_{T(x)} \CI\tau )\circ T=D^\ell (R\Pi)_x \CI\tau$.
\end{proof}

We let $\Upsilon^\tau_\lambda$ be the elementary differentials built from $f_\lambda$.
Recall the map $Q_\lambda$ from \eqref{eq:rescale_SPDE}.

\begin{lemma}\label{lem:crit_scaling}
$\Upsilon^\tau \circ Q_\lambda = \lambda^{-2-\alpha-\mfc(\tau)} \Upsilon_\lambda^\tau$
for all $\lambda>0$ and conforming trees $\tau$.
\end{lemma}

Note that $-2-\alpha-\mfc(\tau)$ can be both negative or positive.
It is negative if $f$ is constant because $-2-\alpha= \mfc(\Xi_1)$ and every tree satisfies $\mfc(\tau)\geq \mfc(\Xi_1)$.
But for non-constant $f$, there are conforming trees for which $\mfc(\tau)<\mfc(\Xi_1)$.

\begin{proof}
We proceed by induction. For the base cases, one has $\Upsilon^\bone(Q_\lambda \bphi) = P(Q_\lambda\bphi) = \lambda^{- \alpha p} P(\bphi) = \lambda^{-\alpha-2}\Upsilon^{\bone}(\bphi)$ where we used $-\alpha-2 = -\alpha p$.
Moreover, for $i=1,\ldots,n$, by definition $\Upsilon^{\Xi_i}(Q_\lambda \bphi) = (f_\lambda)_i(\bphi) = \Upsilon_\lambda^{\Xi_i}(\bphi)$, which proves the claim since $\mfc(\Xi_i) = -\alpha-2$.

Suppose the claim is true for $\tau$. Then, treating $i\in[d]$ canonically as an element of $\N^d$,
\begin{equs}
(\Upsilon^{X^{i}\tau}\circ Q_\lambda)(\bphi)
&= \sum_{\ell\in\N^d} \lambda^{-\alpha-|\ell|_s-|i|_\s} \bphi_{\ell+i}
(D^\ell \Upsilon^\tau)(Q_\lambda \bphi)
=\sum_{\ell\in\N^d} \lambda^{-|i|_\s} \bphi_{\ell+i}
D^\ell (\Upsilon^\tau \circ Q_\lambda)(\bphi)
\\
&=
\sum_{\ell\in\N^d} \lambda^{-|i|_\s - 2-\alpha-\mfc(\tau)} \bphi_{\ell+i}
D^\ell \Upsilon_\lambda^\tau(\bphi) =
\lambda^{- 2-\alpha-\mfc(X^{i}\tau)} \Upsilon^{X^{i}\tau}_\lambda(\bphi)\;,
\end{equs}
where we used the induction hypothesis in the third equality.
So it suffices to consider $\tau = \Xi_j \prod_{i=1}^\ell \CI^{k_i}[\sigma_i]$.
Note that
\begin{equ}[eq:mfc_relation]
\mfc(\tau) = \mfc(\Xi_j) + \sum_i \mfc(\sigma_i) + 2\ell - \sum_i |k_i|_\s\;.
\end{equ}
Then, if $j \in \{1,\ldots,n\}$,
\begin{equs}
\Upsilon^\tau \circ Q_\lambda(\bphi)
&= \Big[ \Big(\prod_{i} D^{k_i}\Big) \Upsilon^{\Xi_j} \Big](Q_\lambda\bphi)(\Upsilon^{\sigma_1}(Q_\lambda\bphi),\ldots, \Upsilon^{\sigma_\ell}(Q_\lambda\bphi))
\\
&=\lambda^{-\sum_i \mfc(\sigma_i) - 2\ell-\alpha\ell}
\Big[\Big(\prod_{i} D^{k_i}\Big) \Upsilon^{\Xi_j}\Big](Q_\lambda\bphi)(\Upsilon_\lambda^{\sigma_1}(\bphi),\ldots, \Upsilon_\lambda^{\sigma_\ell}(\bphi))
\\
&=\lambda^{-\sum_i \mfc(\sigma_i) - 2\ell-\alpha\ell + \sum_i |k_i|_\s + \alpha\ell}
\Big[\Big(\prod_{i} D^{k_i}\Big) (f_j\circ Q_\lambda)\Big](\bphi)(\Upsilon_\lambda^{\sigma_1}(\bphi),\ldots, \Upsilon_\lambda^{\sigma_\ell}(\bphi))
\\
&= \lambda^{-\mfc(\tau) + \mfc(\Xi_j)} \Upsilon_\lambda^\tau(\bphi)
=\lambda^{-\mfc(\tau) -\alpha-2} \Upsilon_\lambda^\tau(\bphi)
\end{equs}
where we used the induction hypothesis in the second equality and \eqref{eq:mfc_relation} in the fourth equality.
For $j=0$, the same equalities hold except $f_j$ is replaced by $P$ and we use that $P\circ Q_\lambda=\lambda^{-\alpha p}P$ to conclude that the final line becomes $\lambda^{-\mfc(\tau) - \alpha p} \Upsilon_\lambda^\tau(\bphi)
=\lambda^{-\mfc(\tau) -\alpha-2} \Upsilon_\lambda^\tau(\bphi)$.
\end{proof}

The following lemma gives a scaling bound analogous to \eqref{eq:D_ell_infty}.

\begin{lemma}\label{lem:scaling_f}
Under Assumption \ref{as:f_poly},
for all $\tau\in\mfW_{\leq0}$ and $\bk\in B^\tau$,
\begin{equ}
\trinorm{D^\bk \Upsilon_\lambda^\tau}
\leq
\lambda^{\alpha+2 + \mfc(\tau) - |\bk|\alpha - |\bk|_\s - \alpha\eta(\tau,\bk)}
\trinorm{D^\bk \Upsilon^\tau}
\;.
\end{equ}
\end{lemma}

\begin{proof}
By Lemma \ref{lem:crit_scaling}, for $\lambda\in (0,1]$,
\begin{equ}
D^\bk \Upsilon_\lambda^\tau
=
\lambda^{\alpha+2 + \mfc(\tau)} D^\bk (\Upsilon^\tau\circ Q_\lambda)
=
\lambda^{\alpha+2 + \mfc(\tau) - |\bk|\alpha - |\bk|_\s} (D^\bk \Upsilon^\tau) \circ Q_\lambda\;.
\end{equ}
We conclude by \eqref{eq:Upsilon_bound} and \eqref{eq:scaling_assump2} of Assumption \ref{as:f_poly} and that $|rx|_\eta \leq r^\eta|x|_\eta$ for $r\geq 1$.
\end{proof}

\begin{lemma}\label{lem:model_scaling}
$\|R_{z,\lambda}(f,Z)\|_{z,\lambda} \leq \lambda^\alpha \disc{(f,Z)}_{z,\lambda}$.
\end{lemma}

\begin{proof}
This follows by combining Lemmas \ref{lem:model_rescale} and \ref{lem:scaling_f}, definitions \eqref{eq:Zf_local_norms} and \eqref{eq:fZ_norms}, and the fact that $b\leq 1$.
\end{proof}

\begin{lemma}\label{lem:sol_scale}
Suppose $\phi\in\bbS_M$. Recall $R_{z,\lambda}\phi = \lambda^{\alpha}\phi\circ T_{z,\lambda}$.
Then $R_{z,\lambda}\phi \in \bbS_{R_{z,\lambda} M}$.
\end{lemma}

\begin{proof}
Denote again $T = T_{z,\lambda},R = R_{z,\lambda}$.
We extend the bijection $Q_\lambda$ from \eqref{eq:rescale_SPDE}
to $\CT_{\leq 2} \to \CT_{\leq 2}$ by $Q_\lambda\tau = \lambda^{-\alpha-\mfc(\tau)}\tau$
and correspondingly to $\CT_{\leq 2}\otimes E \to \CT_{\leq 2}\otimes E$ by scaling the left tensor.

Define $\Psi \colon\Omega_b\to \CU_{\leq 2}\otimes E$
by
\begin{equ}
\Psi = Q_\lambda^{-1} \Phi\circ T
\;,
\end{equ}
where $\Phi$ is defined as above Definition \ref{def:solution_SPDE}.
Note that the jet $\bpsi = Q_\lambda^{-1}\bphi\circ T$ is the polynomial component of $\Psi$ (up to degree $1$)
and that $\psi = \nabla^0\psi = \lambda^{\alpha}\phi\circ T = R_{z,\lambda}\phi$.
To see that $\Psi$ is coherent in the sense of \eqref{eq:bphi_def} for $\bpsi$ and $R(f,Z)$,
we apply Lemma \ref{lem:crit_scaling} to obtain
\begin{equ}[eq:Upsilon_bpsi]
\Upsilon^\tau_\lambda(\bpsi) = \lambda^{2+\alpha+\mfc(\tau)} \Upsilon^\tau(\bphi\circ T)
\;.
\end{equ}

To show that $\Psi$ is a modelled distribution, for $\sigma\in\mfT_{\leq 2}$,
\begin{equs}
\scal{\sigma,\Psi_x - (R\Gamma)_{xy}\Psi_y}
&=
\Big\langle \lambda^{\alpha+\mfc(\sigma)} \sigma, \Phi_{T(x)}
-
\sum_{|k|_\s\leq 2} \nabla^k\phi_{T(y)} \Gamma_{T(x),T(y)} \frac{X^k}{k!}
\\
&\qquad\qquad-
\sum_{-2<|\tau|_\s<0} \Upsilon^\tau(\bphi(T(y))) \Gamma_{T(x), T(y)}\frac{\CI[\tau]}{\tau!}
\Big\rangle
\\
&=
\scal{\lambda^{\alpha+\mfc(\alpha)}\sigma , \Phi_{T(x)} - \Gamma_{T(x),T(y)} \Phi_{T(y)}}
\end{equs}
where we used in the first equality the definition of $R\Gamma$ in \eqref{eq:RGamma_def} and $Q_\lambda$.
The final quantity is of order $\lambda^{2+\zeta-|\sigma|_\s+\alpha+\mfc(\alpha)}|x-y|_\s^{2+\zeta-|\sigma|_\s}$ since we assumed $\Phi$ is in $\CD^{2+\zeta}(\Omega_a,E)$,
therefore $\Psi$ is in $\CD^{2+\zeta}(\Omega_b,E)$ as desired.

Finally, we need to verify \eqref{eq:sol_def} for $\bpsi$ and $R\Pi$.
Observe that
\begin{equ}[eq:CL_psi]
\CL \psi = \lambda^{\alpha+2} (\CL \phi)\circ T\;.
\end{equ}
On the other hand, for $x\in\Omega_b$,
\begin{equ}[eq:Upsilon_lambda]
\sum_{\tau\in\mfW_{\leq0}} \Upsilon_\lambda^\tau(\bpsi_x)
\frac{1}{\tau!} (R\Pi)_x\tau
=
\sum_{\tau\in\mfW_{\leq0}} \lambda^{2+\alpha} \frac{1}{\tau!}\Upsilon^\tau(\bphi_{T(x)})
\{\Pi_{T(x)}\tau\}\circ T
\end{equ}
where we used the definition of $R\Pi$ in \eqref{eq:RPi_def} and \eqref{eq:Upsilon_bpsi}.
It follows that
\begin{equs}{}
&\Big\|\CL \psi - \sum_{\tau\in\mfW_{\leq0}} \Upsilon_\lambda^\tau(\bpsi_x)
\frac{1}{\tau!} (R\Pi)_x\tau
\Big\|_{\zeta;x;\frac12\dist{x}_b;r}
\\
&\qquad\qquad=
\lambda^{\alpha+2+\zeta} \Big\|\CL \phi - \sum_{\tau\in\mfW_{\leq0}} \Upsilon^\tau(\bphi_{T(x)}) \frac{1}{\tau!} (\Pi_{T(x)} \tau)
\Big\|_{\zeta;T(x);\lambda\frac12\dist{x}_b;r}
<
\infty\;,
\end{equs}
where we used \eqref{eq:CL_psi}-\eqref{eq:Upsilon_lambda} and \eqref{eq:scale_pointwise} for the equality and \eqref{eq:sol_def} and the fact that $\lambda\frac12\dist{x}_b \leq \frac12\dist{T(x)}_a$ due to $\lambda\leq \dist{z}_a$ and $b\leq 1$ for the final bound.
\end{proof}

\subsection{Structure of elementary differentials}
\label{sec:elem_diff}

We now turn to the proof of local coercivity, which takes up the rest of the paper.
The main result of this subsection is Lemma \ref{lem:tree_graft}, which gives a Taylor expansion formula for elementary differentials, similar to~\eqref{eq:R_def} for rough paths.

Recall the inner product \eqref{eq:inner_product}.
An equivalent way to define the inner product (and thus $\tau!$) is via the recursion
\begin{equ}[eq:in_prod_induct]
\bscal{X^k \Xi_j \prod_{i=1}^\ell \CI^{k_i}[\tau_i],
X^{\bar k} \Xi_{\bar j} \prod_{i=1}^\ell \CI^{\bar k_i}[\bar \tau_i]
}
=
\delta_{j,\bar j}
\delta_{k,\bar k} k! \sum_{\pi \in S_\ell} \prod_{i=1}^\ell \delta_{k_i,\bar k_{\pi(i)}}\scal{\tau_i,\bar \tau_{\pi(i)}}
\;,
\end{equ}
where $S_\ell$ is the permutation group on $[\ell]$, and with $\scal{\tau,\bar\tau}=0$ if $\tau,\bar\tau$ have different number of edges incident with their roots.



\begin{lemma}\label{lem:tree_graft}
Let $\Gamma\in \CG$,
and $\tau\in\mfW_{\leq 0}$.
Then
\begin{equs}[eq:tree_graft]{}
&\sum_{\sigma\in\mfW_{\leq 0}}
\frac{1}{\sigma!}
\Upsilon^\sigma\scal{\tau,\Gamma \sigma}
\\
&\quad
=
\sum_{a\in\N^d}\
\sum_{l\geq 0}
\sum_{|\bk|=l}
\sum_{(\sigma_1,\ldots,\sigma_{l})}
\frac{(\d^a D^{\bk} \Upsilon^\tau)}{ a! \bk!}
\big(
\Upsilon^{\sigma_1},\ldots,\Upsilon^{\sigma_{l}}
\big)
\Big\{
\prod_{i=1}^l
\frac{1}{\sigma_i!}
\scal{\Gamma\CI^{k_i}[\sigma_i],\bone}
\Big\}
\scal{\Gamma X,\bone}^a
\end{equs}
where, on the right-hand side, the third sum is over all multi-indexes $\bk = \{k_1,\ldots,k_{l}\}\in \N^{\N^d_{<2}}$
of length $l\geq0$
and the fourth sum is over all $l$-tuples of trees $(\sigma_1,\ldots,\sigma_{l})$ in $\mfW_{<0}$
such that $\CI^{k_i} [\sigma_i]\in\mfU$ and $|\tau|_\s + \sum_{i=1}^{l} |\CI^{k_i} [\sigma_i]|_\s + |a|_\s \leq 0$.
In particular,
\begin{equs}[eq:tree_graft_explicit]{}
\sum_{\sigma\in\mfW_{\leq 0}}
\frac{1}{\sigma!}
\Upsilon^\sigma \scal{\tau,\Gamma \sigma}
=
\sum_{l,b\geq 0}
\sum_{|\bk|=l}
\sum_{|\bm|=b}
\sum_{\substack{(\sigma_1,\ldots,\sigma_l)\\
(a_{1},\ldots,a_b)}}
&
\frac{D^{\bk}D^\bm \Upsilon^\tau}{\bk!\bm!}
\big(
\Upsilon^{\sigma_1},\ldots,\Upsilon^{\sigma_{l}},
\nabla^{m_{1}+a_{1}},\ldots,
\nabla^{m_{b}+a_{b}}
\big)
\\
&\times
\Big\{
\prod_{i=1}^l
\frac{1}{\sigma_i!}
\scal{\Gamma\CI^{k_i}\sigma_i,\bone}
\Big\}
\Big\{
\prod_{i=1}^b \frac{1}{a_i!}\scal{\Gamma X,\bone}^{a_i}
\Big\}
\end{equs}
where $\bm$ runs over all multi-indexes $\bm=\{m_1,\ldots,m_b\}\in\N^{\N^d_{<2}}$ of length $b$, $(a_1,\ldots,a_b)$ runs over all $b$-tuples with $a_i\in\N^d$ and $a_i\neq0$,
and where we again restrict the sum to $|\tau|_\s + \sum_{i=1}^{l} |\CI^{k_i} [\sigma_i]|_\s + \sum_{i=1}^b|a_i|_\s \leq 0$.
Above, we write $\nabla^k$ for the projection map $\nabla^k\colon \bphi \mapsto \nabla^k\phi$.
\end{lemma}

\begin{remark}\label{rem:bk_bm}
Recall the notation from Definition \ref{def:B_tau}.
For every $\bk,\bm$ in \eqref{eq:tree_graft_explicit},
$\gamma_{\bm} \leq |\bm| = b \leq \sum_{i=1}^b |a_i|_\s$,
and, by definition of $\gamma_k$ in \eqref{eq:gamma_k},
$\gamma_{\bk} \leq \sum_{i=1}^l |\CI^{k_i} [\sigma_i]|_\s$.
Therefore $\gamma_{\bk+\bm} \leq -|\tau|_\s$
and thus either $\bk+\bm\in\mathring B^\tau$ or $D^{\bk}D^\bm\Upsilon^\tau=0$.
In particular, the right-hand side of \eqref{eq:tree_graft_explicit} is well-defined for $f$ satisfying the regularity condition in Assumption \ref{as:f_poly}
(similar considerations apply to the right-hand side of \eqref{eq:tree_graft}).

To simplify the proof below, we assume $f$ is smooth so that $\Upsilon^\tau$ is well-defined for any tree $\tau$,
but by continuity, the final equalities \eqref{eq:tree_graft}-\eqref{eq:tree_graft_explicit} hold for any $f$ satisfying Assumption \ref{as:f_poly}.
\end{remark}

\begin{proof}[Proof of $\eqref{eq:tree_graft}\Rightarrow\eqref{eq:tree_graft_explicit}$]
Observe that, for any $\bt\in\R^d$ and smooth function $F$ of $\bphi$,
\begin{equ}[eq:F_jet]
\sum_{a\in\N^d} \frac{\bt^a}{a!} \d^a F(\bphi)
=
\sum_{\bm\in\N^{\N^d}} \frac{D^{\bm}F}{\bm!}(\bphi)
\Big(
\sum_{a\neq 0} \frac{\bt^a}{a!}\d^a \bphi
\Big)^{\bm}
\end{equ}
as an equality between formal power series,
where $\d^a\bphi$ is the jet defined by $[\d^a(\bphi)]_j = \nabla^{a+j}\phi$.
For $m\in\N^d$, note that $(\sum_{a\neq 0} \frac{\bt^a}{a!}\d^a \bphi)^{m} = \sum_{a\neq 0} \frac{\bt^a}{a!}\nabla^{m+a}\phi$.
We then obtain \eqref{eq:tree_graft_explicit} by applying \eqref{eq:F_jet} to a fixed $l\geq0,|\bk|=l,\sigma_1,\ldots,\sigma_l$ in \eqref{eq:tree_graft} and taking $\bt = \scal{\Gamma X,\bone}$
and $F(\bphi) = D^\bk \Upsilon^\tau(\bphi)(\Upsilon^{\sigma_1},\ldots,\Upsilon^{\sigma_l})$
(where we treat $\Upsilon^{\sigma_1},\ldots,\Upsilon^{\sigma_l}$ as fixed vectors not depending on $\bphi$).
The fact that only multi-indexes $\bm\in\N^{\N^d_{<2}}$ remain in the sum is due to the fact that $D^{m}\Upsilon^\tau = 0$ for all $|m|_\s\geq 2$ and $\tau\in\mfW_{\leq0}$ by Lemma \ref{lem:tau_dependence}.
\end{proof}

To conclude the proof of Lemma \ref{lem:tree_graft}, it remains to prove \eqref{eq:tree_graft}.
To this end, it is convenient to introduce several algebraic structures stemming from pre-Lie algebras.
The starting point is the observation from \cite{BCCH21} that $\tau\mapsto\Upsilon^\tau$ is a morphism for (multi-)pre-Lie algebra structures on trees and non-linearities (see also \cite{BCFP19} for the case of rough paths).
The key point is to introduce maps $M$ and $M^+$ below that generalise the pre-Lie products from \cite{BCCH21} and interact well with the map $\Upsilon$.

\begin{definition}
Let $\mfV$ denote the set of trees (in the sense of Definition \ref{def:trees}) and $\CV$ denote the vector space spanned by $\mfV$.
Let $\CH$ denote the free commutative algebra generated by the trees $X^i$ with $i\in[d]$ and by trees of the form $\CI^{k}\sigma$ for $\sigma \in\mfV$ and $k\in\N^d$.
We denote the `empty' product by $\bone^+$ which is the unit under multiplication.
A basis for $\CH$ is given by the set of trees of the form $X^a\CI^{k_1}[\sigma_1]\cdots\CI^{k_r}[\sigma_r]$.
We also let $\CH^*$ denote the space of series of trees in $\CH$.
We identify $\CH^*$ with the dual space of $\CH$ via the inner product $\scal{\cdot,\cdot}$.
\end{definition}

For trees $\sigma,\tau$, define $\sigma \curvearrowright_k \tau$ as the sum of trees obtained by grafting $\sigma$ onto the vertices of $\tau$ with a new edge such that, if the grafting happens at vertex $v$, then we sum over the trees obtained by reducing the polynomial label at $v$ to $\mfn(v)-q$ and setting the derivative label of the new grafted edge to $k-q$ for all $q \leq \mfn(v)$ and then multiplying the resulting tree by $\binom{k}{q}$ -- see \cite[Remark~4.12]{BCCH21} or \cite[Eq.~(4.26)]{ESI_lectures} for a precise formula.
In addition, for $i\in[d]$, define the raising operator $\uparrow_i \tau$ as the sum of trees obtained by increasing the degree of every polynomial label by $i$.

We now define several maps that are related to the dual of the map $\Delta$ from Section \ref{sec:models}. These maps are closely related to those in \cite{Bruned_Manchon_23_deform}; see also \cite{Bruned_Katsetsiadis_23_post_Lie} for a post-Lie algebra perspective.
For completeness, we give all necessary definitions.

We define a bilinear map $M^+\colon\CH^*\times\CH^*\to\CH^*$ as follows.
First, we set
\begin{equ}
M^+(\tau , \bone^+) = 0\;,
\qquad
M^+(\bone^+ , \tau) = \tau
\;,\qquad
M^+(X^j , X^i) = 0
\;,
\qquad
M^+(\CI^k[\sigma] , X^i) = 0\;,
\end{equ}
\begin{equ}
M^+(\CI^k[\sigma], \CI^\ell[\tau]) = \CI^\ell[\sigma \curvearrowright_k \tau]\;,
\qquad
M^+(X^i, \CI^\ell[\tau] ) =\: \CI^\ell[\uparrow_i\tau]\;.
\end{equ}
We then extend $M^+(X^i, \cdot)$ and $M^+(\CI^k[\sigma],\cdot)$
as derivations on $\CH$ for the commutative product in $X^a\CI^{k_1}[\sigma_1]\cdots \CI^{k_r}[\sigma_r]$.
We then define inductively, for $\sigma = \CI^{k_2}[\sigma_2]\cdots \CI^{k_r}[\sigma_r]$,
\begin{equ}
M^+(\CI^{k_1}[\sigma_1]\sigma, \tau)
= M^+(\CI^{k_1}[\sigma_1], M^+(\sigma,\tau))
- M^+(M^+(\CI^{k_1}[\sigma_1],\sigma),\tau)\;,
\end{equ}
which one can interpret as a sum of all possible graftings of the trees $\sigma_1,\ldots,\sigma_k$ using the edges $\CI^{k_1}, \ldots, \CI^{k_r}$ onto $\tau$ (changing polynomial and derivative labels as earlier).
It is easy to verify that $M^+(\CI^{k_1}[\sigma_1]\cdots \CI^{k_r}[\sigma_r], \tau)$ is independent of the order of terms in $\CI^{k_1}[\sigma_1]\cdots \CI^{k_r}[\sigma_r]$.

Finally, for $\sigma = \CI^{k_1}[\sigma_1]\cdots \CI^{k_r}[\sigma_r]$ and $a\neq0$, we take any $i\in a$ and define inductively
\begin{equ}
M^+(X^a\sigma, \tau)
= M^+(X^i , M^+(X^{a-i} \sigma,\tau))
- M^+(M^+(X^i,X^{a-i}\sigma),\tau)\;,
\end{equ}
which one can interpret as raising the polynomial labels in $\tau$ by $a$ \emph{after} grafting $\sigma_1,\ldots,\sigma_k$ with the corresponding edges.
Again, it is easy to verify that $M^+(X^a\sigma, \tau)$ is independent of the choice of $i\in a$ used in the inductive definition.
These definitions extend canonically to the space of series $M\colon \CH^*\times\CH^*\to\CH^*$ since the operations are locally finite.

We further define a bilinear map $M\colon \CH^* \times\CV\to \CV$ as follows.
First, we set
\begin{equ}[eq:M_base]
M(\bone^+,\tau) = \tau\;,
\qquad
M(\CI^k[\sigma],\tau) = \sigma \curvearrowright_k \tau
\;,\qquad
M(X^i,\tau) =\; \uparrow_i\tau\;.
\end{equ}

Then we define inductively for $\sigma = \CI^{k_2}[\sigma_2]\cdots \CI^{k_r}[\sigma_r]$
\begin{equ}[eq:M_trees]
M(\CI^{k_1}[\sigma_1]\sigma, \tau)
= 
M(\CI^{k_1}[\sigma_1], M(\sigma,\tau))
- M(M^+(\CI^{k_1}[\sigma_1],\sigma),\tau)
\end{equ}
and, for $\sigma = \CI^{k_1}[\sigma_1]\cdots \CI^{k_r}[\sigma_r]$ and $a\neq 0$ and any $i\in a$,
\begin{equ}[eq:M_poly]
M(X^a\sigma, \tau)
= M(X^i , M(X^{a-i} \sigma,\tau))
- M(M^+(X^i,X^{a-i}\sigma),\tau)\;.
\end{equ}
As earlier,
$M(X^a \CI^k_1[\sigma_1]\cdots\CI^{k_r}[\sigma_r],\tau)$
can be written as a weighted sum of trees formed by grafting $\sigma_1,\ldots,\sigma_r$ onto $\tau$ using edges $\CI^{k_1},\ldots\CI^{k_r}$ 
and then raising the polynomial labels of the vertices of $\tau$ by $a$.
These operations are locally finite so extend to the space of series as before.

\begin{lemma}\label{lem:Upsilon_morph}
For $\tau\in\mfV$ and $X^a \CI^{k_1}\sigma_1\cdots \CI^{k_r}\sigma_r \in \CH$,
\begin{equ}[eq:pre_Lie_mor]
\Upsilon^{M(X^a \CI^{k_1}\sigma_1\cdots \CI^{k_r}\sigma_r,\tau)} = (\d^a D^{k_1}\cdots D^{k_r} \Upsilon^{\tau})(\Upsilon^{\sigma_1},\ldots,\Upsilon^{\sigma_r})
\;.
\end{equ}
\end{lemma}

\begin{proof}
Recall the pre-Lie morphism property of $\Upsilon$, see, e.g.~\cite[Cor.~4.15]{BCCH21} or~\cite[Eq.~(4.29)]{ESI_lectures},
\begin{equ}[eq:Upsilon_graft_morph]
\Upsilon^{ \sigma \curvearrowright_k \tau} = (D^k\Upsilon^\tau)(\Upsilon^\sigma)\;.
\end{equ}
Moreover, by \cite[Cor.~4.15]{BCCH21} or \cite[Eq.~(4.30)]{ESI_lectures},
\begin{equ}[eq:Upsilon_raise_morph]
\Upsilon^{\uparrow_i \tau} = \d^i \Upsilon^\tau\;.
\end{equ}
Then \eqref{eq:pre_Lie_mor} for $r=0$, and for $r=1$ and $a=0$, follows from these identities.

Proceeding by induction on $r$ first, it follows from the definition of $M$ that
\begin{equ}
\Upsilon^{M(\CI^{k_1}\sigma_1\cdots \CI^{k_r}\sigma_r,\tau)}
=
\Upsilon^{M(\CI^{k_1}\sigma_r,M(\CI^{k_2}\sigma_1\cdots\CI^{k_{r}}\sigma_{r-1},\tau))}
-
\Upsilon^{M(M^+(\CI^{k_1}\sigma_1, \CI^{k_2}\sigma_1\cdots\CI^{k_{r}}\sigma_{r}),\tau)}
\end{equ}
The first term on the right-hand side is, by the induction hypothesis and \eqref{eq:Upsilon_graft_morph}, equal to
\begin{equ}
D^{k_1} \{(D^{k_2}\cdots D^{k_{r}}\Upsilon^{\tau})(\Upsilon^{\sigma_2},\cdots,\Upsilon^{\sigma_{r}})\}(\Upsilon^{\sigma_1})
\;,
\end{equ}
while the second term, by the derivation property of $M^+(\CI^{k_1}\sigma_1,\cdot)$ and the induction hypothesis, is equal to
\begin{equ}
-\sum_{i=2}^{r}
D^{k_2}\cdots D^{k_{r}} \Upsilon^{\tau}(\Upsilon^{\sigma_2},\ldots, D^{k_1}\Upsilon^{\sigma_i}(\Upsilon^{\sigma_1}),\ldots,\Upsilon^{\sigma_{r}})
\;.
\end{equ}
Combining the two terms and applying the chain rule, it follows that
\begin{equ}
\Upsilon^{M(\CI^{k_1}\sigma_1\cdots \CI^{k_r}\sigma_r,\tau)}
=
(D^{k_1}\cdots D^{k_r} \Upsilon^{\tau})(\Upsilon^{\sigma_1},\ldots,\Upsilon^{\sigma_r})
\end{equ}
as desired.
Now we proceed by induction on $a$ and write, for some $i\in a$,
\begin{equ}
\Upsilon^{M(X^a \CI^{k_1}\sigma_1\cdots \CI^{k_r}\sigma_r,\tau)}
=
\Upsilon^{M(X^i,M(X^{a-i}\CI^{k_1}\sigma_1\cdots\CI^{k_{r}}\sigma_{r},\tau))}
-
\Upsilon^{M(M^+(X^i,X^{a-i}\CI^{k_1}\sigma_1\cdots\CI^{k_{r}}\sigma_{r}),\tau)}
\;.
\end{equ}
The first term, by the induction hypothesis and \eqref{eq:Upsilon_raise_morph}, is equal to
\begin{equ}
\d^i\{(\d^{a-i} D^{k_1}\cdots D^{k_r}\Upsilon^{\tau})(\Upsilon^{\sigma_1},\ldots,\Upsilon^{\sigma_k})\}\;,
\end{equ}
while the second term, by the induction hypothesis and the derivation property of $M^+(X^i,\cdot)$, is equal to
\begin{equ}
-\sum_{j=1}^r (\d^{a-i} D^{k_1}\cdots D^{k_r}\Upsilon^{\tau})( \Upsilon^{\sigma_1},\ldots,\d^i\Upsilon^{\sigma_j},\ldots,\Upsilon^{\sigma_r})\;.
\end{equ}
Again by the chain rule (which implies $\d^i$ acts as a derivation),
we obtain \eqref{eq:pre_Lie_mor}.
\end{proof}

Recall the subalgebra $\CH^+\subset\CH$ and the linear map  $\Delta\colon \CT_{\leq 2} \to \CH^+\otimes \CT_{\leq 2}$ from Section \ref{sec:models}.
Let $\CT^*_{\leq 2}$ be a copy of $\CT_{\leq 2}$, which we identify with the dual of $\CT_{\leq 2}$ via the inner product $\scal{\cdot,\cdot}$.
The following result is similar to \cite[Prop.~3.17]{Bruned_Manchon_23_deform}.
For completeness, we give a direct proof.

\begin{lemma}\label{lem:M_Delta}
For trees $\sigma\in\CH^+$, $\tau\in\CT_{\leq 2}^*$, and $\eta\in\CT_{\leq 2}$,
\begin{equ}[eq:M_Delta]
\scal{M(\sigma,\tau),\eta} = \scal{\sigma\otimes\tau,\Delta\eta}
\;.
\end{equ}
\end{lemma}

\begin{proof}
It is not difficult to verify (see \cite[p.~939]{BCCH21} or \cite[Eq.~(4.37)-(4.38)]{ESI_lectures}) that
\begin{equ}[eq:graft_raise]
\scal{\CI^k[\sigma] \otimes \tau , \Delta\eta}= \scal{\sigma \curvearrowright_k \tau,\eta}
\;,\qquad
\scal{X^i\otimes \tau , \Delta\eta}= \scal{\uparrow_i\tau,\eta}
\end{equ}
for $i\in[d]$.
By the base case definition of $M$ \eqref{eq:M_base},
this proves \eqref{eq:M_Delta} for $\sigma=\CI^k[\sigma],X^i\in\CH^+$.

Define now a coproduct $\Delta^+\colon \CH^+ \to \CH^+\otimes\CH^+$ by
\begin{equ}
\Delta^+\bone^+ =\bone^+\otimes\bone^+\;,\qquad \Delta^+ X^i = X^i\otimes\bone^+ + \bone^+\otimes X^i
\end{equ}
for $i\in[d]$ and
\begin{equ}
\Delta^+\CI^k[\sigma] = (\id\otimes \CI^k)\Delta\sigma + \sum_{|\ell|_\s < |\CI^k[\sigma]|_\s} \CI^{k+\ell}[\sigma]\otimes \frac{X^{\ell}}{\ell!}\;,
\end{equ}
and extended multiplicatively to $\CH^+$.
Using the shorthand $(\sigma\otimes\bar\sigma)\Delta^+$ for the linear functional $\scal{\sigma\otimes\bar\sigma,\Delta^+\cdot}\colon\CH^+ \to \R$,
and denoting $\tilde\sigma = \CI^{k_2}[\sigma_2]\cdots \CI^{k_r}[\sigma_r]$,
we claim that
\begin{equ}[eq:Delta_p_ind]
(\CI^{k_1}[\sigma_1]\otimes \tilde\sigma)\Delta^+
=
M^+(\CI^{k_1}[\sigma_1], \tilde\sigma)
+\CI^{k_1}[\sigma_1]\tilde\sigma
\;.
\end{equ}
Indeed, for $\omega=\CI^{k_1}[\sigma_1]\tilde\sigma$,
\begin{equ}
\scal{\CI^{k_1}[\sigma_1]\otimes \tilde\sigma,\Delta^+\omega}
= \sigma_1! \beta_1 \scal{\tilde\sigma,\tilde\sigma}
= \scal{\CI^{k_1}[\sigma_1]\tilde\sigma,\omega}
\;,
\end{equ}
where $\beta_1$ is the number of times $\CI^{k_1}[\sigma_1]$ appears in $\CI^{k_1}[\sigma_1],\ldots, \CI^{k_r}[\sigma_r]$ and where in the final equality we used the definition of the tree factorial  and inner product in \eqref{eq:factorial}-\eqref{eq:inner_product}.

On the other hand, for $\omega$ of the form $\omega = \CI^{k_2}[\omega_2]\cdots\CI^{k_r}[\omega_r]$,
\begin{equs}{}
\scal{\CI^{k_1}[\sigma_1]\otimes\tilde\sigma,\Delta^+\omega}
&= \scal{\CI^{k_1}[\sigma_1]\otimes\tilde\sigma,(\id\otimes\CI^{k_2})[\Delta\omega_2]\cdots(\id\otimes\CI^{k_r})[\Delta\omega_r]}
\\
&=
\sum_{i=2}^r \scal{\CI^{k_1}[\sigma_1]\otimes\tilde\sigma,
(\id\otimes\CI^{k_i})[\Delta\omega_i]\prod_{j\neq i} (\bone^+ \otimes \CI^{k_j}[\omega_j])}
\\
&=
\sum_{i=2}^r \sum_{b=2}^r \scal{\CI^{k_b}[\sigma_1\curvearrowright_{k_1}\sigma_b] , \CI^{k_i}[\omega_i]}
\bscal{\prod_{\ell\neq b} \CI^{k_\ell}[\sigma_\ell],\prod_{j\neq i} \CI^{k_j}[\omega_j]}
\\
&=
\sum_{b=2}^r \scal{\CI^{k_b}[\sigma_1\curvearrowright_{k_1}\sigma_b] \prod_{\ell\neq b}\CI^{k_\ell}[\sigma_\ell],\omega}
=
\scal{M^+(\CI^{k_1}[\sigma_1],\tilde\sigma),\omega}
\end{equs}
where in the third equality we used the inductive definition of the inner product \eqref{eq:in_prod_induct} and $\scal{\CI^{k_1}[\sigma_1]\otimes\CI^{k_j}[\sigma_j],(\id\otimes\CI^{k_i})[\Delta\omega_i]} = \scal{\CI^{k_j}[\sigma_1\curvearrowright_{k_1}\sigma_j], \CI^{k_i}[\omega_i]}$,
and in the fourth equality we used again \eqref{eq:in_prod_induct}.
For $\omega$ not of the form above, $\scal{\CI^{k_1}[\sigma_1]\otimes \tilde\sigma,\Delta^+\omega}$=0.
This proves \eqref{eq:Delta_p_ind}.

Writing $\tilde\sigma = \CI^{k_2}[\sigma_2]\cdots \CI^{k_r}[\sigma_r]$, we now prove \eqref{eq:M_Delta} for $\sigma = \CI^{k_1}\tilde\sigma$ by induction on $r$.
As elements of the dual of $\CT_{\leq 2}$,
\begin{equs}[eq:M_chain]
(\CI^{k_1}[\sigma_1]\tilde\sigma \otimes \tau)\Delta
&=
(\{(\CI^{k_1}[\sigma_1]\otimes \tilde\sigma)\Delta^+ -
M^+(\CI^{k_1}[\sigma_1], \tilde\sigma)
\}\otimes \tau)\Delta
\\
&=
(\CI^{k_1}[\sigma_1]\otimes \tilde\sigma \otimes\tau )(\id\otimes\Delta) \Delta -
(M^+(\CI^{k_1}[\sigma_1], \tilde\sigma)
\otimes \tau)\Delta
\\
&=
\{\CI^{k_1}[\sigma_1]\otimes M(\tilde\sigma,\tau)\}\Delta -
M(M^+(\CI^{k_1}[\sigma_1], \tilde\sigma), \tau)
\\
&=
M(\CI^{k_1}[\sigma_1], M(\tilde\sigma,\tau))
-
M(M^+(\CI^{k_1}[\sigma_1], \tilde\sigma), \tau)
= M(\CI^{k_1}[\sigma_1] \tilde\sigma,\tau)
\end{equs}
where we used \eqref{eq:Delta_p_ind} in the first equality, the coassociativity property $(\id\otimes\Delta)\Delta = (\Delta^+\otimes\id)\Delta$ (see \cite[Thm.~8.16]{Hairer14}) in the second equality, induction in $r$ in the third equality (note for the second term that $M^+(\CI^{k_1}[\sigma_1], \tilde\sigma)$ is a linear combination of trees of the form $\CI^{k_2}[\eta_2]\cdots\CI^{k_r}[\eta_r]$ for which we can apply the induction hypothesis), \eqref{eq:graft_raise} and \eqref{eq:M_base} in the fourth equality, and \eqref{eq:M_trees} in the fifth equality.
This concludes the proof of \eqref{eq:M_Delta} for $\sigma=\CI^{k_1}[\sigma_1]\cdots \CI^{k_r}[\sigma_r]$.

The proof for $\sigma = X^a\tilde\sigma$ where $\tilde\sigma = \CI^{k_1}[\sigma_1]\cdots \CI^{k_r}[\sigma_r]$ and $a\neq0$ is similar.
Indeed, we have
\begin{equ}
(X^i\otimes X^{a-i}\tilde\sigma)\Delta^+
=
M^+(X^i, X^{a-i}\tilde\sigma)
+X^a\tilde\sigma
\;,
\end{equ}
the proof of which is similar to that of \eqref{eq:Delta_p_ind} except that we consider the two cases $\omega= X^a\tilde\sigma$ and $\omega = X^{a-i}\CI^{k_1}[\omega_1]\cdots \CI^{k_r}[\omega_r]$.
The proof of \eqref{eq:M_Delta} for $\sigma = X^a\CI^{k_1}[\sigma_1]\cdots \CI^{k_r}[\sigma_r]$ is then identical to \eqref{eq:M_chain} except we use induction in $a$ instead of $r$ and \eqref{eq:M_poly} instead of \eqref{eq:M_trees}.
\end{proof}

We are finally ready for the

\begin{proof}[Proof of \eqref{eq:tree_graft}]
Define the character $g\in\CH^+\to\R$ by \eqref{eq:g_Gamma}, so that $\Gamma = (g \otimes\id)\Delta$.
We treat $g\in\CH^*$ by $\scal{g,\sigma}=g(\sigma)$ if $\sigma\in\CH^+$ and $\scal{g,\sigma}=0$ if $\sigma$ is a tree that is not in $\CH^+$.
Let $\Gamma^*\colon \CT^*_{\leq 2} \to \CT_{\leq 2}^*$ be the dual map defined by the identity, for all $\sigma\in\CT_{\leq 2}$,
\begin{equ}
\scal{\Gamma^*\tau,\sigma} = \scal{\tau,\Gamma\sigma} = \scal{g\otimes\tau,\Delta\sigma}
\;.
\end{equ}
Then, by Lemma \ref{lem:M_Delta}, for $\sigma\in\CT_{\leq 2}$,
\begin{equ}
\scal{\Gamma^*\tau,\sigma} = \scal{M(g,\tau),\sigma}
\;.
\end{equ}
Let $\pi_{\leq0}\colon\CV\to\CV$ denote the linear map such that, for $\sigma\in\mfV$, $\pi_{\leq 0}\sigma = \sigma$ if $\sigma \in \mfW_{\leq 0}$ and $\pi_{\leq 0}\sigma = 0$ otherwise.
Recall from Remark \ref{rem:dependence} that $\Upsilon^\sigma=0$ if $\sigma$ is not conforming.
Moreover, if $\sigma$ is conforming and has a polynomial leaf, then $|\sigma|_\s>0$ by Lemma \ref{lem:branches}.
Recalling the definition of $\CT_{\leq 2}$ from Definition \ref{def:mfT},
it follows that $\Upsilon^{\pi_{\leq 0} \Gamma^*\tau}
=
\Upsilon^{\pi_{\leq 0} M(g,\tau)}$ and thus
\begin{equ}
\Upsilon^{\pi_{\leq 0} \Gamma^*\tau}
=
\sum_{\sigma} \Upsilon^{M(\sigma,\tau)} \frac{1}{\sigma!}\scal{g,X}^a\prod_{i=1}^l \scal{g,\CI^{k_i}\sigma_i} 
\end{equ}
where the sum is over all trees $\sigma\in\CH^+$ of the form $\sigma = X^a\prod_{i=1}^l \CI^{k_i}\sigma_i$
such that $|\sigma|_\s + |\tau|_\s\leq0$.
It follows from Lemma \ref{lem:Upsilon_morph} that
\begin{equ}
\Upsilon^{\pi_{\leq 0} \Gamma^*\tau}
=
\sum_{a\in\N^d}
\sum_{l\geq 0}
\frac{1}{a!l!}
\sum_{(\sigma_1,\ldots,\sigma_l)}
\sum_{(k_1,\ldots, k_l)}
(\d^a D^{k_1}\cdots D^{k_l} \Upsilon^\tau) (\Upsilon^{\sigma_1},\ldots,\Upsilon^{\sigma_l})
\scal{g, X}^a
\prod_{i=1}^{l}
\frac{1}{\sigma_i!}\scal{g,\CI^{k_i}\sigma_i}
\end{equ}
where the final two sums are over all tuples $(\sigma_1,\ldots,\sigma_l)$ and $(k_1,\ldots,k_l)$ such that
$\CI^{k_i}\sigma_i\in \mfU$ and
$|\tau|_\s + \sum_{i=1}^{l} |\CI^{k_i} \sigma_i|_\s + |a|_\s \leq 0$.
Above, we used \eqref{eq:pre_Lie_mor},
the definition of the tree factorial \eqref{eq:factorial},
and the fact that every tree
$\eta = X^a \prod_{i=1}^r (\CI^{q_i}[\eta_i])^{m_i}$
with distinct $(q_i,\eta_i)$ appears $\binom{l}{m_1,\ldots,m_r}$ times in the sum over the tuples $\sigma_j,k_j$.

We now obtain \eqref{eq:tree_graft} by the relation \eqref{eq:g_Gamma} and remarking that the left-hand side of \eqref{eq:tree_graft} is precisely $\Upsilon^{\pi_{\leq 0} \Gamma^*\tau}$ due to the definition of the inner product,
and by using the fact that every tuple $(k_1,\ldots,k_l)$ appears $\frac{l!}{\bk!}$ times in the sum where $\bk = \{k_1,\ldots,k_l\}\in\N^{\N^d}$.
\end{proof}

The following lemma is a useful consequence of \eqref{eq:pre_Lie_mor}.

\begin{lemma}\label{lem:tau_lower}
Consider $\bk\in \N^{\N^d}$ and a conforming tree $\tau$ such that $D^\bk \Upsilon^\tau\neq 0$.
Then
\begin{equ}[eq:tau_lower]
|\tau|_\s + |\bk| (\beta +2) - |\bk|_\s \geq \beta
\;,
\end{equ}
with equality if and only if $\bk = 0$ and $\tau=\Xi_j$ for some $j>0$, and where we recall $\beta = |\Xi_1|_\s$.
\end{lemma}

\begin{proof}
Since $D^\bk \Upsilon^\tau\neq 0$, one has $D^\bk \Upsilon^\tau(\bphi)(v,\ldots,v)\neq0$ for some vector $v\in E$ and $\bphi\in\widehat\CJ$ (recall \eqref{eq:widehat_J}).
We include another noise label $\Xi_{n+1}$ with corresponding function $f_{n+1}\equiv v$.
Then $M(\prod_{k\in\bk}\CI^k[\Xi_{n+1}], \tau)$ is a sum of trees $\sigma$ with $|\sigma|_\s$ equal to the left-hand side of \eqref{eq:tau_lower}.
Moreover, by \eqref{eq:pre_Lie_mor},
we have $\Upsilon^\sigma \neq 0$ for one of these trees $\sigma$, and therefore $\sigma$ is conforming.
The bound \eqref{eq:tau_lower} therefore follows from Lemma \ref{lem:lowest_tree}.
Moreover, Lemma~\ref{lem:lowest_tree} implies that $|\sigma|_\s = \beta$ if and only if $\sigma = \Xi_j$ for some $j>0$, but this happens if and only if $\tau = \Xi_j$ and $\bk=0$.
\end{proof}


\subsection{Local reconstruction bounds}
\label{sec:local_reconst}

In this subsection, we introduce singular modelled distribution spaces and reconstruction bounds that we use later.
Recall the regularity structure $(\CT_{\leq 2},\CG,\CA)$ from Definition \ref{def:model}.
For $\mfr>-2$, a \emph{sector of regularity $\mfr$} is a graded subspace $\CV=\bigoplus_{a\in \CA}\CV_a$ such that $\CV_a \subset \Span\{\tau\in\mfT_{\leq 2}\,:\, |\tau|_\s=a\}$, $\CV_a = \{0\}$ for $a<\mfr$, and $\Gamma\CV\subset\CV$ for all $\Gamma\in\CG$.
We write $\CV_{<\omega} = \bigoplus_{a<\omega} \CV_a$.

\begin{notation}
In the rest of the section, we fix $0<b\leq 1$ and write $\dist{x} = \dist{x}_b$ for $x\in\Omega_b$. We let $\tau$ be a tree in $\mfT_{\leq 2}$.
\end{notation}

Consider $\omega>0$,\footnote{we
take $\omega=\zeta$ or $\omega=2+\zeta$ in later subsections}
$\eta\in\R$, and a sector $\CV$ of regularity $\mfr\leq 0$.
Fix a model $Z = (\Pi,\Gamma)$ on the open parabolic ball $\Omega_b$.
For a function $F\colon \Omega_b \to \CV_{<\omega}$ we write
\begin{equ}
F_x = \sum_{|\tau|_\s<\omega} F^\tau_x \tau\;.
\end{equ}
Define the (extended) norms\footnote{One should not confuse the norm $\|F^\tau\|_{\infty;\theta}$ in this section with $\|f\|_{\infty;\eta}$ from \eqref{eq:weight_norms}. The latter is used in Sections \ref{sec:first_applications}-\ref{sec:RPs}
and does not appear in this section.}
\begin{equ}[eq:F_sup_def]
\|F\|_{\omega,\eta} = \sup_{|\tau|_\s < \omega} \|F^\tau\|_{\infty;\eta-|\tau|_\s}
\;,
\quad
\text{where}
\;\;
\|F^\tau\|_{\infty;\theta} = \sup_{x\in\Omega_b} (1\wedge\dist{x}^{-\theta})|F^\tau_x|
\;,
\end{equ}
and semi-norm
\begin{equ}
|F|_{\omega,\eta} = \sup_{x\in\Omega_b}\sup_{|\tau|_\s<\omega}
|R^{\tau,F}|_{\tau,\omega,\eta}
\;,
\end{equ}
where we define the two parameter function $R^{\tau,F}\colon\Omega_b^2\to\R$ by
\begin{equ}
\sum_{|\tau|_\s < \omega} R^{\tau,F}_{x,y} \tau = F_y - \Gamma_{yx} F_x
\end{equ}
i.e. $R_{x,y}^{\tau,F} = \frac{1}{\tau!}\scal{F_y - \Gamma_{yx} F_x, \tau}$,
and the norm
\begin{equ}[eq:R_norms_def]
|R|_{\tau,\omega,\eta}
=
\sup_{x\in\Omega_b}
\dist{x}^{\omega-\eta}
|R|_{\omega-|\tau|_\s; B_{x}(\frac12\dist{x})}
\;,
\quad
\text{where}
\;\;
|R|_{\theta; \mfK}
=
\sup_{y,z\in \mfK} \frac{|R_{y,z}|}{|y-z|_\s^{\theta}}\;.
\end{equ}

Let $\CD^{\omega,\eta}(\Omega_b,\CV)$ denote the space of $F\colon\Omega_b\to\CV_{<\omega}$ for which $\|F\|_{\omega,\eta} + |F|_{\omega,\eta}<\infty$.
(This is essentially the space $\CD^{\omega,w}$ from \cite{GH19} with $w=(\eta,\eta,\eta)$.)
We write $\CD^{\omega,\eta}_\mfr(\Omega_b)$ for $\CD^{\omega,\eta}(\Omega_b,\CV)$ whenever $\CV$ is clear from the context.

It follows from the local reconstruction theorem \cite[Thm.~1.1]{Zorin_Kranich_23_Reconst} (see also \cite[Lem.~6.7]{Hairer14}
for a less precise but simpler statement)
that there is a unique distribution $\tilde\CR F\in\CD'(\Omega_b)$ such that, uniformly in $x\in\Omega_b$,
\begin{equ}[eq:local_reconst]
\|\tilde\CR F - \Pi_x F_x\|_{\omega;x;\leq\frac18\dist{x};r} \lesssim \max_{|\tau|_\s<\omega} |R^{\tau,F}|_{\omega-|\tau|_\s; B_{x}(\frac12\dist{x})} \|\Pi\|_{\tau;\Omega_b}
\end{equ}
where $r$ is from Definition \ref{def:model} and we recall the semi-norms \eqref{eq:Cbeta-def} and \eqref{eq:model_size_def}.
The bound \eqref{eq:local_reconst} is the analogue of \eqref{eq:sewing_est}.
The fact that $\|\Pi\|_{\tau;\Omega_b}$ appears on the right-hand side, and not $\|\Pi\|_{\tau;\mfK}$ for some enlargement $\mfK$ of $\Omega_b$ follows from the proof of \cite[Lem.~A.2]{CCHS_2D}.

\begin{remark}
We can localise $\Pi$ around $x$ as well, but this would not improve the final result.
\end{remark}
\begin{lemma}\label{lem:loc_reconst}
Let $\eta\leq \mfr\leq 0<\omega$ with $\eta\notin\Z$ and $\floor{1-\eta}\geq r$ and $F \in \CD^{\omega,\eta}_\mfr(\Omega_b)$.
Then there exists $\CR F\in\CD'(\R^d)$ such that $\supp(\CR F) \subset \cl \Omega_b$, $\scal{\CR F,\psi}=\scal{\tilde\CR F,\psi}$ for all $\psi\in\CC^\infty_c(\Omega_b)$,
and
\begin{equ}
\|\CR F\|_{\CC^{\eta}(\R^d)}
\lesssim
\max_{|\tau|_\s<\omega} \|\Pi\|_{\tau;\Omega_b} \{ |R^{\tau,F}|_{\tau,\omega,\eta}
+
\|F^\tau\|_{\infty;\eta - |\tau|_\s}
\}
\;,
\end{equ}
where the proportionality constant is uniform in $F,Z,b$.
\end{lemma}

\begin{proof}
It follows from \eqref{eq:local_reconst} that
\begin{equ}[eq:recon_global]
\sup_{x\in\Omega_b} \dist{x}^{\omega-\eta} \|\tilde\CR F - \Pi_x F_x\|_{\omega;x;\leq\frac12\dist{x};r}
\lesssim \max_{|\tau|_\s < \omega} |R^{\tau,F}|_{\tau,\omega,\eta}
\|\Pi\|_{\tau;\Omega_b}\;.
\end{equ}
(Note $\leq\frac12\dist{x}$ instead of $\leq\frac18\dist{x}$ on the left-hand side, which is possible by decomposing any $\psi\in\CP^r$ as $\psi=\sum_{i=1}^{N}\psi^{1/4}_i$ where $\psi_i\in\CP^r$ for some finite $N$ depending only on $d$.)

For a distribution $\xi\in\CD'(\Omega_b)$, define
\begin{equ}
\|\xi\|_{\mathring\CC^\eta(\Omega_b)}
=
\sup_{x\in \Omega_b} \sup_{\psi\in\CP^r} \sup_{\delta\in (0,\frac12\dist{x})}
\delta^{-\eta}|\scal{\xi,\psi^\delta_x}|
=
\sup_{x\in \Omega_b} \|\xi\|_{\eta;x;\leq \frac12\dist{x};r}
\;.
\end{equ}
Note that, for $\eta\leq \theta$ and $h\in (0,1]$,
\begin{equ}
\|\xi\|_{\eta;x;\leq h;r} \leq h^{\theta-\eta}\|\xi\|_{\theta;x;\leq h;r}\;.
\end{equ}
Since $\eta\leq \mfr \leq |\tau|_\s$, it follows from \eqref{eq:recon_global} that, uniformly in $x\in\Omega_b$,
\begin{equ}
\|\tilde\CR F\|_{\eta;x;\leq\frac12\dist{x};r}
\lesssim
\dist{x}^{\omega - \eta + (\eta-\omega)} \max_{|\tau|_\s<\omega} |R^{\tau,F}|_{\tau,\omega,\eta}\|\Pi\|_{\tau;\Omega_b}
+
\max_{|\tau|_\s<\omega} \dist{x}^{|\tau|_s-\eta} |F^\tau_x| \, \|\Pi\|_{\tau;\Omega_b}
\;,
\end{equ}
and thus
\begin{equ}
\|\tilde\CR F\|_{\mathring\CC^{\eta}(\Omega_b)}
\lesssim
\max_{|\tau|_\s<\omega} \|\Pi\|_{\tau;\Omega_b} \{ |R^{\tau,F}|_{\tau,\omega,\eta}
+
\|F^\tau\|_{\infty;\eta - |\tau|_\s}
\}
\;.
\end{equ}
The proof follows from Lemma \ref{lem:extension} (this is where we use $\eta\notin\Z$) which implies the existence of a bounded extension operator $\CE\colon \mathring\CC^\eta(\Omega_b) \to \CC^\eta(\R^d)$
once we identify $\Omega_b$ with a shifted and dilated ball $\mfB(0,b)$ and note that $\mathring\CC^\eta$ embeds into $\widehat\CC^\eta$ defined in Appendix \ref{sec:extension}
since $\floor{1-\eta}\geq r$.
\end{proof}

\begin{remark}
The extension $\CR F = \CE(\tilde\CR F)$ is not unique in general.
\end{remark}

\subsection{Integration bounds}
\label{sec:integration}

Fix in this subsection $0<\lambda,b\leq 1$, $\mfr\in(-2,0]$, $\eta\leq \mfr$ as in Lemma \ref{lem:loc_reconst} and $\omega \in (0,1)$.
Consider a model $Z = (\Pi,\Gamma)$ on $\Omega_b$
that is adapted to the rescaled kernel $K^\lambda(z) = \lambda^{|\s|-2}K(\lambda\cdot_\s z)$ from \eqref{eq:K_lambda_def}.
Let $F\in\CD^{\omega,\eta}(\Omega_b,\CV)$,
where $\CV$ is a sector of regularity $\mfr$ on which the abstract integration map $\CI$ is well-defined.\footnote{$\CV$ is allowed to contain polynomials by setting $\CI[X^k]=0$, in agreement with Definition \ref{def:mfT}.
In later sections, we will always take $\CV = \CW_{\leq0}$.}
We define $\CK^\lambda F \in \CD^{\omega+2,\eta+2}_{0}(\Omega_b)$ as the integration operator from \cite[Prop.~6.16]{Hairer14} or \cite[Sec.~4.5]{GH19} for the kernel $K^\lambda$ and with input distribution $\CR F\in\CC^\eta(\R^d)$ from Lemma~\ref{lem:loc_reconst}.
Then, on $\Omega_b$,
\begin{equ}
\tilde \CR \CK^\lambda F = K^\lambda*\CR F\;.
\end{equ}
The following is our main integration bound.
The result is similar to others in the literature, e.g. \cite{Hairer14,GH19,BCZ24_Schauder},
but with the key remark that the bound is uniform in the parameter $\lambda$.

\begin{lemma}\label{lem:integration}
With notation as above, denote
\begin{equ}
\Theta = \max_{|\tau|_\s<\omega}
\|\Pi\|_{\tau;\lambda}
\{
|R^{\tau,F}|_{\tau,\omega,\eta}
+
\|F^\tau\|_{\infty;\eta-|\tau|_\s}
\}
\;,
\end{equ}
where we recall $\|\Pi\|_{\tau;\lambda}$ from \eqref{eq:Pi_tau_lambda}.
Then, uniformly in $\kappa \geq 0$ and $0<\lambda,b\leq 1$,
\begin{equ}[eq:K_F_translation_bound]
|\CK^\lambda F|_{\omega+2,\eta+2-\kappa}
\lesssim
b^{\kappa} \Theta
\end{equ}
and for $l\in\N^d$ with $|l|_\s < \omega+2$, uniformly in $x\in\Omega_b$ and $0<\lambda,b\leq 1$,
\begin{equ}[eq:K_F_sup_bound]
|\nabla^l (\CK^\lambda F)_x|
\lesssim
(\dist{x}^{\eta + 2-|l|_\s} +b^{\eta+2-|l|_\s})
\Theta
\;.
\end{equ}
\end{lemma}

\begin{proof}
The proof is a minor adaptation of \cite[Prop.~6.16]{Hairer14} or \cite[Lem.~4.12]{GH19},
so we only indicate the differences with \cite[Lem.~4.12]{GH19}.
We restrict to $\omega<1$ only because the polynomial sector of our regularity structure is spanned by $X^l$ with $|l|_\s\leq2$.

To prove \eqref{eq:K_F_translation_bound},
it suffices to consider $\kappa=0$ since $\dist{x}^{\eta-\omega}\leq b^\kappa\dist{x}^{\eta-\omega-\kappa}$ for all $\kappa\geq0$ and $x\in\Omega_b$.
Controlling $|\scal{(\CK^\lambda F)_y - \Gamma_{yz}(\CK^\lambda F)_z,\sigma}|$ for $|\sigma|_\s \notin \N$ is simple and is done in the same way as in the proof of \cite[Thm.~5.12]{Hairer14} (we use here Assumption \ref{as:non_integer}).
For $|\sigma|_\s\in\N$, so that $\sigma=X^l$ for some $l\in\N^d$, inspecting the proof of \cite[Lem.~4.12]{GH19}, we have uniformly in $x\in\Omega_b$
\begin{equs}[eq:K_Gamma]
\sup_{y,z\in B_x(\frac12\dist{x})}
\sup_{l<\omega+2}
\frac{|\scal{(\CK^\lambda F)_y - \Gamma_{yz}(\CK^\lambda F)_z,X^l}}{|y-z|_\s^{\omega+2-l}}
&\lesssim
\max_{|\tau|_\s < \omega} |R^{\tau,F}|_{\omega-|\tau|_\s; B_{x}(\frac12\dist{x})}
\|\Pi\|_{\tau;\Omega_b}
\\
&
+
\max_{|\tau|_\s < \omega} \dist{x}^{|\tau|_\s-\omega}\|F^\tau\|_{\infty;B_x(\frac12\dist{x})} \|\Pi\|_{\tau;\lambda}
\\
&+
\dist{x}^{\eta-\omega} \|\CR F\|_{\CC^\eta(\R^d)}
\;. 
\end{equs}
This bound arises from evaluating the various contributions of $K^\lambda_n$ in the decomposition $K^\lambda = \sum_{n\geq0}K_n^\lambda$
in the following way.
Here we recall $K=\sum_{n\geq0} K_n$ from \eqref{eq:K_n_def} and denote $K^\lambda_n(z) = \lambda^{|\s|-2}K_n(\lambda \cdot_\s z)$,
which is supported in $B_0(\lambda^{-1} 2^{-n})$ and satisfies
$\|D^k K^\lambda_n\|_\infty \lesssim (\lambda 2^n)^{|k|_\s+|\s|-2}$ for every $k\in\N^d$.

The first term on the right-hand side of \eqref{eq:K_Gamma} arises from those $n$ such that $\lambda^{-1} 2^{-n}\leq \frac12\dist{x}$, for which the situation is identical to the non-rescaled case and we can use the local reconstruction bound \eqref{eq:local_reconst}
and the argument in \cite{GH19} or \cite{Hairer14}.

The second term arises from bounding the contribution in \cite[Eq.~(4.38)]{GH19} of $\Pi_yF_y$ evaluated on $D^{k+l}K^\lambda_n$ for $\lambda^{-1}2^{-n} > \frac12\dist{x}$.
For this, we recall from \eqref{eq:Pi_tau_lambda} that $\|\Pi\|_{\tau;\lambda}$ bounds the evaluation of $\Pi_y\tau$ on non-anticipative test functions at scales up to $\lambda^{-1}$.

The third term arises in the same way as the second except we evaluate $\CR F$ on $D^{k+l}K^\lambda_n$.
For this, the contribution from $n$ such that $\lambda^{-1} 2^{-n} \leq 1$ is bounded by $\dist{x}^{\eta-\omega}\|\CR F\|_{\CC^\eta(\R^d)}$ in the same way as in \cite{GH19}
by exploiting that $K^\lambda_n$ is essentially a non-anticipative test function at scale $\delta=\lambda^{-1} 2^{-n}$ multiplied by $\delta^{2}$.
For $\lambda^{-1} 2^{-n} > 1$, recall that $\CR F$ is supported in $\cl\Omega_b$, therefore the contribution of each such term is of order
$\|\CR F\|_{\CC^\eta(\R^d)} (\lambda 2^n)^{|k+l|_\s + |\s|-2}$.
Recall here that $|k+l|_\s\geq \omega+2$ and $\omega>0$ and that this term has a prefactor $|y-z|_\s^{|k+l|_\s-\omega-2}$ in \cite[Eq.~(4.38)]{GH19} with $|y-z|\lesssim \dist{x}\leq b$.
In particular, the exponent of $\lambda 2^n$ is positive and the prefactor is bounded above by $1$, so the sum over $\lambda 2^n < 1$ gives a contribution of order
$\|\CR F\|_{\CC^\eta(\R^d)}$.

It follows from \eqref{eq:K_Gamma} that
\begin{equs}
|\CK^\lambda F|_{\omega+2,\eta+2}
&\lesssim
\max_{|\tau|_\s<\omega}
|R^{\tau,F}|_{\tau,\omega,\eta}
\|\Pi\|_{\tau;\Omega_b}
+
\sup_{x\in\Omega_b}
\max_{|\tau|_\s < \omega}
\dist{x}^{|\tau|_\s - \eta}
|F^\tau_x|\,
\|\Pi\|_{\tau;\lambda}
+
\|\CR F\|_{\CC^\eta(\R^d)}
\\
&\lesssim
\max_{|\tau|_\s<\omega}
\|\Pi\|_{\tau;\lambda}
\{
|R^{\tau,F}|_{\tau,\omega,\eta}
+
\|F^\tau\|_{\infty;\eta-|\tau|_\s}
\}
\end{equs}
where we used Lemma \ref{lem:loc_reconst} and the fact that $\eta\leq\mfr$ in the final bound.
This proves \eqref{eq:K_F_translation_bound}.

The bound \eqref{eq:K_F_sup_bound} follows in a similar manner.
Specifically, using that $\eta\leq\mfr$,
one can check that the contribution to $|\scal{\CK^\lambda F(x),X^l}$ for $|l|_\s < 2+\omega$ of every term in the proof of \cite[Lem.~4.12]{GH19} is of order
\begin{equ}[eq:small_scale]
\dist{x}^{\eta+2-|l|_\s} \Theta
\end{equ}
except for $\CR F (D^l K^\lambda_n(x-\cdot))$
appearing in \cite[Eq.~(4.36)]{GH19}
with $\lambda^{-1}2^{-n} > \frac12\dist{x}$.
To bound the contributions of these final terms,
we consider first the scales $\frac12\dist{x}<\lambda^{-1}2^{-n}\leq 2b$, for which we use
\begin{equ}
|\CR F (D^l K^\lambda_n(x-\cdot))|
\lesssim
(\lambda^{-1}2^{-n})^{\eta+2-|l|_\s } \|\CR F\|_{\CC^\eta(\R^d)}\;.
\end{equ}
Since $\eta+2-|l|_\s\notin\Z$ by assumption,
summing over the relevant $n$ gives a contribution of order
\begin{equ}[eq:intermediate_scale]
\|\CR F\|_{\CC^\eta(\R^d)}
(
\dist{x}^{\eta+2-|l|_\s} 
+
b^{\eta+2-|l|_\s}
)
\;.
\end{equ}
For the scales $2b<\lambda^{-1}2^{-n}$, since $\CR F$ is supported on $\cl\Omega_b$, it follows from Lemma \ref{lem:xi_ball} that
\begin{equ}
|\CR F (D^l K^\lambda_n(x-\cdot))|
\lesssim
(\lambda^{-1}2^{-n})^{2-|l|_\s -|\s|} b^{\eta +|\s|} \|\CR F\|_{\CC^\eta(\R^d)}\;.
\end{equ}
The exponent of $\lambda^{-1} 2^{-n}$ is strictly negative since $|\s|\geq 3$ (recall \eqref{eq:parab_scale}),
thus summing over $n$ such that $\lambda^{-1} 2^{-n} > 2b$ gives a contribution of order
$
b^{\eta +2 - |l|_\s}
\|\CR F\|_{\CC^\eta(\R^d)}
$.
Since $\|\CR F\|_{\CC^\eta(\R^d)} \lesssim \Theta$ by Lemma \ref{lem:loc_reconst}, and combining with \eqref{eq:small_scale}-\eqref{eq:intermediate_scale}, the bound \eqref{eq:K_F_sup_bound} follows.
\end{proof}

\begin{remark}
Bounding the contribution of $\CR F$ on large scales is the only part where we require the extension of $\tilde \CR F$.
We suspect that a cleaner treatment would be to use the Dirichlet Green's function of $\CL$ on $\Omega_b$ and leverage that it decays at the parabolic boundary,
but this would require different arguments than in \cite{Hairer14,GH19}.
\end{remark}

\begin{lemma}\label{lem:xi_ball}
Let $\eta\leq 0$ and $\xi\in\CC^\eta(\R^d)$ be supported on $\cl \mfB(0,b)$ (recall \eqref{eq:parabolic_ball}).
Then, uniformly in $x\in\R^d,$ $\psi\in\CB^r$, and $\delta\geq b$,
\begin{equ}
|\xi(\psi^\delta_x)| \lesssim b^{\eta+|\s|}\delta^{-|\s|}\|\xi\|_{\CC^\eta(\R^d)}
\;.
\end{equ}
\end{lemma}

\begin{proof}
Since $\xi$ has support in $\cl \mfB(0,b)$,
we can multiply $\psi^\delta_x$ with a suitable smooth cut-off function $\chi$, which is $1$ on $\mfB(0,2b)$ and $0$ outside $\mfB(0,4b)$, so that $\chi\psi^\delta_x$ is a multiple of a test function at scale $4b$ multiplied by $b^{|\s|}\delta^{-|\s|}$ and $\xi(\chi\psi^\delta_x) =\xi(\psi^\delta_x)$.
\end{proof}

A useful consequence of the proof of Lemma \ref{lem:integration} is the following

\begin{lemma}\label{lem:integration_K_xi}
Let $\eta < 0$ with $\eta\notin\Z$ and $\xi \colon \Omega_b \to \R$ be measurable such that $
\|\xi\|_{\infty;\eta} < \infty
$ where we recall \eqref{eq:F_sup_def}.
Let $\CE (\xi) \in \CC^\eta(\R^d)$ be its extension from Lemma \ref{lem:extension}.
Then, uniformly in $x\in\Omega_b$ and $|l|_\s\leq1$,
\begin{equ}
|D^l K^\lambda*\CE(\xi)(x)| \lesssim (\dist{x}^{\eta+2-|l|_\s} + b^{\eta+2-|l|_\s}) \|\xi\|_{\infty;\eta}\;.
\end{equ}
\end{lemma}

\begin{proof}
Remark that $\|\CE(\xi)\|_{\CC^\eta(\R^d)}\lesssim \|\xi\|_{\mathring\CC^\eta(\Omega_b)} \leq \|\xi\|_{\infty;\eta}$,
so indeed $\CE(\xi)$ is well-defined.

We decompose as earlier $K^\lambda = \sum_{n\geq 0}K^\lambda_n$
where $\|D^l K^\lambda_n\|_\infty \lesssim (\lambda 2^n)^{|l|_\s+|\s|-2}$.
Consider first the scales $\lambda^{-1} 2^{-n} \leq \frac12\dist{x}$
for which one has $|\scal{D^l K^\lambda_n(x-\cdot),\xi}| \lesssim \dist{x}^\eta (\lambda 2^n)^{|l|_\s-2}$.
The exponent of $\lambda 2^n$ is strictly negative, so summing over the relevant $n$ gives a contribution of order $\dist{x}^{\eta-|l|_\s+2}$.
For the remaining scales $\lambda^{-1}2^{-n} > \frac12\dist{x}$,
the contribution is of order $(\dist{x}^{\eta+2-|l|_\s} + b^{\eta+2-|l|_\s})\|\CE(\xi)\|_{\CC^\eta(\R^d)}$
by the same argument as at the end of the proof of Lemma \ref{lem:integration}.
\end{proof}

\subsection{Remainder bounds}
\label{sec:remainder_SPDE}

We now make a choice for $\eta$.
Let $\eta\in (-2,-1)$ such that, for all $\bk\in B^\tau$,
\begin{equs}
\eta
&< \min\CA =  \min\{|\tau|_\s\,:\, \tau\in\mfW_{\leq 0} \} 
\label{eq:eta_def_2}
\;,
\\
\eta
&\leq |\tau|_\s- |\bk|_\s - \nabla\eta(\tau,\bk) 
\label{eq:eta_def_3}
\;,
\end{equs}
where we recall $\nabla\eta(\tau,\bk)\geq0$ from Assumption \ref{as:f_poly}
(such $\eta$ exists by \eqref{eq:Schauder_assump} in Assumption~\ref{as:f_poly}).
Note that $\floor{1-\eta}=2\geq r$ as required in Lemma \ref{lem:loc_reconst}.


In this subsection, we fix $0<\lambda,b\leq 1$
and a model $(\Pi,\Gamma)$ on $\Omega_b$ adapted to $K^\lambda$.
Define
\begin{equ}
\|(f,\Gamma)\| = \max_{\CI^k \sigma \in \mfU}
\trinorm{\Upsilon^{\sigma}}\|\Gamma\|_{\CI^k\sigma;\Omega_b}
\;.
\end{equ}
Let $\phi$ be a solution to the remainder of \eqref{eq:SPDE} in the sense of Definition \ref{def:solution_SPDE}
such that $\|\phi\|_\infty \leq \eps$.
Define $F\colon\Omega_b\to\CW_{\leq 0}\otimes E$ by
\begin{equ}
F_x = \sum_{\tau \in \mfW_{\leq 0}} F^\tau_x\tau\;,
\quad
\text{where}
\;\;
F^\tau_x = \frac{1}{\tau!} \Upsilon^\tau(\bphi_x)\;,
\end{equ}
so that $\scal{F_x,\tau}=\Upsilon^\tau(\bphi_x)$,
and denote
\begin{equ}
R^\tau_{x,y} = R^{\tau,F}_{x,y} = \frac{1}{\tau!}\scal{F_y-\Gamma_{yx}F_x,\tau}
\;,
\end{equ}
so that $\sum_{\tau \in \mfW_{\leq0}} R^\tau_{x,y} \tau = F_y - \Gamma_{yx} F_x$.
Define
\begin{equ}
V = \max_{\tau \in \mfW_{\leq0}}
|R^\tau|_{\tau,\zeta,\eta} \|\Pi\|_{\tau;\lambda}
\;,
\qquad
W = \max_{\tau \in \mfW_{\leq0}} \|F^\tau\|_{\infty;\eta-|\tau|_\s} \|\Pi\|_{\tau;\lambda}
\;,
\end{equ}
where we recall $\zeta$ from \eqref{eq:zeta_def}.

\begin{remark}
For all $\tau \in \mfW_{<0}$, recalling $\Phi$ from \eqref{eq:bphi_def}, one has
\begin{equ}
\scal{F_x - \Gamma_{xy} F_y ,\tau} = \scal{\Phi_x - \Gamma_{xy}\Phi_y,\CI[\tau]}
\;.
\end{equ}
Moreover, $\phi$ and $\nabla\phi$ are $\zeta$-H\"older continuous since $\Phi \in \CD^{2+\zeta}(\Omega_b,E)$ by assumption.
Hence
\begin{equ}
|R^\bone_{x,y}| = |P(\bphi_y) - P(\bphi_x)| \lesssim |x-y|_\s^\zeta
\;.
\end{equ}
Therefore $F \in \CD^{\zeta}(\Omega_b,E)$ and
$V,W < \infty$.
\end{remark}

Recall the set of multi-indexes $B^\tau$ from Definition \ref{def:B_tau}.

\begin{lemma}\label{lem:R_tau_F_bound}
Suppose $\eps,V,W\leq 1$ and $b\leq 1/8$.
Define
\begin{equ}
\Lambda
=
\eps+b^{\eta+2}(V+W + \|(f,\Gamma)\|)
\;.
\end{equ}
Recall the quantities 
$\trinorm{D^\bk \Upsilon^\tau}$ from Assumption \ref{as:f_poly}.
Then, for $\tau\in\mfW_{\leq 0}$,
\begin{equ}[eq:F_estimates]
\|F^\tau\|_{\infty;\eta-|\tau|_\s}
\lesssim
\trinorm{\Upsilon^\tau}
\;,
\qquad
\|F^\bone\|_{\infty;\eta}
\leq
\|F^\bone\|_{\infty;-1}
\lesssim
\eps (\eps + b^{\eta+2}(V+W))
\end{equ}
and
\begin{equ}[eq:R_estimate]
|R^\tau|_{\tau,\zeta,\eta}
\lesssim
\max_{\bk \in B^\tau  \,:\,\bk\neq0}
\trinorm{D^\bk \Upsilon^\tau} \Lambda^{|\bk|}
\;.
\end{equ}
The proportionality constants above do not depend on $b,\lambda,\eps,f,Z,\phi$.
\end{lemma}

\begin{proof}
By Definition \ref{def:solution_SPDE} and \eqref{eq:local_reconst}, we have
$\CL \phi = \tilde \CR F$ in $\Omega_b$.
Moreover, $\CL K^\lambda = \delta_0 + U^\lambda$ where $U^\lambda$ is smooth, compactly supported, and vanishes on $B_0(\lambda^{-1}/2)$.
Define $Y\colon \Omega_b\to E$ by
\begin{equ}[eq:phi_decomp]
\phi = K^\lambda*\CR F + Y\;.
\end{equ}
Then, on $\Omega_b$,
\begin{equ}
\CL Y + U^\lambda*\CR F = 0
\;.
\end{equ}
Since $\CR F$ is supported on $\cl\Omega_b$, which has diameter $b\leq 1/8$,
and $U^\lambda$ is compactly supported outside $B_0(1/2)$, we have $U^\lambda*(\CR F) \equiv 0$ 
on $\Omega_b$ and thus $\CL Y = 0$.

Our conditions on $\eta$ (including \eqref{eq:eta_def_2}) imply that we are in the scope of the integration Lemma \ref{lem:integration}.
Therefore, by the integration bound \eqref{eq:K_F_sup_bound} with $l=0$,
since $\eta+2>0$,
\begin{equ}
\|K^\lambda*\CR F\|_{\infty;\Omega_b}
\lesssim
b^{\eta+2}
\max_{|\tau|_\s<\zeta}
\|\Pi\|_{\tau;\lambda}
\{
|R^{\tau}|_{\tau,\zeta,\eta}
+
\|F^\tau\|_{\infty;\eta-|\tau|_\s}
\}
\leq
b^{\eta+2}
(V+W)
\;.
\end{equ}
Therefore, since $\|\phi\|_{\infty;\Omega_b} \leq \eps$ by assumption,
\begin{equ}
\|Y\|_{\infty;\Omega_b}
\lesssim \eps +
b^{\eta+2}(V + W)
\;.
\end{equ}
Therefore, since $\CL Y=0$ in $\Omega_b$, it follows from classical weighted Schauder estimates \cite[Thm.~4.9,~p.59]{Lieberman_96_Parabolic} that, for any $|k|_\s\leq 2$ and $\kappa\in [0,1)$,
uniformly in $x\in\Omega_b$,
\begin{equ}[eq:Y_bound]
\|D^k Y\|_{\CC^\kappa;B_x(\frac12\dist{x})}
\lesssim
\dist{x}^{-|k|_\s-\kappa}(\eps + b^{\eta+2}(V+W))
\;.
\end{equ}
Moreover, one has the equality between modelled distributions in $\CD^{\zeta+2,\eta+2}_0(\Omega_b)$
\begin{equ}
\Phi = \CK^\lambda F + Y
\;,
\end{equ}
where we recall $\Phi$ from \eqref{eq:bphi_def} and lift $Y$ canonically to the polynomial regularity structure.

Next, for $|k|_\s\leq 2$, by the integration bound \eqref{eq:K_F_sup_bound}
and by \eqref{eq:Y_bound}, we have
\begin{equ}[eq:nabla_phi_bound]
|(\nabla^k \phi)_x|
\lesssim
\dist{x}^{-|k|_\s}
(\eps+b^{\eta+2}(V+W))
\lesssim
\dist{x}^{-|k|_\s}
\;,
\end{equ}
where the final bound follows from our assumption that $\eps,V,W,b\leq 1$.
It follows that
\begin{equ}[eq:Upsilon_bphi_bound]
|D^\bk\Upsilon^\tau(\bphi_x)|
\lesssim
\trinorm{D^\bk\Upsilon^\tau}(1+|\nabla\phi_x|^{\nabla\eta(\tau,\bk)})
\lesssim
\trinorm{D^\bk\Upsilon^\tau}
\dist{x}^{-\nabla\eta(\tau,\bk)}
\lesssim
\trinorm{D^\bk\Upsilon^\tau}
\dist{x}^{\eta-|\tau|_\s + |\bk|_\s}
\end{equ}
where we used \eqref{eq:Upsilon_bound} in
Assumption \ref{as:f_poly} and that $\eta(\tau,\bk),\bar\eta(\tau,\bk)\geq 0$ for the first bound,
\eqref{eq:nabla_phi_bound} and that $\nabla\eta(\tau,\bk)\geq0$ for the second bound,
and the condition \eqref{eq:eta_def_3} on $\eta$ for the third bound.

In particular,
\begin{equ}[eq:f_sigma_bound]
|\Upsilon^\tau(\bphi_x)|
\lesssim
\trinorm{\Upsilon^\tau}
\dist{x}^{\eta-|\tau|_\s}
\lesssim
\trinorm{\Upsilon^\tau}
b^{\eta+2}
\dist{x}^{-|\tau|_\s - 2}
\;.
\end{equ}
Moreover, for $\tau=\bone$, so that $\Upsilon^\bone = P$, using  Assumption \ref{as:P} on $P$ and that $|\phi_x|\leq\eps\leq1$, this can be improved to
\begin{equ}[eq:f1_bound]
|\Upsilon^\bone(\bphi_x)|
\lesssim
|\phi_x|^2+
|\phi_x| \,|\nabla\phi_x|
\lesssim
\dist{x}^{-1}
\eps (\eps+b^{\eta+2}(V+W))
\;.
\end{equ}
We obtain \eqref{eq:F_estimates} from \eqref{eq:f1_bound} and the first bound in \eqref{eq:f_sigma_bound}.

It remains to prove \eqref{eq:R_estimate}.
Fix $\tau\in\mfW_{\leq0}$.
Recall $\mathring B^\tau$ from Definition \ref{def:B_tau}.
Define further
\begin{equ}
\d\mathring B^\tau =
\{\bk \in \N^{\N^d_{<2}}\, :\, \bk\notin \mathring B^\tau \text{ and } \bk\setminus \{k\} \in  \mathring B^\tau \text{ for all } k\in\bk \text{ and } D^\bk\Upsilon^\tau\neq 0\}\;,
\end{equ}
so that $B^\tau = \mathring B^\tau \sqcup \d \mathring B^\tau$.
By the generalised Taylor formula \cite[Prop.~11.1]{Hairer14},\footnote{To apply
\cite[Prop.~11.1]{Hairer14}, we put an arbitrary total order $<$ on $\N^d_{<2}$ such that $k<\bar k$ implies $\gamma_k \leq \gamma_{\bar k}$,
so that our definition of $\d \mathring B^\tau$ coincides with that of \cite[Appendix~A]{Hairer14} modulo multi-indexes $\bk\in\N^{\N^d_{<2}}$ such that $D^\bk\Upsilon^\tau=0$, which can be ignored for our argument.}
\begin{equ}[eq:f_expan_RS]
\Upsilon^\tau(\bphi_y) = \sum_{\bk\in \mathring B^\tau} \frac{D^\bk \Upsilon^\tau}{\bk!}(\bphi_x) (\bphi_y - \bphi_x)^\bk
+
O\big(\max_{\bk\in\d \mathring B^\tau} \|D^{\bk}\Upsilon^\tau\|_{\infty;B_x(\frac12\dist{x})} (\bphi_y-\bphi_x)^{\bk}
\big)
\;,
\end{equ}
where we denote
\begin{equ}
\|D^{\bk}\Upsilon^\tau\|_{\infty;B_x(\frac12\dist{x})}
=
\sup_{z\in B_x(\frac12\dist{x})} |D^{\bk}\Upsilon^\tau(\bphi_z)|
\end{equ}
and for $\bk = \{k_1,\ldots,k_{|\bk|}\} \in \N^{\N^d_{<2}}$
and $\bphi = \sum_{|l|_\s\leq 2} \frac{1}{l!} (\nabla^l \phi) X^l \in \bar\CT_{\leq 2}\otimes E$,
we write $\bphi^{\bk} = (\nabla^{k_1}\phi , \ldots, \nabla^{k_{|\bk|}} \phi)$,
which we treat as a tensor in $E^{\otimes|\bk|}$.

For $|k|_\s<2$,
define the remainder
\begin{equ}
\nabla^k Q_{x,y} = \scal{\cD^k\bphi_y - \Gamma_{yx}\cD^k\bphi_x, \bone}
\end{equ}
where
\begin{equ}
\cD^k\colon
\CI[\CW_{< 0}]\oplus \bar\CT_{\leq 2} \to \CI^k[\CW_{< 0}]\oplus \bar\CT_{\leq 2-|k|_\s}
\end{equ}
is the linear map for which $\cD^k \CI[\tau] = \CI^k[\tau]$ and $\cD^k X^m = \frac{m!}{(m-k)!} X^{m-k}$ if $m_i\geq k_i$ for all $i\in[d]$ and $\cD^kX^m=0$ otherwise.
Then
\begin{equ}[eq:Q_def]
\nabla^k Q_{x,y}
= \nabla^k \phi_y - \nabla^k \phi_x
- \sum_{\CI^k[\sigma]\in\mfU} \frac{1}{\sigma!}\Upsilon^\sigma(\bphi_x) \scal{\Gamma_{yx} \CI^{k}[\sigma],\bone}
- \sum_{0<|l|_\s \leq \zeta + 2-|k|_\s} \frac{1}{l!}(\nabla^{k+l}\phi)_x (y-x)^{l}
\;.
\end{equ}

Recall that we fixed $\tau \in \mfW_{\leq 0}$.
Then
\begin{equ}[eq:R_tau]
R^\tau_{x,y}
=
\frac{1}{\tau!} \scal{F_y - \Gamma_{yx}F_x,\tau}
=
\frac{1}{\tau!}
\Big(
\Upsilon^\tau(\bphi_y)
- \sum_{\sigma\in\mfW_{\leq0}}\frac{1}{\sigma!}\Upsilon^\sigma(\bphi_x)\scal{\tau,\Gamma_{yx} \sigma}
\Big)\;.
\end{equ}
Recall that \eqref{eq:tree_graft_explicit} in Lemma~\ref{lem:tree_graft} gives another expression for the final sum.
Substituting into \eqref{eq:f_expan_RS} the expression for $\nabla^k\phi_y - \nabla^k\phi_x$ that we obtain by rearranging \eqref{eq:Q_def},
we see that the terms arising from \eqref{eq:tree_graft_explicit} cancel the lowest order terms of $\Upsilon^\tau(\bphi_y)$ in \eqref{eq:f_expan_RS}
(for this, recall by Remark \ref{rem:bk_bm} that $\bk+\bm\in \mathring B^{\tau}$ for all $\bk,\bm$ appearing in \eqref{eq:tree_graft_explicit}).
Therefore, for $y\in B_x(\frac12\dist{x})$, $R^\tau_{x,y}$ is made up of a term of order
\begin{equ}[eq:first_error]
\max_{\bk \in \d \mathring B^\tau} \|D^{\bk}\Upsilon^\tau\|_{\infty;B_x(\frac12\dist{x})}
|(\bphi_y - \bphi_x)^\bk|
\end{equ}
plus a finite sum of terms which are multiples of
\begin{equ}[eq:second_error]
D^\bk \Upsilon^\tau(\bphi_x)
\Big(
\prod_{1\leq i \leq q} \nabla^{k_i} Q_{x,y}
\Big)
\Big(
\prod_{q<i \leq j} \Upsilon^{\tau_i}(\bphi_x)\scal{\Gamma_{yx} \CI^{k_i}\tau_i,\bone}
\Big)
\Big(
\prod_{j<i\leq |\bk|}
(\nabla^{k_i+l_i}\phi)_x
(y-x)^{|l_i|_\s}
\Big)
\end{equ}
where $\bk = \{k_1,\ldots,k_{|\bk|}\} \in \mathring B^\tau$ with $\bk\neq0$, $0\leq q\leq j\leq |\bk|$, $\CI^{k_i}\tau_i\in\mfU$, and $l_i \in\N^d$ subject to $0<|l_i|_\s \leq 2-|k_i|_\s$ and the following conditions:
\begin{itemize}
  \item $q\geq 1$, or
  \item $q=0$ and $|\tau|_\s + \sum_{0 < i\leq j} |\CI^{k_i} \tau_i|_\s + \sum_{j<i\leq|\bk|}|l_i|_\s > 0$.
\end{itemize}
We proceed to bound \eqref{eq:first_error}-\eqref{eq:second_error}.

First, for $|k|_\s<2$, by the fact that $\cD^k$ commutes with $\Gamma_{yx}$, by the integration bound \eqref{eq:K_F_translation_bound} with $\kappa = \eta+2$, and by \eqref{eq:Y_bound}, we obtain for $y\in B_x(\frac12\dist{x})$
\begin{equ}[eq:Q_bound_SPDE]
|\nabla^k Q_{x,y}|
\lesssim
\dist{x}^{-\zeta-2} |x-y|_\s^{\zeta+2-|k|_\s}
(\eps+b^{\eta+2}(V+W))
\;.
\end{equ}
Then, by \eqref{eq:Q_def}, \eqref{eq:Q_bound_SPDE},
\eqref{eq:f_sigma_bound}, and \eqref{eq:nabla_phi_bound},
for $|k|_\s<2$,
\begin{equs}[eq:nabla_phi_diff]
|\nabla^k\phi_y - \nabla^k\phi_x|
&\lesssim
\dist{x}^{-\zeta-2}|x-y|^{\zeta+2-|k|_\s}
(\eps + b^{\eta+2}(V + W))
\\
&\qquad+
\max_{\CI^k[\sigma] \in \mfU} b^{\eta+2}\dist{x}^{-|\sigma|_\s - 2} |x-y|_\s^{|\sigma|_\s+2 - |k|_\s}
\trinorm{\Upsilon^\sigma}\|\Gamma\|_{\CI^k[\sigma]}
\\
&\qquad+
\max_{0<|l|_\s<\zeta+2 - |k|_\s}
\dist{x}^{-|k|_\s-|l|_\s}
|x-y|^{|l|_\s}
(\eps + b^{\eta+2}(V + W))
\\
&\lesssim
\big\{\eps+b^{\eta+2}(V+W + \|(f,\Gamma)\|)
\big\}
\dist{x}^{-|k|_\s-\gamma_k} |x-y|_\s^{\gamma_k}
\\
&=
\Lambda
\dist{x}^{-|k|_\s-\gamma_k} |x-y|_\s^{\gamma_k}
\end{equs}
where in the second bound we used that $\zeta+2-|k|_\s, |l|_\s \geq 1 \geq \gamma_k$ 
and that 
$|\sigma|_\s+2-|k|_\s = |\CI^k[\sigma]|_\s>0$ implies $|\CI^k[\sigma]|_\s\geq \gamma_k$ by definition of $\gamma_k$ from \eqref{eq:gamma_k}.

We can now bound the first remainder term \eqref{eq:first_error}.
By \eqref{eq:Upsilon_bphi_bound} and \eqref{eq:nabla_phi_diff},
for $\bk \in\d \mathring B^\tau$,
\begin{equs}[eq:DkUpsilon]
\|D^{\bk}\Upsilon^\tau\|_{\infty;B_x(\frac12\dist{x})}
|(\bphi_y - \bphi_x)^\bk|
&\lesssim
\trinorm{D^\bk \Upsilon^\tau}
\dist{x}^{\eta -|\tau|_\s + |\bk|_\s}
|(\bphi_y - \bphi_x)^\bk|
\\
&\lesssim
\trinorm{D^\bk \Upsilon^\tau}
\Lambda^{|\bk|}
\dist{x}^{\eta-|\tau|_\s-\gamma_\bk} |x-y|_\s^{\gamma_\bk}
\\
&\leq
\trinorm{D^\bk \Upsilon^\tau}
\Lambda^{|\bk|}
\dist{x}^{\eta-\zeta} |x-y|_\s^{\zeta-|\tau|_\s}
\end{equs}
where we used that $\gamma_\bk \geq \zeta-|\tau|_\s$ by definition of $\d \mathring B^\tau$.

Turning to the terms \eqref{eq:second_error}, suppose first that
$q\geq 1$.
Then, by \eqref{eq:Q_bound_SPDE}, \eqref{eq:f_sigma_bound}, and \eqref{eq:nabla_phi_bound},
\begin{equs}[eq:qgeq1]
\eqref{eq:second_error}
&\lesssim
|D^\bk \Upsilon^\tau (\bphi_x)|
\Big\{
\prod_{1\leq i\leq q}
\dist{x}^{-\zeta-2}
|x-y|_\s^{\zeta+2-|k_i|_\s}
(\eps+b^{\eta+2}(V+W))
\Big\}
\\
&\quad
\times
\Big\{
\prod_{q<i\leq j}
\trinorm{\Upsilon^{\tau_i}}
b^{\eta+2} \dist{x}^{-|\tau_i|_\s - 2}
\|\Gamma\|_{\CI^{k_i}[\tau_i]}
|x-y|^{|\tau_i|_\s + 2 - |k_i|_\s}_\s
\Big\}
\\
&\quad
\times
\Big\{
\prod_{j<i\leq|\bk|}
\dist{x}^{- |k_i|_\s-|l_i|_\s}
(\eps+b^{\eta+2}(V+W))
|y-x|^{|l_i|_\s}_\s
\Big\}
\;.
\end{equs}
By Lemma \ref{lem:tau_lower}, since $D^\bk \Upsilon^\tau\neq 0$, 
one has $|\tau|_\s -\beta + \sum_{k\in\bk} \{\beta +2 - |k_i|_\s\} \geq 0$,
so in particular $2-|k_1|_\s\geq -|\tau|_\s$ (recall that $\bk\neq0$).
Furthermore,
the exponents of $\dist{x}$ and $|x-y|_\s$ in the three parentheses in \eqref{eq:qgeq1} sum to $-|\bk|_\s$ and, since $q\geq 1$, the total exponent of $|x-y|_\s$ is at least $\zeta+2-|k_1|_\s \geq \zeta - |\tau|_s$.
Since $|x-y|_\s\leq \dist{x}$, it follows from \eqref{eq:Upsilon_bphi_bound} that
\begin{equs}[eq:second_error_bound]
\eqref{eq:second_error}
&\lesssim
\trinorm{D^\bk \Upsilon^\tau}
\dist{x}^{\eta - \zeta}
|x-y|_\s^{\zeta-|\tau|_\s}
(\eps + b^{\eta+2}(V+W))^{q+|\bk|-j}
(b^{\eta+2}\|(f,\Gamma)\|)^{j-q}
\\
&\leq
\dist{x}^{\eta - \zeta}
|x-y|_\s^{\zeta-|\tau|_\s}
\trinorm{D^\bk \Upsilon^\tau}
\Lambda^{|\bk|}
\;.
\end{equs}

Suppose now that $q=0$.
Since $D^\bk \Upsilon^\tau \neq 0$, by the same reasoning as in the proof of Lemma \ref{lem:tau_lower}, there exists a conforming tree $\sigma\in \mfW$ formed by grafting the family of trees $\{\CI^{k_i}[\tau_i]\}_{0<i\leq j}$
onto vertices of $\tau$,
and for which $|\sigma|_\s = |\tau|_\s + \sum_{0 < i\leq j} |\CI^{k_i} [\tau_i]|_\s$.
Therefore, by our choice of $\zeta\in(0,1)$ in \eqref{eq:zeta_def} and the facts that $|l_i|_\s\in\N\setminus\{0\}$ and $|\bk|\neq0$,
\begin{equ}[eq:lower_zeta]
|\tau|_\s + \sum_{0 < i\leq j} |\CI^{k_i} [\tau_i]|_\s + \sum_{j<i\leq|\bk|}|l_i|_\s > 0
\quad
\text{implies}
\quad
|\tau|_\s + \sum_{0 < i\leq j} |\CI^{k_i} [\tau_i]|_\s + \sum_{j<i\leq|\bk|}|l_i|_\s \geq \zeta
\;.
\end{equ}
Then we obtain the same bound \eqref{eq:second_error_bound} in a similar way as for the case $q\geq 1$. 

In conclusion,
$|R^\tau|_{\tau,\zeta,\eta}
\lesssim
\max_{\bk\in B^\tau \,:\,\bk\neq0}
\trinorm{D^\bk \Upsilon^\tau} \Lambda^{|\bk|}
$, which is precisely  \eqref{eq:R_estimate}.
\end{proof}

\begin{remark}[On optimality]\label{rem:optimality}
It is possible to bound every term in \eqref{eq:R_tau} directly, which requires fewer derivatives of $\Upsilon^\sigma$ but comes with lower powers of $|x-y|_\s$.
Interpolating with our current bounds yields an estimate on $R^\tau$ with a power of $D^\bk\Upsilon^\tau$ that is slightly lower for $|\bk|>0$,
and thus to slightly better exponents in the final result of Theorem \ref{thm:SPDE}.
However, the gain would vanish as the degree of the smallest positive tree approaches $0$ (corresponding to $\gamma \uparrow N+1$ for rough paths).
It is not clear if this strategy is optimal and, for sake of clarity,
we prefer to present only the simpler method above.
\end{remark}

Define now
\begin{equ}
V^* = \max_{\tau\in\mfW_{<0}}
|R^\tau|_{\tau,\zeta,\eta} \|\Pi\|_{\tau;\lambda}
\;,
\qquad
W^* = \max_{\tau\in\mfW_{<0}} \|F^\tau\|_{\infty;\eta-|\tau|_\s} \|\Pi\|_{\tau;\lambda}
\;,
\end{equ}
\begin{equ}
\|(f,\Pi)\| =
\max_{|\tau|_\s<0} \|\Pi\|_{\tau;\lambda} \trinorm{\Upsilon^\tau} \;,
\qquad
\trinorm{(f,\Pi)}
=
\max_{|\tau|_\s < 0}
\max_{\bk\in B^\tau\,:\,\bk\neq0}
\|\Pi\|_{\tau;\lambda}
\trinorm{D^\bk \Upsilon^\tau}
\;,
\end{equ}
\begin{equ}
\|(f,Z)\|
=
\|(f,\Gamma)\|
+
\|(f,\Pi)\|
+
\trinorm{(f,\Pi)}
\;.
\end{equ}

\begin{lemma}\label{lem:VW_bound}
There exists $\tilde\delta>0$ such that, if
$\|(f,Z)\|, b,\eps<\tilde\delta$,
then
\begin{equ}[eq:VW_bound]
V+W \leq
1/2
\;,
\end{equ}
\begin{equ}[eq:VW_star_bound]
V^* \leq
\trinorm{(f,\Pi)}
\;,
\qquad
W^* \lesssim \|(f,\Pi)\|\;,
\end{equ}
where the proportionality constants do not depend on $b,\lambda,\eps,f,Z,\phi$.
\end{lemma}

\begin{proof}
Suppose $b,\eps,\|(f,\Gamma)\|,\trinorm{(f,\Pi)}\leq \tilde\delta \ll 1$.
Suppose also first that $V+W \leq 1$.
Recall $\Lambda
=
\eps+b^{\eta+2}(V+W + \|(f,\Gamma)\|)$.
Multiplying \eqref{eq:R_estimate} by $\|\Pi\|_{\tau;\lambda}$ and taking the max over $\tau$, we obtain by Lemma \ref{lem:R_tau_F_bound} that, for $\tilde\delta>0$ sufficiently small,
\begin{equ}[eq:V_bound]
V \leq C \max_{\tau \in \mfW_{\leq0}}
\max_{\bk \in B^\tau \,:\,\bk\neq0}
\|\Pi\|_{\tau;\lambda}
\trinorm{D^\bk \Upsilon^\tau}
\Lambda^{|\bk|}
\leq 1/4
\;,
\end{equ}
where $C>0$ does not depend on $b,Z,f,\lambda,\eps,\phi$.
If in addition $\|(f,\Pi)\|\leq \tilde\delta$, then, by \eqref{eq:F_estimates},
\begin{equ}
W \leq
C\eps(\eps+ b^{\eta+2}(V+W))
+
C\|(f,\Pi)\|
\leq
1/4
\;.
\end{equ}
This proves \eqref{eq:VW_bound} under the assumption that $V+W\leq 1$.

On the other hand, remark that if $\phi$ is a solution for $b>0$, then it is also a solution for every $\bar b<b$.
Let $V^b,W^b$ be the corresponding quantities as functions of $b$.
Then $V^b,W^b$ are continuous and increasing in $b$ since the coefficients $\bphi_x$ are continuous in $x$
(to see this for $V^b$, we make use of the weight $\dist{x}^{\zeta-\eta}$ in the definition \eqref{eq:R_norms_def} of $|R|_{\tau,\zeta,\eta}$).
Moreover, since $\eta-|\tau|_\s<0$ for every $\tau\in\mfW_{\leq0}$ due to condition \eqref{eq:eta_def_2},
we have $V^b,W^b\to0$ as $b\to 0$.
Therefore, for all $b$ sufficiently small, possibly depending on $\phi$,
we can ensure that $V^b+W^b\leq 1/2$.
We thus obtain \eqref{eq:VW_bound} by continuity in $b$ without the prior assumption that $V+W\leq 1$.

The bound on $W^*$ follows directly from the first bound in \eqref{eq:F_estimates}.
To bound $V^*$, note that the first inequality in \eqref{eq:V_bound} holds with $V$ replaced by $V^*$ and where the first maximum is taken over $\tau \in\mfW_{<0}$.
The bound on $V^*$ then follows by taking $\tilde\delta$ small so that $\Lambda\ll 1$.
\end{proof}

\subsection{Local coercivity}
\label{sec:local_coer_SPDE}

We now verify the local coercivity condition of Corollary \ref{cor:parameter_choice}, which concludes the proof of Theorem \ref{thm:SPDE}.
Recall the notation from Section \ref{sec:setup_SPDE}.

\begin{lemma}[Local coercivity]\label{lem:local_coer_SPDE}
There exists $b,r,\eps,\delta>0$ sufficiently small such that,
if $M=(f,Z)\in \drivers_{z,\lambda}$ with $\|M\|_{z,\lambda} \leq r$ and $\phi\in\bbS_{M}$ with $\|\phi\|_{\infty;\Omega_b}\leq \eps$, then $|\phi(0)|<\eps-\delta$.
\end{lemma}

\begin{proof}
We use notation from Sections \ref{sec:local_reconst}-\ref{sec:remainder_SPDE}.
Let $\tilde\delta>0$ be as in Lemma~\ref{lem:VW_bound}.
We take $b,\eps,r\leq \tilde\delta$ sufficiently small and $M=(f,Z)\in\drivers_{z,\lambda}$ such that
$
\|M\|_{z,\lambda}
\leq r
$.
Let $\phi\in\bbS_{M}$ with $\|\phi\|_{\infty;\Omega_b}\leq\eps$.
Below, all implicit proportionality constants do not depend on $b,\eps,r,\lambda,M,\phi$.

By Assumption \ref{as:coercive} and a scaling argument, there exist $\delta>0$ and
$\mathring\phi\in\CC(\cl\Omega_{b},E)$ such that $\mathring\phi\restriction_{\d\Omega_{b}} = \phi\restriction_{\d\Omega_{b}}$,
$\CL\mathring\phi =P(\mathring\phi)$ on $\Omega_{b}$, and $|\mathring\phi(0)|< \eps-2\delta$.
Our goal is to estimate $\phi -\mathring\phi$.

We lift $\mathring\phi$ canonically to a modelled distribution $\mathring\bphi \colon\Omega_{b}\to \bar\CT_{\leq 2}\otimes E$ taking values in the polynomial regularity structure with $\scal{\mathring\bphi_x,X^k} = (\nabla^k\mathring\phi)_x$.
Denote
\begin{equ}
\|\mathring\bphi\|_{\infty;<2} = \max_{|k|_\s< 2} \sup_{x\in\Omega_{b}}\dist{x}^{|k|_\s} |\nabla^k\mathring\phi_x|
\;.
\end{equ}
Then, since $\|\mathring\phi\|_{\infty;\d\Omega_{b}}\leq\eps$ and by stability of the equation $\CL\mathring\phi = P(\mathring\phi)$ at the zero solution,
\begin{equ}[eq:Dk_phi]
\|\mathring\bphi\|_{\infty;<2}\lesssim \eps
\;.
\end{equ}
As in \eqref{eq:phi_decomp} define $Y,\mathring Y\in\CC(\cl\Omega_b,E)$ by
\begin{equ}[eq:Ys_def]
\phi = K^\lambda*\CR \{P(\bphi)\bone + \CR \tilde F\} + Y
\;,
\quad
\mathring\phi = K^\lambda* \CR \{P(\mathring\bphi)\bone\} + \mathring Y
\end{equ}
where we write $\tilde F = \sum_{|\tau|_\s < 0} F^\tau \tau$
and where $\CR$ is the same reconstruction operator as in \eqref{eq:phi_decomp}, i.e. the operator from Lemma \ref{lem:loc_reconst} for the domain $\Omega_{b}$ and for $\eta\in(-2,-1)$ from Section \ref{sec:remainder_SPDE}.

\textbf{NB.}
We apply Lemma \ref{lem:loc_reconst} with $\eta \in (-2,-1)$,
therefore $\CR \{P(\mathring\bphi)\bone\}$ is in $\CC^{\eta}(\R^d)$ but not necessarily in $\CC^{\omega}(\R^d)$ for any $\omega\geq -1$.

By the same argument as at the start of the proof of Lemma \ref{lem:R_tau_F_bound},
on $\Omega_{b}$
\begin{equ}
\CL Y = \CL \mathring Y = 0
\;.
\end{equ}
We furthermore have the equality between modelled distributions on $\Omega_{b}$
\begin{equ}
\Phi = \CK^\lambda (P(\bphi)\bone + \tilde F) + Y
\;,
\qquad
\mathring\bphi = \CK^\lambda (P(\mathring\bphi)\bone) + \mathring Y
\;,
\end{equ}
where we lift $Y,\mathring Y$ canonically to the polynomial regularity structure
and $\CK^\lambda$ is the integration map as in Section \ref{sec:integration} that uses $\CR$ as input.

In a similar way as in \eqref{eq:f1_bound}, using \eqref{eq:nabla_phi_bound} and \eqref{eq:Dk_phi},
we have
\begin{equs}
|P(\bphi_x) - P(\mathring\bphi_x)|
&\lesssim
\dist{x}^{-1}
\|\bphi - \mathring\bphi\|_{\infty;<2}
(\|\bphi\|_{\infty;<2}+\|\mathring\bphi\|_{\infty;<2})
\\
&\lesssim
\dist{x}^{-1}
\|\bphi - \mathring\bphi\|_{\infty;<2}
(\eps+b^{\eta+2}(V+W))
\\
&\lesssim 
\dist{x}^{\eta}\|\bphi - \mathring\bphi\|_{\infty;<2}
(\eps+b^{\eta+2}(V+W))
\;,
\end{equs}
where we used $\eta<-1$ in the final bound.
Since $\eta>-2$, we obtain by Lemma \ref{lem:integration_K_xi}
\begin{equ}[eq:K_P_diff]
\|K^\lambda* \CR (P(\bphi)\bone - P(\mathring\bphi)\bone)\|_{\infty;<2}
\lesssim
b^{\eta+2}\|\bphi - \mathring\bphi\|_{\infty;<2}
(\eps+b^{\eta+2}(V+W))
\;.
\end{equ}

Turning to $\tilde F$,
by the integration bound \eqref{eq:K_F_sup_bound},
\begin{equ}[eq:K_tilde_F]
|(\nabla^l \CK^\lambda \tilde F)_x|
\lesssim \dist{x}^{-|l|_\s} b^{\eta+2}(V^*+W^*)
\;.
\end{equ}

Turning to $Y$ and $\mathring Y$, since $\CL(Y-\mathring Y)
=0$ on $\Omega_{b}$,
it follows from weighted Schauder estimates \cite[Thm.~4.9,~p.59]{Lieberman_96_Parabolic} and the maximum principle that, for any $|k|_\s< 2$,
\begin{equ}[eq:DY_diff]
|\nabla^k (Y-\mathring Y)_x|
\lesssim \dist{x}^{-|k|_\s}
\|Y-\mathring Y\|_{\infty;\d\Omega_{b}}
\;.
\end{equ}
Furthermore, since $\phi=\mathring\phi$ on $\d\Omega_{b}$ by our choice of $\mathring\phi$,
\begin{equs}[eq:Y_diff_bound]
\|Y-\mathring Y\|_{\infty;\d\Omega_{b}}
&=
\big\|
K^\lambda* \CR \big\{P(\bphi)\bone  - P(\mathring\bphi)\bone + \tilde F\big\}
\big\|_{\infty;\d\Omega_{b}}
\\
&\lesssim
b^{\eta+2}
\big\{
\|\bphi - \mathring\bphi\|_{\infty;<2}(\eps+b^{\eta+2}(V+W))
+ (V^*+W^*)
\big\}
\end{equs}
where the final bound is due \eqref{eq:K_P_diff}-\eqref{eq:K_tilde_F}.
Treating $Y,\mathring Y$ as modelled distributions, we obtain from \eqref{eq:DY_diff}-\eqref{eq:Y_diff_bound}
\begin{equ}[eq:Y_diff]
\| Y-\mathring Y\|_{\infty;<2}
\lesssim
b^{\eta+2}
\big\{
\|\bphi - \mathring\bphi\|_{\infty;<2}(\eps+b^{\eta+2}(V+W))
+ (V^*+W^*)
\big\}
\;.
\end{equ}
Finally, combining \eqref{eq:K_P_diff}, \eqref{eq:K_tilde_F}, and \eqref{eq:Y_diff}, and using the definitions \eqref{eq:Ys_def}, we obtain
\begin{equ}
\|\bphi - \mathring\bphi\|_{\infty;<2}
\lesssim
b^{\eta+2}
\big\{
\|\bphi - \mathring\bphi\|_{\infty;<2}
(\eps+b^{\eta+2}(V+W))
+ (V^*+W^*)
\big\}
\;.
\end{equ}
By \eqref{eq:VW_bound} in Lemma \ref{lem:VW_bound}, we have $V+W \leq 1/2$. Taking $b$  sufficiently small, we conclude that
\begin{equ}
\|\bphi - \mathring\bphi\|_{\infty;<2}
\lesssim
b^{\eta+2}(V^*+W^*)\;,
\end{equ}
which, by \eqref{eq:VW_star_bound} in Lemma \ref{lem:VW_bound}, can be made arbitrarily small by taking $r\to 0$.
In particular, if $r>0$ is sufficiently small, then $|\phi(0)-\mathring\phi(0)|<\delta$ and thus $|\phi(0)|<\eps-\delta$.
\end{proof}

\begin{remark}\label{rem:eta_neg}
The reason we require $\eta(\tau,\bk),\bar\eta(\tau,\bk)$ to be non-negative in Assumption \ref{as:f_poly}
is that, if these exponents were negative, we could not obtain \eqref{eq:Upsilon_bphi_bound}.
It would be natural in this case to follow the argument in Section \ref{sec:Young_weighted} and consider the largest $b$ such that $|\phi|\in [\frac\eps2,\eps]$ on $\Omega_b$,
so the estimates in Section \ref{sec:remainder_SPDE} hold on $\Omega_b$.
However, because $b$ depends on $\phi$, and it is possible that $|\mathring\phi|\approx \eps$ on a large subset of $\d\Omega_b$,
the domain $\Omega_b$ may be too small to ensure $|\mathring\phi(0)|<\eps-\delta$ for $\delta>0$ \emph{uniform} in $\phi$,
so the proof of the final Lemma \ref{lem:local_coer_SPDE} would not go through.
For ODEs (Sections \ref{sec:Young_weighted} and \ref{sec:RPs}) this is not an issue because, by maximality of $b$ and the fact that the boundary of the domain is a single point, we have $|\mathring\phi|=\eps/2$ on the boundary, thus $\mathring\phi$ does not need `space' to come down.
\end{remark}

\appendix

\section{Extension of distributions}
\label{sec:extension}

Fix a scaling $\s = (\s_1,\ldots,\s_d)\in\N^d$ and follow the notation of Section \ref{sec:scaling}.
We use the shorthand $\mfB_b = \mfB(0,b)$.
Fix $b\in (0,1]$ and, for $x\in \mfB_b$, denote
\begin{equ}
\disc{x} = \inf_{y \in \bd(\mfB_b)} |x-y|_\s
\;,
\end{equ}
where $\bd(A) \subset \R^d$ is the topological boundary of a set $A\subset \R^d$.
For $\eta\leq 0$, let $\widehat\CC^\eta(\mfB_b)$ denote the Banach space of distributions $\xi\in\CD'(\mfB_b)$ for which
\begin{equ}
\disc{\xi}_{\widehat \CC^\eta(\mfB_b)} =
\sup_{\psi\in\CB^r} \sup_{x\in\mfB_b}
\sup_{0<\delta<\frac12\disc{x}}
\delta^{-\eta}|\scal{\xi,\psi^\delta_x}|
< \infty
\;,
\end{equ}
where $r= \floor{1-\eta}\geq 1$ and we recall $\CB^r$ from \eqref{eq:Br_def}.

We start with a simple lemma, the proof of which is clear.

\begin{lemma}\label{lem:equiv_norm}
Consider $C>0$.
An equivalent norm on $\CC^\eta(\R^d)$ is
\begin{equ}
\sup_{z \in \R^d} \sup_{\delta \in (0,1)} \sup_{\psi}
\delta^{-\eta}|\scal{\xi,\psi}|
\end{equ}
where the final $\sup$ is taken over all
\begin{equ}[eq:psi_condition]
\psi \in \CC^\infty_c(\mfB(z,\delta)) \quad
\text{such that, for all $|k|_{\s}\leq r$,}
\quad
\|D^k \psi\|_\infty\leq C \delta^{-|\s|-|k|_\s}
\;.
\end{equ}
The same statement holds for $\widehat\CC^\eta(\mfB_b)$ but with the constraints that $z\in\mfB_b$ and $\mfB(z,2\delta) \subset \mfB_b$.
\end{lemma}

The following extension lemma is the main result of this appendix.
Similar results can be found in the literature (see, e.g. \cite{Rychkov98_extension,Rychkov99_extension} for extensions whose support goes outside $\cl \mfB_b$),
but we prefer to give a self-contained proof which, we believe, is comparatively simple.

\begin{lemma}[Extension]\label{lem:extension}
Let $\eta<0$ with $\eta \notin \Z$.
There exists a bounded linear operator $\CE \colon \widehat\CC^\eta(\mfB_b)\to\CC^\eta(\R^d)$ such that
$\supp \CE (\xi) \subset \cl \mfB_b$ and
$\scal{\CE(\xi),\psi}=\scal{\xi,\psi}$ for all $\psi\in\CC^\infty_c(\mfB_b)$,
and such that the operator norm of $\CE$ does not depend on $b$.
\end{lemma}

\begin{proof}
The key to the proof is a partition of unity $\{\phi_{x}\}_{x\in\Lambda}$ such that
\begin{enumerate}[label=(\alph*)]
\item \label{pt:sum}$\Lambda$ is a countable subset of $\mfB_b$ and $\sum_{x\in\Lambda} \phi_x = 1$ on $\mfB_b$,
\item\label{pt:support} $\phi_x \colon \mfB_b\to [0,1]$ with $\supp \phi_x$ contained in the open ball $\mfB(x,\frac12\disc{x})$,
\item\label{pt:partition_bound} $\|D^k\phi_x\|_\infty \lesssim \disc{x}^{-|k|_\s}$ uniformly in $|k|_\s \leq r$, $x\in\Lambda$, and $b\in (0,1]$, and 
\item\label{pt:cluster} for every $x\in\Lambda$, the number of $y\in\Lambda$ such that
$\mfB(x,\frac12\disc{x})\cap \mfB(y,\frac12\disc{y})\neq\emptyset$ is bounded 
from above by a constant $N\geq 1$ that is uniform in $x,y$.
\end{enumerate}
The existence of such a partition of unity follows in the same manner as in~\cite{Whitney_1934_extension} or \cite[Sec.~VI.1]{Stein70_singular}, see also \cite[Sec.~5.3.3]{Martin_2018_thesis} and \cite[Proof of Lem~A.2]{CCHS_2D}.


Note that $\bd\mfB_b$ is a box in which each side is parallel to a hyperplane $\{z\,:\,z_i=0\}$ for some $i\in[d]$.
For every $x\in\Lambda$, there exists a point $p_x\in\bd \mfB_b$ such that $|x-p_x|_\s \leq 2\disc{x}$ and such that $p_x$ is at least distance $\frac{1}{4}\disc{x}$ from the intersection of any two sides of $\mfB_b$.
Let $\d B_x$ be the side of $\cl\mfB_b$ containing $p_x$ and denote $\s_x = \s_i$ where $i\in [d]$ is such that $\d B_x$ is parallel to the hyperplane $\{z\,:\,z_i=0\}$.
Denote $B_x = \{y\in\bd\mfB_b\,:\, |y-p_x|_\s < \frac{1}{8}\disc{x}\}$.
Remark that $B_x \subset \d B_x$ by our choice of $p_x$.
We let $\fint_{B_x} = \frac{1}{|B_x|}\int_{B_x}$ where the integral $\int_{B_x}$ and volume $|B_x|$ are taken with respect to the $(d-1)$-dimensional Lebesgue measure on $B_x$ (so $|B_x|\asymp \disc{x}^{|\s|-\s_x}$).

Fix $\delta \in (0,1)$, $z\in\R^d$, and $\psi\in\CC^\infty_c(\mfB(z,\delta))$ as in \eqref{eq:psi_condition} with $C=1$.
For every $x\in\Lambda$ we define $\psi_x \in \CC^\infty_c(\mfB(x,\frac12\disc{x}))$ by
\begin{equ}
\psi_x(y) = \phi_x(y)
\Big\{
\psi(y) - \sum_{0\leq |k|_\s < r-\s_x}  \fint_{B_x} \frac{D^k\psi(p)}{k!}(y-p)^k
\mrd p
\Big\}
\eqdef
\phi_x(y) \T^{r-\s_x}_{x}(y)
\;.
\end{equ}
While $\supp \psi \subset \mfB(z,\delta)$, remark that $\supp \T^{r-\s_x}_{x}$ might not be contained in $\mfB(z,\delta)$.

Fix furthermore $\xi\in\widehat\CC^\eta(\mfB_b)$.
We claim that the sum
\begin{equ}[eq:extension_def]
\scal{\CE(\xi),\psi} \eqdef \sum_{x\in\Lambda}\scal{\xi,\psi_x}
\end{equ}
converges absolutely, agrees with $\scal{\xi,\psi}$ for $\psi\in\CC^\infty_c(\mfB_b)$, and satisfies the bound
\begin{equ}[eq:extension_bound]
|\scal{\CE(\xi),\psi}|
\lesssim
\delta^\eta
\|\xi\|_{\widehat\CC^\eta(\mfB_b)} \;.
\end{equ}
Clearly $\scal{\CE(\xi),\psi}=0$ whenever $\psi\in\CC^\infty_c(\R^d\setminus\cl\mfB_b)$ and thus $\supp \CE(\xi) \subset \cl\mfB_b$.
Therefore the proof follows from Lemma \ref{lem:equiv_norm} once the claim is shown.

To this end,
assume first $\psi \in \CC^\infty_c(\mfB_b)$.
Then the sum in \eqref{eq:extension_def} is finite by property \ref{pt:cluster}. 
Moreover, $D^k\psi(p)=0$ for all $p\in\bd \mfB_b$ and therefore, by \ref{pt:sum}, $\sum_{x\in\Lambda}\psi_x = \psi$ and $\scal{\CE(\xi),\psi} = \scal{\xi,\psi}$.
Furthermore, if $\mfB(z,2\delta)\subset\mfB_b$,
then the claimed bound \eqref{eq:extension_bound} follows from the equivalent norm on $\widehat\CC^\eta(\mfB_b)$ of Lemma \ref{lem:equiv_norm}.

Suppose now that $\mfB(z,2\delta)$ is not contained in $\mfB_b$ (and we do not necessarily assume that $\psi$ is supported in $\mfB_b$).
Consider first $\disc{x}\leq 8\delta$.
Let $l\in\N^d$ with $|l|_\s\leq r$.
By \eqref{eq:psi_condition}, if $|l|_\s < r-\s_x$ and $r-\s_x > 0$, then
\begin{equ}
\|D^l \T^{r-\s_x}_{x}\|_{\infty;\mfB(x,\frac12\disc{x})}
\lesssim
\delta^{-|\s|-r+\s_x}
\disc{x}^{r-\s_x-|l|_\s}
\;,
\end{equ}
while if $|l|_\s\geq r-\s_x$ (which is always the case if $r-\s_x \leq 0$), then $D^l \T^{r-\s_x}_{x}=D^l\psi$ and thus
\begin{equ}
\|D^l \T^{r-\s_x}_{x}\|_{\infty}
\lesssim
\delta^{-|\s|-|l|_\s}
\;.
\end{equ}
Then, for any multi-index $k\in\N^d$ with $|k|_\s\leq r$,
by \ref{pt:support}-\ref{pt:partition_bound},
\begin{equs}
\|D^k\psi_x\|_\infty
&\lesssim
\sum_{l\leq k,\,|l|_\s<r-\s_x}
\disc{x}^{-|k-l|_\s}
\delta^{-|\s|-r+\s_x}
\disc{x}^{r-\s_x-|l|_\s}
+
\sum_{l\leq k,\,|l|_\s\geq r-\s_x}
\disc{x}^{-|k-l|_\s}
\delta^{-|\s|-|l|_\s}
\\
&\lesssim
\begin{cases}
\disc{x}^{r-\s_x-|k|_\s}
\delta^{-|\s|-r+\s_x}
&\text{ if } r-\s_x > 0
\\
\disc{x}^{-|k|_\s}
\delta^{-|\s|}
&\text{ if } r-\s_x \leq 0
\;.
\end{cases}
\end{equs}
It follows that
\begin{equ}
\max_{|k|_\s\leq r}
\disc{x}^{|\s|+|k|_\s}\|D^k \psi_x\|_\infty
\lesssim
\begin{cases}
\disc{x}^{r-\s_x+|\s|}
\delta^{-|\s|-r+\s_x}
&\text{ if } r-\s_x > 0
\\
\disc{x}^{|\s|}
\delta^{-|\s|}
&\text{ if } r-\s_x \leq 0
\end{cases}
\end{equ}
and therefore,
since $\psi_x$ is supported in $\mfB(x,\frac12\disc{x})$,
it follows from Lemma \ref{lem:equiv_norm} that
\begin{equ}
|\scal{\xi, \psi_x}|
\lesssim
\disc{x}^{\eta + |\s| + (r-\s_x)\vee 0}
\delta^{-|\s| - (r-\s_x)\vee 0}
\|\xi\|_{\widehat\CC^\eta(\mfB_b)} 
\;.
\end{equ}
By \ref{pt:cluster},
uniformly in $n\geq0$,
there are $O(2^{n(|\s|-\s_i)}\delta^{|\s|-\s_i})$ points $x\in\Lambda$
in any ball of radius $\delta$ for the scaling $\s$
such that $\s_x = \s_i$ and $\disc{x} \in [2^{-n},2^{-n+1}]$.
Moreover, if $\disc{x}\leq 8\delta$ and $x$ is outside the ball $\mfB(z, C\delta)$ for $C>0$ sufficiently large, then $\psi$ vanishes on a neighbourhood of $B_x$ and therefore $\psi_x=0$.
It follows that
\begin{equ}[eq:extension_small]
\sum_{\disc{x}\leq 8 \delta}
|\scal{\xi, \psi_x}|
\lesssim
\max_{i\in[d]}
\sum_{2^{-n} \leq C\delta}
2^{n(|\s|-\s_i - \eta - |\s| - (r-\s_i)\vee 0)}\delta^{|\s|-\s_i -|\s|-(r-\s_i)\vee 0}
\|\xi\|_{\widehat\CC^\eta(\mfB_b)} 
\lesssim
\delta^{\eta}
\|\xi\|_{\widehat\CC^\eta(\mfB_b)} \;,
\end{equ}
where we use that the exponent of $2^n$ is strictly negative.

Consider now $\disc{x} > 8\delta$.
Note that $\psi$ vanishes on $\mfB(x,\frac12\disc{x})$ and thus, for $|l|_\s \geq r-\s_x$, we have $D^l \T^{r-\s_x}_{x} \equiv 0$ on $\mfB(x,\frac12\disc{x})$.
On the other hand, for $|l|_\s < r-\s_x$, by \eqref{eq:psi_condition} and the fact that $|B_x|^{-1}\lesssim \disc{x}^{-|\s|+\s_x}$ and the support of $\psi$ in $B_x$ has volume at most $\lesssim \delta^{|\s|-\s_x}$,
\begin{equ}
\|D^l \T^{r-\s_x}_{x}\|_{\infty;\mfB(x,\frac12\disc{x})}
\lesssim
\disc{x}^{-|\s|+\s_x}
\delta^{|\s|-\s_x}
\delta^{-|\s|-r+\s_x+1}
\disc{x}^{r-\s_x-|l|_\s-1}
=
\delta^{-r+1}
\disc{x}^{-|\s|+r-1-|l|_\s}
\;.
\end{equ}
Therefore, if $r-\s_x\leq 0$, then $\psi_x\equiv 0$,
and we are done by the previous bound \eqref{eq:extension_small}.
Suppose now that $r-\s_x>0$. Then, using in addition \ref{pt:support}-\ref{pt:partition_bound},
uniformly in $|k|_\s\leq r$,
\begin{equ}
\|D^k\psi_x\|_\infty
\lesssim
\sum_{l\leq k,\, |l|_\s<r-\s_x}
\disc{x}^{-|k-l|_\s}
\delta^{-r+1}
\disc{x}^{-|\s|+r-1-|l|_\s}
\lesssim
\disc{x}^{-|\s|+r-1-|k|_\s}
\delta^{-r+1}
\;.
\end{equ}
It follows that
\begin{equ}
\max_{|k|_\s\leq r}
\disc{x}^{|\s|+|k|_\s}\|D^k \psi_x\|_\infty
\lesssim
\disc{x}^{r-1}
\delta^{-r+1}
\;.
\end{equ}
Since $\psi_x$ is supported in $\mfB(x,\frac12\disc{x})$, we obtain from Lemma \ref{lem:equiv_norm} that
\begin{equ}
|\scal{\xi, \psi_x}|
\lesssim
\disc{x}^{\eta+r-1}
\delta^{-r+1}
\|\xi\|_{\widehat\CC^\eta(\mfB_b)}
\;.
\end{equ}
By \ref{pt:cluster}, uniformly in $n\geq 0$, there are $O(1)$ points $x\in\Lambda$ such that $\disc{x} > 8 \delta$, $\disc{x}\in [2^{-n},2^{-n+1}]$, and $B_x$ overlaps with $\mfB(z,\delta)$.
It follows that
\begin{equ}
\sum_{\disc{x}\geq 8\delta}
|\scal{\xi, \psi_x}|
\lesssim
\max_{i\in[d]}
\sum_{2^{-n} \geq \delta}
2^{n(- \eta - r + 1)}\delta^{-r+1}
\|\xi\|_{\widehat\CC^\eta(\mfB_b)} 
\lesssim
\delta^{\eta}
\|\xi\|_{\widehat\CC^\eta(\mfB_b)} \;,
\end{equ}
where we used $-\eta-r+1 > 0$ because we assumed $\eta\notin\Z$.
This proves \eqref{eq:extension_bound}
and the claim.
\end{proof}

\begin{remark}
Note that $\supp \xi\subset\supp\CE(\xi)$ and that the containment is strict in general.
It is not clear to us whether an extension exists that does not enlarge the support. 
\end{remark}

\subsubsection*{Acknowledgments}

I.~C. acknowledges support from the European Research Council via the grant SQGT 101116964.
M.~G. was partly supported by UK Research and Innovation via the grant ``StochFields'' EP/Z534328/1.

For the purpose of open access, the author has applied a CC BY public copyright licence to any author accepted manuscript arising from this submission.

\bibliographystyle{Martin}
\bibliography{refs}{}

\end{document}